\documentclass{amsart}
\usepackage{fullpage}
\usepackage[table]{xcolor}
\usepackage[T1]{fontenc}
\usepackage{amsmath}
\usepackage{amsfonts}
\usepackage{amssymb}
\usepackage{amsthm}
\usepackage{mathtools}
\usepackage{array}
\usepackage{stmaryrd}
\usepackage{tikz-cd, tikz}
\usepackage{graphicx}
\usepackage{enumitem}
\usepackage{multicol}
\usepackage{multirow}
\usepackage{colortbl}
\usepackage{afterpage}
\usepackage[bottom]{footmisc}

\usepackage{pgfplots}
\pgfplotsset{width=10cm,compat=1.9}

\usepackage[algoruled,linesnumbered]{algorithm2e}
\usepackage{comment}
\usepackage{thmtools}
\usepackage{thm-restate}
\allowdisplaybreaks
\usepackage{hyperref}
\hypersetup{
	colorlinks,
	hyperfootnotes=true,
	linkcolor={red!50!black},
	citecolor={blue!50!black},
	urlcolor={blue!80!black},
}
\usepackage[hyphenbreaks]{breakurl}
\usepackage[capitalise]{cleveref}
\usepackage{tabulary}
\usepackage{adjustbox}
\usepackage[disable]{todonotes}
\usepackage{arydshln}
\usepackage{mathdots}

\newtheorem{theorem}{Theorem}[section]
\newtheorem*{theorem*}{Theorem}
\newtheorem{lemma}[theorem]{Lemma}
\newtheorem*{lemma*}{Lemma}
\newtheorem{proposition}[theorem]{Proposition}
\newtheorem*{proposition*}{Proposition}
\newtheorem{corollary}[theorem]{Corollary}

\theoremstyle{definition}
\newtheorem{definition}[theorem]{Definition}

\theoremstyle{remark}
\newtheorem{remark}[theorem]{Remark}
\newtheorem{example}[theorem]{Example}

\newtheorem{convention}[theorem]{Convention}

\newcommand{\Z}{\mathbf{Z}}

\newcommand{\Q}{\mathbf{Q}}

\newcommand{\F}{\mathbf{F}}

\newcommand{\Proj}{\mathbf{P}}

\newcommand{\Fbar}{\overline{\mathbf{F}}}
\newcommand{\Vbar}{\overline{V}}

\DeclareMathOperator{\Frob}{Frob}
\DeclareMathOperator{\frob}{frob}
\let\oldeta\eta
\DeclareMathOperator{\neweta}{\oldeta}
\let\eta\neweta
\let\oldzeta\zeta
\DeclareMathOperator{\newzeta}{\oldzeta}
\let\zeta\newzeta

\newcommand{\Div}{\textnormal{Div}}
\newcommand{\End}{\textnormal{End}}

\newcommand{\Gal}{\textnormal{Gal}}
\newcommand{\Hom}{\textnormal{Hom}}
\newcommand{\Cl}{\textnormal{Cl}}

\newcommand{\res}{\textnormal{res}}

\newcommand{\ord}{\textnormal{ord}}

\newcommand{\Sym}{\textnormal{Sym}}
\newcommand{\cyclo}{\chi}

\newcommand{\divisor}{\textnormal{div}}

\newcommand{\GL}{\textnormal{GL}}

\newcommand{\bigrep}[2]{\Sym^{#1}\rho \otimes \cyclo^{#2}}

\newcommand{\qg}{{q^\gamma}}

\newcommand{\Czeta}{\zeta}

\newcommand{\Kprime}{\F_\qg(C)}
\newcommand{\sep}{\textnormal{sep}}

\usepackage{url}

 \usepackage[backend=bibtex,doi=false,url=false]{biblatex}
\title{On $\ell$-torsion in degree $\ell$ superelliptic Jacobians over $\F_q$}

\author[Wanlin Li]{Wanlin Li}
\address{Department of Mathematics, Washington University in St. Louis, 1 Brookings Dr, St.Louis, MO 63105, USA}
\email{wanlin@wustl.edu}
\author{Jonathan Love} 
\address{Mathematics Institute, Leiden University, Leiden, the Netherlands
}
\email{j.r.love@math.leidenuniv.nl}
\author{Eric Stubley}
\email{ericdavidstubey@gmail.com}
\date{December 29, 2024}

\begin{document}

\begin{abstract}
    We study the $\ell$-torsion subgroup in Jacobians of curves of the form $y^{\ell} = f(x)$ for irreducible $f(x)$ over a finite field $\F_{q}$ of characteristic $p \neq \ell$. This is a function field analogue of the study of $\ell$-torsion subgroups of ideal class groups of number fields $\Q(\sqrt[\ell]{N})$. We establish an upper bound, lower bound, and parity constraint on the rank of the $\ell$-torsion which depend only on the parameters $\ell$, $q$, and $\deg f$.
    Using tools from class field theory, we show that additional criteria depending on congruence conditions involving the polynomial $f(x)$ can be used to refine the upper and lower bounds. For certain values of the parameters $\ell,q,\deg f$, we determine the $\ell$-torsion of the Jacobian for all curves with the given parameters.
\end{abstract}

\maketitle

\section{Introduction}\label{sec:intro}

The ideal class group of a number field is one of the central topics of interest in algebraic number theory. If we consider the collection of degree $n$ extensions $K/\Q$ with some fixed Galois group, then for all but finitely many primes $\ell$, the $\ell$-torsion of the class group of $K$ is conjectured to be distributed according to the Cohen-Lenstra heuristics~\cite{cohenlenstra}. If $\ell\mid n$, however, the $\ell$-torsion is expected to have qualitatively different behavior. For instance, in the case $\ell=2$ and $K$ is an imaginary quadratic field, Gauss' genus theory completely describes the $2$-torsion of the class group of $K$ in terms of the number of ramified primes. In general, if $\ell\geq 3$ divides $n$, then the $\ell$-torsion structure can be considerably more mysterious. 

In \cite{schaeferstubley}, the authors used Galois cohomology to study the $\ell$-torsion of the ideal class groups of the degree $\ell$ number fields $\Q(\sqrt[\ell]{N})$ for prime $N$; see \cref{sec:priorwork} for more on the history of this problem. In this paper, we study an analogous problem over global function fields, namely the divisor class groups of fields of the form $\F_q(\sqrt[\ell]{f},x)$ for $f(x)\in\F_q[x]$ irreducible. In this setting, we are able to utilize both Galois cohomology inspired by \cite{schaeferstubley} and tools from arithmetic geometry to obtain more refined constraints on the $\ell$-torsion, and we encounter interesting behavior which does not occur in the number field setting. 

Computing the $\ell$-torsion structure of the divisor class group of a function field is typically a computationally intensive problem that requires first finding the full class group. We produce constraints on the $\ell$-torsion using data that are much easier to compute, and in some cases, these constraints uniquely determine the $\ell$-torsion. The full results are discussed in \cref{sec:rlC_constraints}, but we give one example application here.
\begin{theorem}\label{thm:intro example}
    Let $\ell\geq 3$ be prime, $q$ a prime power with $q^2\equiv 1\bmod\ell$, and $f(x)\in\F_q[x]$ irreducible with $\deg f$ coprime to $\ell$. The $\ell$-torsion of the divisor class group of $\F_q(\sqrt[\ell]{f},x)$ is isomorphic to $(\Z/\ell\Z)^{(\ell-1)/2}$ if $q\equiv -1\bmod\ell$ and $\deg f$ is even, and is trivial otherwise.
\end{theorem}
\noindent
If $\ell=3$, and $q$ and $\deg f$ are coprime to $3$, \cref{thm:intro example} shows that the $3$-torsion can be determined using no information about $f$ other than its degree. If $\ell=5$, and $q$ and $\deg f$ are coprime to $5$, then we can completely determine the $5$-torsion structure using easily computable conditions depending on $f$ (\cref{cor:small primes}).

\begin{table}[b]
    \begin{tabular}{r|cccccc}
        $7$-torsion: & \hspace{8pt}$0$\hspace{8pt} & $(\Z/7\Z)^{\phantom{1}}$ & $(\Z/7\Z)^2$ & $(\Z/7\Z)^3$ & $(\Z/7\Z)^4$ & $(\Z/7\Z)^5$ \\
        \hline
        count: & $0$ & $5552$ & $0$ & $1840$ & $0$ & $12$
    \end{tabular}
    \caption{The number of isomorphism classes of fields $\F_3(\sqrt[7]{f},x)$ attaining each possible $7$-torsion structure in its divisor class group, with $f(x)\in\F_3[x]$ irreducible of degree $12$.} \label{tab:7data}
\end{table}

For $\ell\geq 7$, the $\ell$-torsion structure is typically not fully determined by the easily computable conditions mentioned above, but we prove a parity constraint which gives us a better understanding of the $\ell$-torsion. To illustrate this phenomenon, up to isomorphism there are $7404$ function fields of the form $\F_3(\sqrt[7]{f},x)$ with $f(x)\in\F_3[x]$ irreducible of degree $12$. For each of these fields, the authors used Magma to compute the divisor class group and recorded the $7$-torsion structure of each; see \cref{tab:7data}. While the relative distribution of curves across the possible torsion structures is a subject of future exploration, our results explain the zeroes in the table. More generally, we will see that the largest power of $(\Z/\ell\Z)$ occurring as a subgroup of the divisor class group of $\F_q(\sqrt[\ell]{f},x)$ must be odd whenever $\ell\geq 3$, $q$ is a primitive root mod $\ell$, and $\deg f$ is even and coprime to $\ell$ (\cref{intro_upper_bound}). It seems as though this phenomenon is unique to the function field setting and does not arise for number fields.

\subsection{Main results}\label{sec:rlC_constraints} 

For all the results that follow, we assume $\ell\geq 3$ is prime, $q$ is a prime power coprime to $\ell$, and $f(x)\in\F_q[x]$ is an irreducible polynomial with $d:= \deg f$ coprime to $\ell$. Let $C$ be the smooth projective curve with affine equation given by $y^\ell=f(x)$; such a curve is an example of a ``superelliptic curve.'' Let $J$ be the Jacobian of $C$, so the degree $0$ subgroup of the divisor class group of $C$ is isomorphic to $J(\F_{q})$. The $\ell$-torsion of $J(\F_q)$ can be equipped with the structure of a vector space over $\F_\ell$, and we define the $\ell$-rank of $C$ to be the dimension of this $\F_\ell$ vector space,
\[r_\ell(C):=\dim_{\F_{\ell}}J[\ell](\F_{q}).\]
The function field of $C$ is isomorphic to $\F_q(\sqrt[\ell]{f},x)$, and up to isomorphism $C$ is the only smooth projective curve with this function field. We define the divisor class group of $\F_q(\sqrt[\ell]{f},x)$ to be the divisor class group of $C$. Then an equivalent definition for $r_\ell(C)$ is that it is the largest power of $\Z/\ell\Z$ that occurs as a subgroup of the divisor class group of the function field $\F_q(\sqrt[\ell]{f},x)$.

\begin{remark}
    The above definitions are valid also for $\ell=2$, but in this case we always have $r_2(C)=0$, because a hyperelliptic curve $y^2=f(x)$ has no $\F_q$-rational $2$-torsion in its Jacobian when $f$ is irreducible.
\end{remark}

Let $\gamma=\ord_{\ell}(q)$ be the multiplicative order of $q \bmod \ell$ in $(\Z/\ell\Z)^\times$, that is, the smallest positive integer such that $q^\gamma\equiv 1\bmod\ell$. This is an important invariant for this problem because $\F_\qg$ is the smallest extension of $\F_q$ containing $\ell$-th roots of unity, and hence the Galois closure of $\F_q(\sqrt[\ell]{f},x)/\F_q(x)$ is a degree $\gamma$ extension field, namely $\F_\qg(\sqrt[\ell]{f},x)$.

\begin{theorem}\label{intro_upper_bound} 
Set
\[B:=(\gcd(d, \gamma) - 1) \frac{\ell-1}{\gamma}.\]
Then the $\ell$-rank $r_\ell(C)$ satisfies $\min\{B,1\}\leq r_\ell(C) \leq B$ and $r_\ell(C)\equiv B\bmod 2$.
\end{theorem}
The parity constraint $r_\ell(C)\equiv B\bmod 2$ is proved using the Weil pairing on $J[\ell]$. This phenomenon does not appear to occur in the analogous situation in number fields, namely ideal class groups of cyclic extensions $\Q(\sqrt[p]{N})$ for $N$ prime discussed in \cite{schaeferstubley}; see \cref{sec:priorwork} for a discussion of the number field case.

\begin{example}
    Consider the case $\ell=3$. If $q\equiv 1\bmod 3$ or if $d$ is odd, then $\gcd(d, \gamma)=1$, so \cref{intro_upper_bound} implies that $r_\ell(C)=0$. Otherwise, if $\deg f$ is even and $q \equiv 2 \bmod 3$, we have $r_3(C)=1$, and we recover \cref{thm:intro example} for $\ell=3$. Compare \cite[Theorem 6.1.1]{schaeferstubley} which addresses extensions $\Q(\sqrt[3]{N})/\Q$ for prime $N \equiv 1 \bmod 3$.
\end{example}

In one special case, we can prove a lower bound that equals the upper bound in \cref{intro_upper_bound}, allowing us to construct families of curves with large $\ell$-torsion subgroups in their divisor class groups.

\begin{proposition}\label{prop:2lifting}
    If $\gcd(d,\gamma)=2$, then $r_\ell(C)=\frac{\ell-1}{\gamma}$.
\end{proposition}

\Cref{thm:intro example} follows immediately from \cref{intro_upper_bound} and \cref{prop:2lifting}.
Both \cref{intro_upper_bound} and \cref{prop:2lifting} can be proven with linear algebra, using linear maps on $J[\ell]$ defined using endomorphisms of $J$. The parity constraint $r_\ell(C)\equiv B\bmod 2$ is proved using the Weil pairing on $J[\ell]$.  These topics are summarized in \cref{sec:lifting_behavior_functions} and discussed in depth in \cref{sec:geo_to_reps} and \cref{sec:lifting_relations_new}. The proofs of \cref{intro_upper_bound} and \cref{prop:2lifting} are then completed in \cref{sec:consequences}. 

If $\gamma\leq 2$ then $r_\ell(C)$ is completely determined by \cref{thm:intro example}, so for the remainder of this section we assume $\gamma\geq 3$. We can compute more refined bounds on $r_\ell(C)$ if we additionally assume $\gamma\mid d$. This constraint ensures that $f$ totally splits in the extension $\F_\qg(x)/\F_q(x)$; this is analogous to the constraint $N\equiv 1 \bmod p$ in \cite{schaeferstubley} which guarantees that $N$ totally splits in $\Q(\zeta_p)/\Q$.
Over $\F_{q^\gamma}[x]$, $f(x)$ splits into $\gamma$ irreducible factors, which we label $f_1(x),f_2(x),\ldots, f_{\gamma}(x)$ in such a way that the Frobenius automorphism on $\F_{q^\gamma}$ sends $f_i(x)$ to $f_{i+1}(x)$ for all $i$ (and $f_{\gamma}(x)$ to $f_1(x)$). Set
\begin{align}
    \nonumber h_n(x)&:=\prod_{i=1}^{\gamma} f_{i}(x)^{q^{(i-1)(\gamma-n)}-1},\qquad n\in \{2,\ldots,\gamma-1\},\\
    \label{eq:local_cup} 
    \mathcal{T}&:=\{1\}\cup \{n\in \{2,\ldots,\gamma-1\}: 
    h_n(x)\text{ is an $\ell^{\text{th}}$ power in }\F_{q^\gamma}[x]/(f_1(x))\}.
\end{align} 

The polynomials $h_n(x)$ are associated via Kummer theory to certain cyclic degree $\ell$ extensions of $\F_\qg(x)$; see \cref{subsec:congruence conditions} for more on how these polynomials arise.

\begin{theorem}\label{intro_refined_upper_bound}
Let $f(x)\in\F_q[x]$ be irreducible of degree $d$, with $d$ coprime to $\ell$ and $3\leq \gamma\mid d$.
Set
\[B':=|\mathcal{T}|\frac{\ell-1}{\gamma}.\]
Then the $\ell$-rank $r_\ell(C)$ satisfies $\min\{B',2\}\leq r_\ell(C)\leq B'$ and $r_\ell(C)\equiv B'\bmod 2$.  

If in addition $\gamma$ is even and $1+\frac{\gamma}{2}\in\mathcal{T}$, then $r_\ell(C)\geq 3$.
\end{theorem}
Since $|\mathcal{T}|\leq \gamma-1=\gcd(d,\gamma)-1$ we have $B'\leq B$, and from $\min\{B,1\}\leq r_\ell(C)\leq B'$ we can conclude $\min\{B',2\}\geq \min\{B,1\}$. Thus \cref{intro_refined_upper_bound} gives both upper and lower bounds that are at least as strong as those in \cref{intro_upper_bound}. 

In addition to the linear algebra on $J[\ell]$ discussed in \cref{sec:geo_to_reps} and \cref{sec:lifting_relations_new}, the proof of \cref{intro_refined_upper_bound} requires techniques from Kummer theory and Galois cohomology. These techniques are introduced in \cref{subsec:results from cohom} and discussed in depth in Sections \ref{sec:galoisreps}, \ref{sec:galois_cohom}, and \ref{sec:cup products}. The proof of \cref{intro_refined_upper_bound} is then completed in \cref{sec:consequences}.

In some cases, \cref{intro_upper_bound} and \cref{intro_refined_upper_bound} are sufficient to determine $r_\ell(C)$ precisely. 

\begin{corollary}\label{cor:small primes}
    Suppose $\ell=5$. If $\gamma=4\mid d$ then $r_5(C)=B'$, and otherwise $r_5(C)=B$.
\end{corollary}
\begin{proof}
    If $\gamma\leq 2$ then \cref{thm:intro example} implies $r_5(C)=B$, so the only remaining option to consider is $\gamma=4$. If $\gcd(d,\gamma)=1$ then $B=0$, and if $\gcd(d,\gamma)=2$ then $B=1$; in both cases we must have $r_5(C)=B$ by \cref{intro_upper_bound}. So we may now assume $4\mid d$.    
    If $\mathcal{T}\neq \{1,2,3\}$, then $B'=|\mathcal{T}|$ is either $1$ or $2$. In either case $\min\{B',2\}=B'$, so $r_5(C)=B'$ by \cref{intro_refined_upper_bound}. If $\mathcal{T}= \{1,2,3\}$, then $r_5(C)\leq B'=3$, but we also have $1+\frac{\gamma}{2}=3\in\mathcal{T}$ and so $r_5(C)\geq 3$, so again $r_5(C)=B'$ by \cref{intro_refined_upper_bound}.
\end{proof}
For larger values of $\ell$, \cref{intro_upper_bound} and \cref{intro_refined_upper_bound} are not sufficient to determine $r_\ell(C)$. For example, we have the following options when $\ell=7$:
    \[r_7(C)=\left\{\begin{array}{ll}
        2\text{ or }4, & \text{if }\gamma=3\mid d\text{ and }|\mathcal{T}|=2,\\
        3\text{ or }5, & \text{if }\gamma=6\mid d\text{ and }|\mathcal{T}|=5,\\
        B', & \text{ if }3\leq \gamma\mid d\text{ but not the above cases,} \\
        B, & \text{ otherwise.}
    \end{array}\right.\]
For the first two rows, we can exhibit curves attaining both possible values of $r_7(C)$, demonstrating that the parameters $\ell,q,d,|\mathcal{T}|$ are not sufficient to fully determine the value of $r_\ell(C)$ in general. For instance, consider the case $\ell=7$, $q=3$ (so $\gamma=6$), and $d=12$ from the introduction, summarized in \cref{tab:7data}. We may categorize these  function fields further by the sets $\mathcal{T}$ associated to each. See \cref{tab:refined-7data}, and note in particular the last two columns, consisting of curves with the same $\mathcal{T}$ but different values of $r_\ell(C)$. \footnote{The astute reader may notice in \cref{tab:refined-7data} that $\mathcal{T}\setminus\{1\}$ is closed under $n\mapsto 1-n\bmod\gamma$. This symmetry does hold in general, following from \cref{thm:rooftop_pairs} and \cref{thm:congruence_conditions} below, and can be used to cut down the number of computations needed in order to find the set $\mathcal{T}$.}

\begin{table}[h!]
    \centering
    \begin{tabular}{r|ccccc}
        $r_7(C)$: & $1$ & $3$ & $3$ & $3$ & $5$\\
        $\mathcal{T}$: & $\{1\}$ & $\{1,3,4\}$ & $\{1,2,5\}$ & $\{1,2,3,4,5\}$ & $\{1,2,3,4,5\}$ \\
        \hline
        count: & $5552$ & $852$ & $810$ & $178$ & $12$
    \end{tabular}
    \caption{The number of isomorphism classes of fields $\F_3(\sqrt[7]{f},x)$ attaining each possible $7$-rank and set $\mathcal{T}$, with $f(x)\in\F_3[x]$ irreducible of degree $12$.}
    \label{tab:refined-7data}
\end{table}

\subsection{Results from the linear algebra of Frobenius eigenvectors}\label{sec:lifting_behavior_functions}

The most important feature of working with function fields is that we can represent elements of the ideal class group of $\F_q(\sqrt[\ell]{f},x)$ using geometric objects, as described in \cref{sec:rlC_constraints}. This allows us to use morphisms from $C$ to itself to study $r_\ell(C)$.

The $q$-power Frobenius map $(x,y)\to (x^q,y^q)$ on $C(\Fbar_q)$ induces a linear operator $\Frob$ on the $\F_\ell$-vector space $J[\ell]$. The  eigenspace of eigenvalue $1$ for this action is $J[\ell](\F_q)$, so $r_\ell(C)$ can be recovered as the dimension of this eigenspace. The primary difficulty we will encounter is that the action of $\Frob$ on $J[\ell]$ is not semi-simple in general. The action of Frobenius on the $\ell$-adic Tate module $J[\ell^\infty]$ is semi-simple, but this property does not descend to the mod $\ell$ reduction. So even though we can determine the full characteristic polynomial of $\Frob$ acting on $J[\ell]$ with relatively little work (\cref{rmk:full-basis}), this is not enough to determine the dimension of any particular eigenspace.

To study $r_\ell(C)$, we use a filtration coming from the automorphism $(x,y) \mapsto (x,\zeta_\ell y)$ on the space $J[\ell]$ denoted as
\[0=V^0\subseteq V^1\subseteq V^2\subseteq \cdots \subseteq V^{\ell-1}=J[\ell]\]
that is preserved by $\Frob$, and we consider the intersection of \emph{generalized} eigenspaces for $\Frob$ with this filtration.

\begin{restatable}{definition}{EnkFnk}
    \label{def:EnkFnk}
    For $1\leq k\leq \ell-1$ and $n\in\Z/\gamma\Z$, let $F_n^k$ denote the set of $v\in V^k\setminus V^{k-1}$ for which 
    \[(\Frob-q^{n-k+1})^iv=0\qquad\text{for some } i\geq 1.\]
\end{restatable}

Note that for $n\in\Z/\gamma\Z$, multiplication by $q^{n-k+1}$ is a well-defined scalar operator on $J[\ell]$ because $q^\gamma=1$ in $\F_\ell$.
We will show that 
$J[\ell](\F_q)$ has a basis formed by taking one true eigenvector of $\Frob$ from each $F_{k-1}^k$, whenever such an eigenvector exists. Thus $r_\ell(C)$ is directed related to $F_n^k$ in the following result.

\afterpage{
\clearpage
\null
\vfill
\begin{figure}[h]
    \centering
    \usetikzlibrary{math} 
\tikzmath{\a0=0;\a1=4; \a2=1;\a3=4;\a4=6;\a5=1;}
\begin{tabular}{l|r}
\begin{tikzpicture}[scale=0.9]
\draw  (0,0) edge (6,0);
\draw  (0,0) edge (0,6);
\draw[draw=none, fill=gray!15]  (1,0) rectangle (6,6);
\draw[fill=gray!35]  (0,0) rectangle (1,\a0);
\draw[fill=gray!35]  (1,0) rectangle (2,\a1);
\draw[fill=gray!35]  (2,0) rectangle (3,\a2);
\draw[fill=gray!35]  (3,0) rectangle (4,\a3);
\draw[fill=gray!35]  (4,0) rectangle (5,\a4);
\draw[fill=gray!35]  (5,0) rectangle (6,\a5);
\node at (-0.5,0.5) {$1$};
\node at (-0.5,1.5) {$2$};
\node at (-0.5,2.5) {$3$};
\node at (-0.5,3.5) {$4$};
\node at (-0.5,4.5) {$5$};
\node at (-0.5,5.5) {$6$};
\node at (-0.5,6.5) {$k$};
\node at (6.5,-0.5) {$n$};
\node at (0.5,-0.5) {$0$};
\node at (1.5,-0.5) {$1$};
\node at (2.5,-0.5) {$2$};
\node at (3.5,-0.5) {$3$};
\node at (4.5,-0.5) {$4$};
\node at (5.5,-0.5) {$5$};
\node at (1.5,\a1-.5) {$F_1^{\a1}$};
\node at (2.5,\a2-.5) {$F_2^{\a2}$};
\node at (3.5,\a3-.5) {$F_3^{\a3}$};
\node at (4.5,\a4-.5) {$F_4^{\a4}$};
\node at (5.5,\a5-.5) {$F_5^{\a5}$};
\draw  (0.5,0.5) ellipse (0.5 and 0.5);
\draw  (1.5,1.5) ellipse (0.5 and 0.5);
\draw  (2.5,2.5) ellipse (0.5 and 0.5);
\draw  (3.5,3.5) ellipse (0.5 and 0.5);
\draw  (4.5,4.5) ellipse (0.5 and 0.5);
\draw  (5.5,5.5) ellipse (0.5 and 0.5);
\node at (3,-1) {$\ell=7$, $\gamma=d=6$, $r_{7}=3$};
\end{tikzpicture}
&
\tikzmath{\a0=0; \a1=10; \a2=6;\a3=1;\a4=6;}
\multirow{2}{*}[14em]{
\begin{tikzpicture}[scale=0.9]
\draw  (0,0) edge (5,0);
\draw  (0,0) edge (0,10);
\node at (2.5,-1) {$\ell=11$, $\gamma=d=5$, $r_{11}=4$};
\draw[draw=none, fill=gray!15]  (1,0) rectangle (5,10);
\draw[fill=gray!35]  (0,0) rectangle (1,\a0);
\draw[fill=gray!35]  (1,0) rectangle (2,\a1);
\draw[fill=gray!35]  (2,0) rectangle (3,\a2);
\draw[fill=gray!35]  (3,0) rectangle (4,\a3);
\draw[fill=gray!35]  (4,0) rectangle (5,\a4);
\node at (-0.5,0.5) {$1$};
\node at (-0.5,1.5) {$2$};
\node at (-0.5,2.5) {$3$};
\node at (-0.5,3.5) {$4$};
\node at (-0.5,4.5) {$5$};
\node at (-0.5,5.5) {$6$};
\node at (-0.5,6.5) {$7$};
\node at (-0.5,7.5) {$8$};
\node at (-0.5,8.5) {$9$};
\node at (-0.5,9.5) {$10$};
\node at (-0.5,10.5) {$k$};
\node at (5.5,-0.5) {$n$};
\node at (0.5,-0.5) {$0$};
\node at (1.5,-0.5) {$1$};
\node at (2.5,-0.5) {$2$};
\node at (3.5,-0.5) {$3$};
\node at (4.5,-0.5) {$4$};
\node at (1.5,\a1-.5) {$F_1^{\a1}$};
\node at (2.5,\a2-.5) {$F_2^{\a2}$};
\node at (3.5,\a3-.5) {$F_3^{\a3}$};
\node at (4.5,\a4-.5) {$F_4^{\a4}$};
\draw  (0.5,0.5) ellipse (0.5 and 0.5);
\draw  (1.5,1.5) ellipse (0.5 and 0.5);
\draw  (2.5,2.5) ellipse (0.5 and 0.5);
\draw  (3.5,3.5) ellipse (0.5 and 0.5);
\draw  (4.5,4.5) ellipse (0.5 and 0.5);
\draw  (0.5,5.5) ellipse (0.5 and 0.5);
\draw  (1.5,6.5) ellipse (0.5 and 0.5);
\draw  (2.5,7.5) ellipse (0.5 and 0.5);
\draw  (3.5,8.5) ellipse (0.5 and 0.5);
\draw  (4.5,9.5) ellipse (0.5 and 0.5);
\end{tikzpicture}
}
\\ \cline{1-1}
\tikzmath{\a0=0; \a1=0; \a2=6;\a3=0;\a4=6;\a5=0;}
\begin{tikzpicture}[scale=0.9]
\draw  (0,0) edge (6,0);
\draw  (0,0) edge (0,6);
\node at (3,-1) {$\ell=7$, $\gamma=6$, $d=3$, $r_{7}=2$};
\node at (0,7) {};
\draw[fill=gray!35]  (0,0) rectangle (1,\a0);
\draw[fill=gray!35]  (1,0) rectangle (2,\a1);
\draw[fill=gray!35]  (2,0) rectangle (3,\a2);
\draw[fill=gray!35]  (3,0) rectangle (4,\a3);
\draw[fill=gray!35]  (4,0) rectangle (5,\a4);
\draw[fill=gray!35]  (5,0) rectangle (6,\a5);
\node at (-0.5,0.5) {$1$};
\node at (-0.5,1.5) {$2$};
\node at (-0.5,2.5) {$3$};
\node at (-0.5,3.5) {$4$};
\node at (-0.5,4.5) {$5$};
\node at (-0.5,5.5) {$6$};
\node at (-0.5,6.5) {$k$};
\node at (6.5,-0.5) {$n$};
\node at (0.5,-0.5) {$0$};
\node at (1.5,-0.5) {$1$};
\node at (2.5,-0.5) {$2$};
\node at (3.5,-0.5) {$3$};
\node at (4.5,-0.5) {$4$};
\node at (5.5,-0.5) {$5$};
\node at (2.5,\a2-.5) {$F_2^{\a2}$};
\node at (4.5,\a4-.5) {$F_4^{\a4}$};
\draw  (0.5,0.5) ellipse (0.5 and 0.5);
\draw  (1.5,1.5) ellipse (0.5 and 0.5);
\draw  (2.5,2.5) ellipse (0.5 and 0.5);
\draw  (3.5,3.5) ellipse (0.5 and 0.5);
\draw  (4.5,4.5) ellipse (0.5 and 0.5);
\draw  (5.5,5.5) ellipse (0.5 and 0.5);
\end{tikzpicture}
\end{tabular}
    \caption[Caption for LOF]{Examples of possible $\Frob$ eigenvector configurations. Each chart represents one curve\footnotemark \ with some specified parameters $\ell,\gamma,d$. Cell $(n,k)$ is colored light gray if $F_n^{k}$ is nonempty, and dark gray if $F_n^k$ contains an eigenvector of $\Frob$. Cell $(n,k)$ is circled if $n-k+1\equiv 0\pmod\gamma$. Rooftops $F_n^k$ are labeled. The number of dark gray circles equals to
    $r_\ell(C)$. 
    For more on how to read these diagrams see \cref{rmk: visual guide}.}
    \label{fig:lifting_charts}
\end{figure}
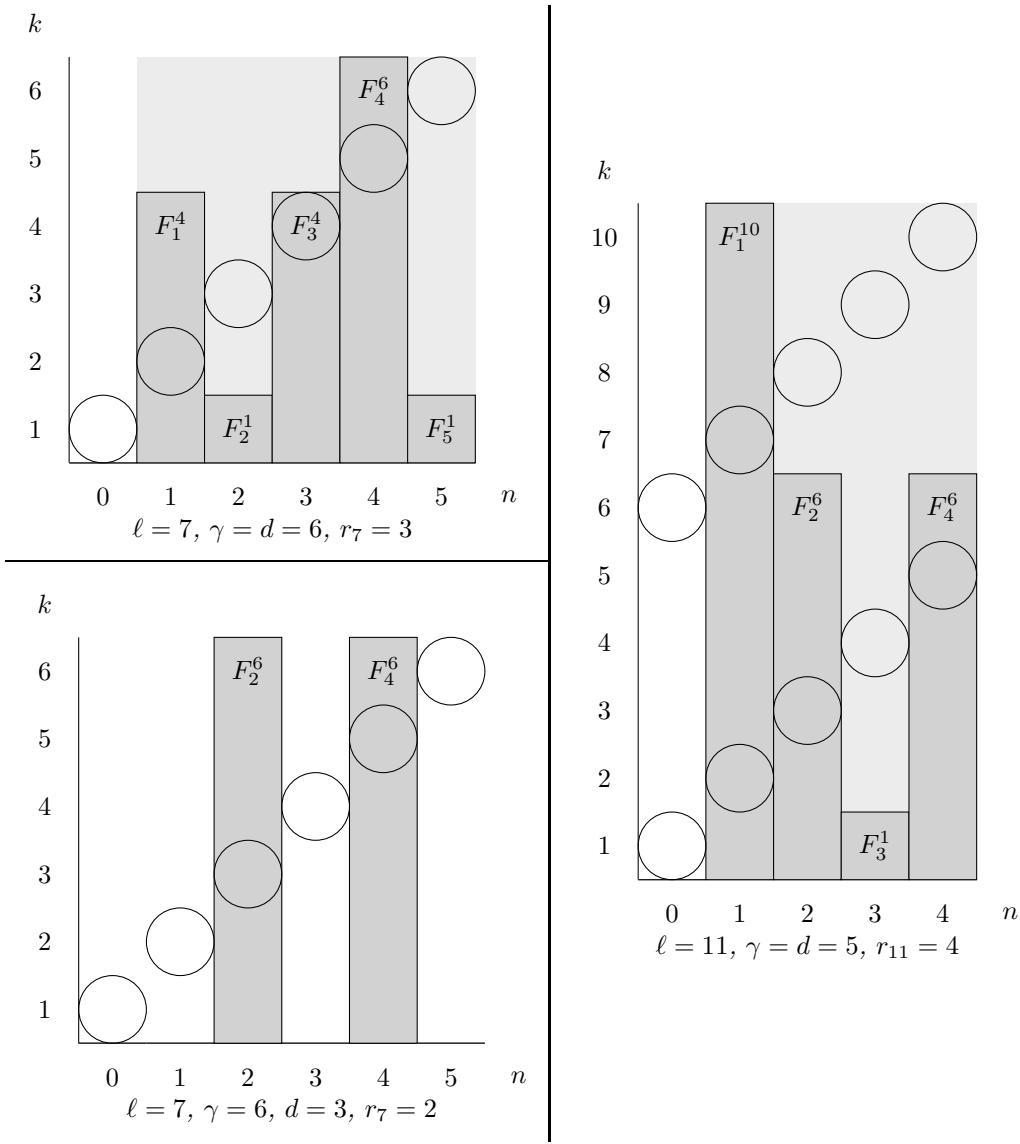
\vfill
\footnotetext{These diagrams do not come from explicit curves, but are examples of configurations satisfying all the combinatorial constraints discussed in \cref{sec:lifting_behavior_functions}.} 
\clearpage
}

\begin{restatable}{theorem}{rlCcount}
    \label{thm:rlC_count} The $\ell$-rank of the divisor class group
    $r_\ell(C)$ equals the number of values $1\leq k\leq \ell-1$ such that $F_{k-1}^k$ contains an eigenvector of $\Frob$.
\end{restatable}

For a visual interpretation, see \cref{fig:lifting_charts}. Each of the three figures represents a different possible curve. If the cell with coordinates $(n,k)$ is shaded dark gray, this means that there exists a $\Frob$ eigenvector with eigenvalue $q^{n-k+1}$ in $V^k\setminus V^{k-1}$. The circles correspond to cells $(n,k)$ with $q^{n-k+1}=1$ in $\F_\ell$, so any dark gray circle corresponds to a $\Frob$ eigenvector of eigenvalue $1$, that is, an element of $J[\ell](\F_q)$. \textbf{The number of dark gray circles equals to
    $r_\ell(C)$.} See \cref{rmk: visual guide} for a more thorough guide to reading these diagrams.

We prove several constraints that determine when $F_n^{k}$ contains an eigenvector of $\Frob$. A number of relations follow fairly directly from linear algebra on $J[\ell]$. We say that $F_n^{k}$ is a \textbf{``rooftop''} if $F_n^k$ has a $\Frob$ eigenvector but there is no $k<k'\leq \ell-1$ such that $F_n^{k'}$ has an eigenvector. This notion is justified by statement (b) of the following Proposition. We say a rooftop $F_n^{k}$ is \textbf{``non-maximal''} if $k\neq\ell-1$.

\begin{restatable}{proposition}{basicliftingproperties}
    \label{basic_lifting_properties}
    Let $n\in\Z/\gamma\Z$ and $1\leq k\leq \ell-1$. We have the following:
    \begin{enumerate}[label=(\alph*)]
        \item $F_n^1$ has a $\Frob$ eigenvector if and only if $\gamma\mid dn$ and $\gamma\nmid n$.
        \item If $F_n^k$ has a $\Frob$ eigenvector, then $F_n^{k'}$ has a $\Frob$ eigenvector for all $1 \le  k'\leq k$.
        \item If $F_n^k$ is a rooftop, then $F_{n-k}^1$ has a $\Frob$ eigenvector (that is, $\gamma\mid d(n-k)$ and $\gamma\nmid (n-k)$).
        \item If $F_n^k$ is a non-maximal rooftop, there is no rooftop of the form $F_{n+i}^{k+i}$ for $i\neq 0$ and $1\leq k+i\leq \ell-2$.
    \end{enumerate}
\end{restatable}

These relations are proven in \cref{sec:proofs_of_basic_lifting}. Part (a) is obtained by constructing an explicit basis of $\Frob$ eigenvectors for $V^1$, and part (b) follows from a relation between $\Frob$ and the map used to define the filtration. If $F_n^k$ is a non-maximal rooftop, we will show that applying $(\Frob-q^{n-k})$ to a vector in $F_n^{k+1}$ yields an eigenvector in $F_{n-k}^1$ (a ``Jordan chain'' of length $2$), giving part (c). Part (d) comes from that if two such Jordan chains exist, a linear combination of the two generalized eigenvectors will produce an eigenvector.

The remaining three constraints \cref{thm:rooftop_pairs}, \cref{thm:evenselfdual}, and \cref{thm:congruence_conditions} are less straightforward, and form the main technical contributions of this paper. The first two can be proven using the Weil pairing on $J[\ell]$, by proving a numerical relation between pairs of Jordan chains (\cref{lem:pairingrelation}).

\begin{restatable}{theorem}{rooftoppairs}
    \label{thm:rooftop_pairs}
    Suppose $1\leq k\leq \ell-2$ and $n\in\Z/\gamma\Z$. Then $F_n^k$ is a rooftop if and only if $F_{k-n}^k$ is a rooftop.
\end{restatable} 

When $k\equiv 2n\bmod\gamma$, the above result is vacuous; however in this case we will see that the numerical relation gives us the following result.

\begin{restatable}{theorem}{evenselfdual}
    \label{thm:evenselfdual}
    Suppose $1\leq k\leq \ell-2$ and $n\in\Z/\gamma\Z$. If $k\equiv 2n\bmod\gamma$ and $k$ is even then $F_n^k$ is not a rooftop.
\end{restatable}

Both \cref{thm:rooftop_pairs} and \cref{thm:evenselfdual} are proved in \cref{sec:lifting_relations_new}. 
\Cref{thm:rooftop_pairs} says that $F_n^k\mapsto F_{k-n}^k$ defines an involution on the set of non-maximal rooftops, and \cref{thm:evenselfdual} imposes a parity constraint on the fixed points of this involution; these two observations will be used together in \cref{subsec:parity} to prove the parity constraint $r_\ell(C)\equiv B\bmod 2$. 

\subsection{Results from Galois cohomology}\label{subsec:results from cohom}

Another method we use to study the $\ell$-rank of $J[\ell](\F_q)$ comes from Galois cohomology. This approach is more closely related to the approach used to study the number field version of this problem, where there is no direct analogue of the geometric $\ell$-torsion subgroup $J[\ell](\overline{\F}_q)$ (see \cref{sec:priorwork}). In \cref{sec:galoisreps} and \cref{sec:galois_cohom} we use Kummer theory to relate the two perspectives. The culmination of these sections is \cref{prop:eigen to cohom}, which relates the existence of an eigenvector of $\Frob$ in $F_n^k$ to the existence of a certain cohomology class in a Selmer group associated to a $(k+1)$-dimensional representation of $\Gal(\F_q(x)^\sep /\F_q(x))$.

In \cref{sec:cup products} we show that $F_n^k$ is a non-maximal rooftop if and only if a certain cup product of cohomology classes does not vanish. The vanishing of this cup product can be determined by a residue field calculation, leading us to the last constraint.

\begin{restatable}{theorem}{congruenceconditions}
    \label{thm:congruence_conditions}
    Suppose $2\leq \gamma\mid d$ and $n\in\Z/\gamma\Z$. Then $F_n^2$ has a $\Frob$ eigenvector if and only if $n\in\mathcal{T}$ (as in \cref{eq:local_cup}).
\end{restatable}

We note that \cref{basic_lifting_properties} and \cref{thm:rooftop_pairs} can also be proven entirely using this Galois cohomology framework: for \cref{thm:rooftop_pairs}, instead of the Weil pairing we one can use Poitou-Tate duality. On the other hand, we have not yet found a way to prove \cref{thm:evenselfdual} using this framework. The key difficulty comes from determining whether a Selmer class associated to a self-dual representation lifts to a Selmer class associated to a higher-dimensional representation. These self-dual representations are quite difficult to work with compared to their non-self-dual counterparts, so our geometric proof of \cref{thm:evenselfdual} using the Weil pairing illustrates a method that can be used to work with them in the function field setting. On the other hand, we were only able to prove \cref{thm:congruence_conditions} using cohomological techniques. Thus, using both the geometric approach (Sections \ref{sec:geo_to_reps}--\ref{sec:lifting_relations_new}) and the cohomological approach (Sections \ref{sec:galoisreps}--\ref{sec:cup products}) allows us to prove stronger results than any one approach individually.

Together with \cref{basic_lifting_properties} and \cref{thm:rlC_count}, constraints on $r_\ell(C)$ can be obtained by counting arguments, analyzing the various restrictions on pairs $(n,k)$. \Cref{sec:consequences} contains proofs of some such constraints, including all the results stated in \cref{sec:rlC_constraints}.

\subsection{Prior work}\label{sec:priorwork}

The study of $\ell$-torsion in divisor class groups of superelliptic extensions $K(\sqrt[\ell]{f},x)/K(x)$ (for some field $K$) has been explored in many other contexts; for some examples~\cite{wawrow,jedrzejak}. Most of these explorations are largely independent from the content of this paper; for instance, some take $K=\Q$ instead of $K=\F_q$, and they impose different conditions on $\ell$ and $f(x)$. Further, these works typically focus on a particular subgroup of the $\ell$-torsion generated by divisors supported at the ramification locus of $f(x)$, which is the first stage $V^1$ in the filtration of $J[\ell]$ discussed in \cref{sec:lifting_behavior_functions}. 

In the case of hyperelliptic function fields $\F_q(\sqrt{f},x)/\F_q(x)$, this first stage $V^1$ contains the entirety of the $2$-torsion; Cornelissen uses this to compute the $2$-rank of $J[2](\F_q)$ for arbitrary hyperelliptic curves over $\F_q$ (allowing $f(x)$ to be reducible)~\cite{cornelissen}. However, for $\ell>2$, there is more to the filtration than this first stage, and these deeper filtration stages are one of the primary focuses of this paper. The primary difficulty we face is that unlike the action of Frobenius on $V^1$, the action of $\Frob$ on $J[\ell]$ as a whole is not semi-simple. See \cref{rmk:full-basis} for a discussion.

The aforementioned filtration can be defined using an endomorphism $(x,y)\mapsto (x,\zeta_\ell y)$ on $C$, where $\zeta_\ell$ is some $\ell^{\text{th}}$ root of unity in $\overline{\F}_q$. This endomorphism and the filtration it defines were used by Poonen--Schaefer~\cite{poonen_schaefer} and were further explored by Arul~\cite{arul2020}. 

Other authors have studied the $\ell$-torsion subgroups of divisor class groups of different kinds of degree $\ell$ extensions of $\F_q(x)$. For instance, Wittmann considered degree $\ell$ Galois extensions $K/\F_q(x)$, and studied the Galois module structure of the $\ell$-torsion in the divisor class group of $K$ \cite{Wittmann}.

The question this paper is exploring has a direct analogue in number fields: namely, to study the $p$-rank of the ideal class group of $\Q(N^{1/p})$ for distinct primes $N$ and $p$. Several authors have studied this question under the assumption $N \equiv 1 \bmod p$, which is analogous to the assumption $\gamma\mid d$ we make in \cref{intro_refined_upper_bound}. Using deformations of Galois representations, Calegari--Emerton~\cite{calegari_emerton} determined conditions under which the $p$-part is cyclic (i.e.~$p$-rank $1$). For instance, one of their results is that if $\prod_{i=1}^{(N-1)/2}i^i$
is a $p$-th power modulo $N$, then the $p$-rank of $\Cl(\Q(N^{1/p}))$ is at least $2$~\cite[Theorem 1.3(ii)]{calegari_emerton}. These results were generalized by Wake--Wang-Erickson~\cite[Proposition 11.1.1]{wake_wangerickson}; in particular, they interpreted the congruence condition as a cup product on Galois cohomology. The techniques of Wake--Wang-Erickson were used by Karl Schaefer and the third author to prove a full converse of Calegari--Emerton's result by imposing additional congruence conditions. 

The cohomological methods used in Sections \ref{sec:galoisreps}--\ref{sec:cup products} of this paper closely follow the work of Schaefer and the third author, using the Galois cohomology framework developed by Wake--Wang-Erickson. In particular, the upper bound in \cref{intro_refined_upper_bound} is directly analogous to \cite[Theorem 1.1.1]{schaeferstubley}. On the other hand, \cite[Table 3]{schaeferstubley} shows that there is no parity constraint on the $p$-rank in the number field setting; the parity constraint on $r_\ell(C)$ appears to be a phenomenon unique to function fields.

\subsection{Acknowledgements}

The bulk of this research was conducted while all three authors held a CRM-ISM postdoctoral fellowship. 
The first author was partially supported by NSF grant DMS-2302511, and the second author was partially supported by ERC Starting Grant 101076941 (`\textsc{Gagarin}').

The authors thank Patrick Allen, Jordan Ellenberg, Jaclyn Lang, Bjorn Poonen, Karl Schaefer, Jacob Stix, Yunqing Tang, Carl Wang-Erickson for conversations that pointed them in helpful directions. 
The first two authors want to thank the third author for suggesting this project and for introducing them to the technical details of Galois cohomology needed for this paper, and the third author wishes to thank his collaborators for seeing this project through after he left academia.

\section{Structure of the $\ell$-torsion subgroup}\label{sec:geo_to_reps}

In this section, we introduce our setup and prove \cref{thm:rlC_count} and \cref{basic_lifting_properties}.

\subsection{Notation and Setup}\label{sec:setup}

We use the following notation throughout the paper.
\begin{itemize}
    \item $\ell\geq 3$ is a prime.
    \item $q$ is a prime power coprime to $\ell$.
    \item $\gamma$ is the multiplicative order of $q$ in $(\Z/\ell\Z)^\times$.
    \item $\zeta\in\F_{q^\gamma}$ is a fixed nontrivial $\ell^{\text{th}}$ root of unity.
    \item $f(x)\in \F_q[x]$ is an irreducible polynomial with $d:=\deg f$ coprime to $\ell$.
    \item $C/\F_q$ is the smooth projective curve with affine equation $y^\ell=f(x)$, and $C_{\Fbar_q}$ its base change to $\Fbar_q$.
    \item $J/\F_q$ is the Jacobian of $C$, and $J_{\Fbar_q}$ its base change to $\Fbar_q$.
    \item For a field extension $\F/\F_q$, $J(\F)$ denotes the $\F$-points of $J$. Elements of $J(\Fbar_q)$ can be interpreted as divisors on $C_{\Fbar_q}$ modulo linear equivalence, and if $\F/\F_q$ is a finite extension, elements of $J(\F)$ correspond to divisor classes in $J(\Fbar_q)$ that are invariant under $\Gal(\Fbar_q/\F)$.
    \item $J[\ell](\F)$ denotes the $\ell$-torsion subgroup of $J(\F)$, and $J[\ell]:=J[\ell](\Fbar_q)$ the geometric $\ell$-torsion group.
    \item The $\ell$-torsion rank of $C$ is defined to be
\[r_\ell(C):=\dim_{\F_{\ell}}(J[\ell](\F_{q})).\]
\end{itemize}

We also define two morphisms of $C_{\F_\qg}$ by giving their actions on geometric points $(x,y)\in C(\Fbar_q)$. Using the $\ell^{\text{th}}$ root of unity $\zeta\in\F_\qg$ chosen above, by abuse of notation we also let $\zeta:C_{\F_\qg}\to C_{\F_\qg}$ denote the morphism defined by
\[\zeta:(x,y)\mapsto (x,\zeta y).\]
We let $\Frob:C_{\F_\qg}\to C_{\F_\qg}$ denote the \emph{relative Frobenius} map,
\[\Frob:(x,y)\mapsto (x^q, y^q).\]
Note that $\zeta$ is an automorphism of $C_{\F_\qg}$, while $\Frob$ is a degree $q$ endomorphism; both act invertibly on $C(\Fbar_q)$. On $C(\Fbar_q)$, we also have the relation
\[\Frob\circ \zeta=\zeta^q\circ\Frob.\]
By further abuse of notation, we also let $\zeta$ and $\Frob$ denote the respective endomorphisms of $J_{\Fbar_q}$ induced by their namesakes, as well as the induced linear maps on $J[\ell]$ considered as a vector space over $\F_\ell$.

\begin{convention}
    When not otherwise specified, an ``eigenvector'' will refer to an eigenvector of $\Frob$ acting as a linear map on the $\F_\ell$--vector space $J[\ell]$, i.e.~a nonzero $v\in J[\ell]$ satisfying $\Frob v= cv$ for some $c\in\F_\ell$. Likewise, a ``generalized eigenvector'' will refer to a generalized eigenvector of $\Frob$, i.e.~a nonzero $v\in J[\ell]$ satisfying $(\Frob -c)^iv=0$ for some $c\in\F_\ell$ and $i\geq 0$).
\end{convention}

\subsection{The $1-\zeta$ Filtration of $J[\ell]$}\label{sec:filtration}
The automorphism $\zeta$ and endomorphism $1-\zeta$ on $J_{\Fbar_q}$ were discussed in detail in \cite[Section 2.3]{arul2020}. Here we use them to construct a filtration on $J[\ell]$.

Noting that the endomorphism $\zeta$ is annihilated by the $\ell^{\text{th}}$ cyclotomic polynomial, we can derive the relation
\[\prod_{i=1}^{\ell-1}(1-\zeta^i)=\ell\]
in the endomorphism ring $\End(J_{\Fbar_q})$. For each $1\leq i\leq \ell-1$, $1-\zeta^i$ is equal to $1-\zeta$ times a unit, and 
so $(1-\zeta)^{\ell-1}$ is $\ell$ times a unit. We can conclude that the kernel of $(1-\zeta)^{\ell-1}$ on $J_{\Fbar_q}$ is exactly $J[\ell]$.

Define $V^{k}$ to be the $\Fbar_q$-points of $\ker\,(1-\zeta)^{k}$; these $V^{k}$ then give a filtration
\[
0=V^0 \subset V^{1} \subset \ldots \subset V^{\ell-1} = J[\ell](\Fbar_q).
\]
Note that each subgroup $V^k$ has the structure of a $\F_\ell$-vector space. 
\begin{lemma}\label{orig_subquotient_iso}
    For each $k=2,\ldots,\ell-1$, the endomorphism $1-\zeta$ induces an isomorphism $V^k/V^{k-1}\to V^{k-1}/V^{k-2}$ of $(d-1)$-dimensional vector spaces over $\F_\ell$.
\end{lemma}
\begin{proof}
    We have $N_{\Q(\zeta)/\Q}(1-\zeta)=\ell$, $[\Q(\zeta):\Q]=\ell-1$, and the curve $C$ has genus $g=\frac12(\ell-1)(d-1)$ by Riemann-Hurwitz. So for all $0\leq k\leq \ell-1$, 
    \[\deg(1-\Czeta)^k=\ell^{k(d-1)}\]
    by \cite[Proposition 12.12]{Milne1986}. These endomorphisms are all separable and so $\ker\,(1-\Czeta)^{k}$ has $\ell^{k(d-1)}$ points in $J(\Fbar_q)$. This implies $\dim_{\F_\ell}V^k=k(d-1)$. For $2\leq k\leq \ell-1$, the kernel of the map $V^k\to V^{k-1}/V^{k-2}$ induced by $1-\zeta$ is $V^{k-1}$, so $(1-\zeta):V^k/V^{k-1}\to V^{k-1}/V^{k-2}$ is an isomorphism by dimension considerations. 
\end{proof}

\subsection{A modification of $1-\zeta$}

Recall from \cref{sec:setup} that the relative Frobenius map $\Frob:J_{\Fbar_q}\to J_{\Fbar_q}$ is induced by the action $(x,y)\mapsto (x^q, y^q)$ on $C(\Fbar_q)$. The maps $\Frob$ and $1-\zeta$ on $J_{\Fbar_q}$ satisfy the relation
\begin{align}\label{eq:comm_relation}
    \Frob\circ(1-\zeta)=(1-\zeta^q)\circ\Frob.
\end{align}
Since $1-\zeta^q$ and $1-\zeta$ are associates in $\End(J_{\Fbar_q})$, this identity shows that the action of $\Frob$ on $J[\ell](\Fbar_q)$ preserves the filtration stages $V^k$. However, the automorphism $\frac{1-\zeta^q}{1-\zeta}$ of $J_{\Fbar_q}$ does not preserve the generalized eigenspaces of $\Frob$. To account for this, we introduce a modification of $1-\zeta$ that interacts in a more predictable way with the Frobenius map.

\begin{definition}\label{def:eta}
    Let $\eta\in \End(J_{\Fbar_q})$ be defined by
    \[\eta:=-\sum_{i=1}^{\ell-2}i^{-1}(1-\zeta)^i,\] 
    where $i^{-1}\in\Z$ denotes an inverse of $i$ modulo $\ell$.
\end{definition}
While the endomorphism $\eta$ depends on the choice of inverses mod $\ell$, the action of $\eta$ on $J[\ell]$ is well-defined, independent of the choice of $i^{-1}$ for each $i$.
We have the following two important facts about $\eta$, which both capture the idea that $\eta$ behaves like a ``logarithm'' of $\zeta$. The first statement in the following lemma says that $\eta$ acts like $\zeta-1$ up to higher-order terms.

\begin{lemma}\label{subquotient_iso}
    We have $V^1=\ker\eta\cap J[\ell]$, and for each $2\leq k\leq \ell-1$, $\eta$ and $\zeta-1$ are equal as isomorphisms $V^k/V^{k-1}\to V^{k-1}/V^{k-2}$.
\end{lemma}
\begin{proof}
    This follows from \cref{orig_subquotient_iso} and the fact that $\eta+(1-\zeta)$ is in the ideal generated by $\ell$ and $(1-\zeta)^2$.
\end{proof}

The second statement about $\eta$ in the following lemma can be thought of as a linearization of the relation $\Frob\circ\zeta=\zeta^q\circ\Frob$. As a note of caution, the following relation does \emph{not} hold when $\Frob$ and $\eta$ are considered as endomorphisms of $J_{\Fbar_q}$; we obtain the desired equality only when we restrict to the actions on $J[\ell]$.

\begin{lemma}\label{lem:linearcommute}
    As linear maps on $J[\ell]$, 
    \[\Frob\circ \eta=q\eta\circ\Frob.\]
\end{lemma}

\begin{proof}
    This can be proven by formal manipulation of polynomials; see \cref{appendix_linearcommute}.
\end{proof}

A consequence of this result is that if $v\in J[\ell]$ is a (generalized) eigenvector of $\Frob$, then so is $\eta v$.

\subsection{Generalized eigenvectors of $\Frob$}\label{sec:basisJl}

Any $v\in J[\ell](\F_q)$ lies in some filtration stage $v\in V^k\setminus V^{k-1}$. Then $\eta^{k-1}v\in V^1$, and by \cref{lem:linearcommute} we have 
\[\Frob(\eta^{k-1}v)=q^{k-1}\eta^{k-1}(\Frob v)=q^{k-1}\eta^{k-1} v,\]
so $\eta^{k-1}v$ is an eigenvector of eigenvalue $q^{k-1}$. So our first goal is to identify which powers of $q$ arise as eigenvalues of $\Frob$ acting on $V^1$.

To start, we find a basis of $V^1$ which is most suitable for our study of the $\Frob$ action. The dimension of $V^1$ is $d-1$ for $d:=\deg f$. The following Lemma, which is well-known in the literature (see e.g.~the proof of \cite[Theorem 1.7]{ELS}), shows that the action of $\Frob$ on $V^1$ is diagonalizable over $\overline{\F}_\ell$.

\begin{lemma}\label{V0basis}
    The action of $\Frob$ on $V^1\otimes\Fbar_\ell$ has a basis
    \[\{u_\beta:\beta\in\Fbar_\ell,\;\beta^d=1,\;\beta\neq 1\}\]
    where $u_\beta$ is an eigenvector of $\Frob$ with eigenvalue $\beta$.
\end{lemma}
(Caution: recall that $\Frob$ is induced by the $q$-power Frobenius map on $J_{\Fbar_q}$, \emph{not} an $\ell$-power Frobenius map on $\Fbar_\ell$.)
\begin{proof}
    Suppose that $f(x)$ factors over $\Fbar_q$ as $(x - x_{1}) \ldots (x - x_{d})$, where $\Frob(x_{i}) = x_{i+1}$.
    The curve $C$ has a unique point above $\infty\in\Proj^1(\Fbar_q)$ by the assumption $\ell \nmid d$, which we also call $\infty$.
    Define the points $P_{i} := [(x_{i},0)] - [\infty]\in J(\Fbar_q)$ for $i=1,\ldots,d$. We have the relation $P_1+\cdots P_d=0$; by \cite[Proposition 2.3.1]{arul2020}, the points $P_1,\ldots,P_{d-1}$ form a basis for $V^1$ (see also \cite[Proof of Theorem 1.7]{ELS}).

    Given $\beta\in\Fbar_\ell$ satisfying $\beta^d=1$ and $\beta\neq 1$, set
    \begin{align}
        u_\beta:=\sum_{i=1}^d \beta^{-i} P_i\in V^1\otimes\Fbar_\ell.
    \end{align}
    Since $\beta\neq 1$, the coefficients of $P_1$ and $P_d$ are distinct and so $u_\beta\neq 0$. We have    \begin{align*}
        \Frob u_\beta &= \sum_{i=1}^{d-1} \beta^{-i} P_{i+1}+\beta^{-d}P_1\\
        &= \beta\left(\sum_{i=1}^{d-1} \beta^{-i-1} P_{i+1}+\beta^{-1}P_1\right)\\
        &= \beta u_\beta.
    \end{align*}
    Thus we have $d-1$ eigenvectors with distinct eigenvalues, so these form a basis for the $(d-1)$-dimensional vector space $V^1\otimes\Fbar_\ell$.
\end{proof}

For $2\leq k\leq \ell-1$, an eigenvector in $V^1$ with eigenvalue $q^n$ lifts under $\eta^{k-1}$ to an eigenvector of $\Frob$ in $V^k/V^{k-1}$ with eigenvalue $q^{n-k+1}$ (by \cref{subquotient_iso} and \cref{lem:linearcommute}). This lift is a priori only an eigenvector in the quotient space, but we show in \cref{Fnk_updown} that we can always take the lift to be a generalized eigenvector of $\Frob$ acting on $J[\ell]$ using properties of the operator $\eta$.

Recall the definition of the sets $F_n^k$.

\EnkFnk*

\noindent
Namely $F_n^k$ is the set of generalized eigenvectors of $\Frob$ in the $k$-th filtration stage with eigenvalue $q^{n-k+1}$.

\begin{remark}\label{rmk: visual guide}
    By this point we have developed a lot of notation, so it may be helpful to have a picture in mind as we proceed. Examples are provided in \cref{fig:lifting_charts}. To each curve $C$, we associate a grid of cells $(n,k)$ with $n\in\Z/\gamma\Z$ and $k=1,\ldots,\ell-1$. Roughly speaking, the shaded cells can be matched bijectively with an independent set of vectors in $J[\ell]$; rows (indexed by $k$) correspond to the filtration stages $V^k$; columns (indexed by $n$) are $\eta$-invariant subspaces; and each diagonal with $n-k\in\Z/\gamma\Z$ constant corresponds to a distinct $\Frob$ eigenvalue.

    More precisely, we shade the cell with coordinates $(n,k)$ light gray if the set $F_n^k$ is nonempty: that is, if there exists a generalized $\Frob$ eigenvector with eigenvalue $q^{n-k+1}$ in $V^k\setminus V^{k-1}$. \Cref{Fnk_updown}(a) says that $\eta$ acts on the grid by shifting everything down one cell, taking $F_n^k$ to $F_n^{k-1}$ and annihilating the bottom layer $k=1$. \Cref{Fnk_updown}(b) tells us that we can also go backwards: any shaded cell has a shaded cell above it. Thus each column is either entirely shaded or entirely empty. \Cref{cor:eigenvalues} tells us precisely which columns are shaded or empty. Cells $(n,k)$ and $(n',k')$ correspond to the same generalized $\Frob$ eigenvector if $q^{n-k+1}=q^{n'-k'+1}$, or equivalently, if $n-k\equiv n'-k'\bmod \gamma$.

    In \cref{sec:proofs_of_basic_lifting} we will see that true eigenvectors of $\Frob$ form ``towers'' in these grids, and in \cref{sec:lifting_relations_new} we will see that the Weil pairing imposes a kind of \emph{rotational symmetry} on these grids.
\end{remark}

\begin{lemma}\label{Fnk_updown}
    Let $n\in\Z/\gamma\Z$ and $1\leq k\leq \ell-1$. Suppose $w\in F_n^k$.
    \begin{enumerate}[label=(\alph*)]
        \item If $k\geq 2$, then $\eta w\in F_n^{k-1}$.
        \item If $k\leq \ell-2$, there exists $v\in F_n^{k+1}$ with $\eta v=w$.
    \end{enumerate}
\end{lemma}
\begin{proof}
    By \cref{subquotient_iso} we have $\eta w\in V^{k-1}\setminus V^{k-2}$. If $i\geq 1$ is such that $(\Frob- q^{n-k+1})^iw=0$, then by \cref{lem:linearcommute} we have
    \[(\Frob- q^{n-k+2})^i\eta w=q^i\eta(\Frob- q^{n-k+1})^iw=0,\]
    proving (a).

    To set up the proof of (b), we first note that the action of $\Frob$ on $V^1$ is semisimple (for instance by recalling from \cref{V0basis} that $V^1\otimes\Fbar_\ell$ splits into a direct sum of $\Frob$ eigenspaces). Let $U$ be the $q^{n-k}$ eigenspace of $\Frob$ acting on $V^1$ (note that $U$ may be $0$- or $1$-dimensional), and let $W$ be the $\Frob$-invariant complementary subspace of $U$ in $V^1$.

    Now suppose $k\leq \ell-2$ and $w\in F_n^k$. Since $\eta$ maps $V^{k+1}$ surjectively onto $V^k$, there exists $v\in V^{k+1}\setminus V^{k}$ with $\eta(v)=w$. We have
    \begin{align}\label{eq:annihilate}
        q^i\eta(\Frob-q^{n-k})^iv=(\Frob-q^{n-k+1})^i\eta v=0,
    \end{align}
    so $(\Frob-q^{n-k})^iv\in\ker\eta\cap J[\ell]=V^1$. Write $(\Frob-q^{n-k})^iv=u+s$ for $u\in U$ and $s\in W$. Since $W$ is preserved by $\Frob$ and does not contain any $\Frob$ eigenvectors of eigenvalue $q^{n-k}$, there exists $z\in W$ such that $(\Frob-q^{n-k})^iz=s$. Therefore
    \begin{align}\label{eq:Frobremainder}
        (\Frob-q^{n-k})^i(v-z)=u,
    \end{align}
    so $(\Frob-q^{n-k})^{i+1}(v-z)=0$. This proves $v-z\in F_n^{k+1}$ and $\eta(v-z)=w$, so (b) holds.
\end{proof}

\begin{corollary}\label{cor:eigenvalues}
    Let $n\in\Z/\gamma\Z$ and $1\leq k\leq \ell-1$. Then $F_n^k$ is nonempty if and only if $\gamma\mid dn$ and $\gamma\nmid n$.
\end{corollary}
\begin{proof}
    By \cref{Fnk_updown}, $F_n^k$ is nonempty if and only if $F_n^1$ is nonempty, which holds if and only if $q^n$ is an eigenvalue of $\Frob$ on $V^1$. By \cref{V0basis}, this holds if and only if $(q^n)^d=1$ and $q^n\neq 1$ in $\Fbar_\ell$.
\end{proof}

\begin{remark}\label{rmk:full-basis}
    One can generalize the above discussion to determine a basis for $J[\ell]\otimes\Fbar_\ell$ consisting of generalized $\Frob$ eigenvectors. More precisely, for all $1\leq k\leq \ell-1$ and all $\beta\in\Fbar_\ell$ with $\beta^d=1\neq\beta$, there exists $v_\beta^k\in J[\ell]\otimes\Fbar_\ell$ satisfying the following conditions:
    \begin{itemize}
        \item $v_\beta^k\in (V^k\otimes\Fbar_\ell)\setminus (V^{k-1}\otimes\Fbar_\ell)$.
        \item For $k\geq 2$, $\eta(v_\beta^k)=v_\beta^{k-1}$.
        \item $v_\beta^k$ is a generalized $\Frob$ eigenvector with eigenvalue $\beta q^{1-k}$.
    \end{itemize}
    Further, any set $\{v_\beta^k\}$ satisfying the above conditions is a basis for $J[\ell]\otimes\Fbar_\ell$. We can use this to explicitly determine all the diagonal entries of the Jordan canonical form of $\Frob$ acting on $J[\ell]\otimes\Fbar_\ell$, and hence compute the characteristic polynomial of $\Frob$ acting on $J[\ell]$. Thus the only real obstacle remaining is the failure of $\Frob$ to be diagonalizable over $\overline{\F}_\ell$: the rest of this paper can be thought of as a study of the Jordan blocks in the Jordan canonical form of $\Frob$.

    Since our interest lies with $J[\ell](\F_q)$, we will typically restrict our attention to the subspace of $J[\ell]$ generated by generalized $\Frob$ eigenvectors with eigenvalues equal to a power of $q$, that is, the $\Frob$-invariant and $\eta$-invariant subspace spanned by all the sets $F_n^k$. 
\end{remark}

\subsection{Basic counts and lifting results}\label{sec:proofs_of_basic_lifting}

The primary goal of this section is to determine, given some $F_n^k$ containing a $\Frob$ eigenvector, whether $F_n^{k+1}$ also contains a $\Frob$ eigenvector; that is, whether the property of containing a $\Frob$ eigenvector ``lifts'' from $F_n^k$ to $F_n^{k+1}$. Having an understanding of when this lifting occurs will help us to compute $r_\ell(C)$ because of \cref{thm:rlC_count}, which we recall and prove below.

\rlCcount*

\begin{proof}
    For each $1\leq k\leq \ell-1$ such that $F_{k-1}^k$ contains a $\Frob$ eigenvector, let $v_k$ denote such an eigenvector. We claim that the set of all such $v_k$ is a basis for $J[\ell](\F_q)$. First observe that by definition of $F_{k-1}^k$, $v_k$ has $\Frob$ eigenvalue $1$, so $v_k\in J[\ell](\F_q)$. Further, the set of all $v_k$ is linearly independent because each lies in a distinct filtration stage. So it just remains to prove that these vectors span $J[\ell](\F_q)$.
    
    We will prove by induction on $k$ that if $v\in J[\ell](\F_q)\cap V^k$ then $v$ is in the span of the eigenvectors $v_i$ with $i\leq k$. If $k=1$, then we must have $v=0$, as there is no eigenvector of eigenvalue $1$ in $V^1$ by \cref{V0basis}. Now let $k\geq 2$. If $v\in V^{k-1}$ then the result follows by the induction hypothesis, so we can assume $v\in V^k\setminus V^{k-1}$. This means that $v\in F_{k-1}^k$ is an eigenvector, and so $v_k$ must be defined. Now $v$ and $v_k$ define nonzero elements of the quotient $V^k/V^{k-1}$, and the $1$-eigenspace of $\Frob$ in this quotient is at most one-dimensional by \cref{subquotient_iso} and \cref{lem:linearcommute}. So for some $c\in\F_\ell$ we have $v-cv_k\in V^{k-1}$. By the induction hypothesis, $v-cv_k$ is a linear combination of $v_i$ for $i<k$, so $v$ is a linear combination of $v_i$ for $i\leq k$.
\end{proof}

In the remainder of this section we will prove \cref{basic_lifting_properties}. We first note the following important fact about the generalized $\Frob$ eigenvector lifts defined in \cref{Fnk_updown}(b). 

\begin{lemma}\label{Frob_remainder}
    Let $n\in\Z/\gamma\Z$, $1\leq k\leq \ell-2$, and suppose $w\in F_n^k$ is a $\Frob$ eigenvector. Exactly one of the following holds:
    \begin{itemize}
        \item For all $v\in F_n^{k+1}$ with $\eta v=w$, $v$ is an eigenvector.
        \item For all $v\in F_n^{k+1}$ with $\eta v=w$, $(\Frob-q^{n-k})v\in F_{n-k}^1$ is an eigenvector.
    \end{itemize}
\end{lemma}
\begin{proof}
    We first make the following observation: if $u\in J[\ell]$ satisfies $\eta u=0$, and $u$ is in the generalized $\Frob$ eigenspace with eigenvalue $q^{n-k}$, then either $u=0$ or $u\in F_{n-k}^1$. If $v,v'\in F_n^{k+1}$ with $\eta v=\eta v'=w$, we apply this observation to $v-v'$ to conclude that the value of $(\Frob-q^{n-k})v$ does not depend on the choice of $v$. Since $w$ is an eigenvector we have
    \[q\eta(\Frob-q^{n-k})v=(\Frob-q^{n-k+1})\eta v=0,\]
    so we apply the same observation to $(\Frob-q^{n-k})v$ to reach the desired conclusion.
\end{proof}

Recall that $F_n^k$ is a \emph{rooftop} if $F_n^k$ has a $\Frob$ eigenvector but there is no $k<k'\leq \ell-1$ for which $F_n^k$ has a $\Frob$ eigenvector, and that $F_n^k$ is a \emph{non-maximal rooftop} if $k\neq\ell-1$. The non-maximal rooftops are exactly the sets $F_n^k$ where the property of having a $\Frob$ eigenvector fails to lift to $F_n^{k+1}$.

\basicliftingproperties*

\begin{remark}
    Following from \cref{rmk: visual guide}, we give a brief visual explanation of each of these conditions. We shade a cell dark gray if $F_n^k$ contains a true $\Frob$ eigenvector. (a) says that in the bottom layer $k=1$, if a cell is shaded at all then it is shaded dark gray. (b) says that the dark gray cells form ``towers:'' any cell below a dark gray cell must also be dark gray. (c) and (d) both place limitations on which cells can contain the top cells of towers (i.e.~the rooftops). (c) says that if a diagonal intersects the $k=1$ layer in an empty cell, then the diagonal below it cannot contain any non-maximal rooftops. (d) says that no diagonal can contain two non-maximal rooftops. These constraints place some limitations on the possible ``skylines'' that can occur as in \cref{fig:lifting_charts}.
\end{remark}

\begin{proof}[Proof of \cref{basic_lifting_properties}]
\begin{enumerate}[label=(\alph*)]
    \item It follows from \cref{cor:eigenvalues} and that all elements of $F_n^1$ are eigenvectors.
    \item If $v$ is an eigenvector in $F_n^{k}$, then $\eta^{k-k'}v\in F_n^{k'}$ by \cref{Fnk_updown}, and this is an eigenvector by \cref{lem:linearcommute}.
    \item If $F_n^{k}$ is a maximal rooftop ($k=\ell-1$), then $\gamma\mid k$. Now $F_n^1$ has an eigenvector by (b) and so $\gamma\mid dn$ and $\gamma\nmid n$ by (a); this implies $\gamma\mid d(n-k)$ and $\gamma\nmid (n-k)$. So again by (a), $F_{n-k}^1$ has an eigenvector.

    Now suppose $F_n^k$ is a non-maximal rooftop, so $1\leq k\leq \ell-2$. By \cref{Frob_remainder}, $F_{n-k}^1$ is nonempty, so by (a), $\gamma\mid d(n-k)$ and $\gamma\nmid (n-k)$.
    \item  Suppose $F_n^k$ and $F_{n+i}^{k+i}$ are both non-maximal rooftops, where $k<k+i\leq \ell-2$ (the case $1\leq k+i<k$ follows by symmetry). By \cref{Frob_remainder}, there exist $v\in V^{k+i+1}\setminus V^{k+i}$ and $w\in V^{k+1}\setminus V^{k}$ for which both $v$ and $w$ map under $(\Frob-q^{n-k})$ to $F_{n-k}^1$, the set of nonzero vectors in a one-dimensional eigenspace. In particular, there exists $d\in\F_\ell^\times$ such that
    \[(\Frob -q^{n-k})v=d(\Frob -q^{n-k})w.\]
    Then $v-dw\in V^{k+i+1}\setminus V^{k+i}$ is an eigenvector of eigenvalue $q^{n-k}$, contradicting the assumption that $F_{n+i}^{k+i}$ is a rooftop. \qedhere
\end{enumerate}
\end{proof}

The remaining proofs require more setup. \Cref{Frob_remainder} tells us that the obstruction to lifting an eigenvector in $F_n^{k}$ to an eigenvector in $F_n^{k+1}$ is given by an element of $F_{n-k}^1$. Our next goal is to establish relations between these obstructions for different values of $n$ (with the same $k$).

\section{The Weil pairing}\label{sec:lifting_relations_new}
In this section, we will use the Weil pairing to prove \cref{thm:rooftop_pairs} and \cref{thm:evenselfdual}. These will then be used in \cref{sec:consequences} to prove \cref{intro_upper_bound}.

The \emph{Weil pairing} is a non-degenerate alternating $\Gal(\overline{\F}_q/\F_q)$-equivariant bilinear form
\begin{align*}
    e:J[\ell]\times J[\ell]\to \mu_\ell
\end{align*}
with the property that
\[e(f(u),f(v))=e(u,v)^{\deg f}\]
for any endomorphism $f:J\to J$ and $u,v\in J[\ell]$.

For the purposes of this paper we will only need the existence of a pairing satisfying the properties listed above; in particular we will never need to compute the pairing explicitly.
For the definition of the Weil pairing and proofs of the stated properties, see for example~\cite[Section 16]{Milne1986}. Note that what we call $e$ is obtained by taking what Milne calls $e_\ell^{\lambda}$ (with $\lambda$ the canonical principal polarization of $J$) and restricting to $J[\ell]\times J[\ell]$. The alternating property follows from Milne's Lemma 16.2(e), and the endomorphism property follows from Lemma 16.2(c).

In the remainder of this section we will prove that the Weil pairing interacts with the $(1-\zeta)$ filtration and with Frobenius in a compatible way.

\begin{lemma}\label{lem:weilproperties}
    Let $1\leq k,k'\leq \ell-1$, and let $v\in V^k$ and $v'\in V^{k'}$.
    \begin{enumerate}[label=(\alph*)]
        \item If $k+k'\le \ell-1$, then $e(v,v')=1$.
        \item If $k+k'\leq\ell+1$, then 
        \[e(\eta(v),v')=e(v,\eta(v'))^{-1}.\]
        \item If $k+k'=\ell$, and $v\in F_n^k$ and $v'\in F_{n'}^{k'}$ for some $n,n'\in\Z/\gamma\Z$, then $e(v,v')\neq 1$ if and only if $n+n'=0$.
    \end{enumerate}
\end{lemma}

In light of \cref{lem:weilproperties}(c), we make the following definition.

\begin{definition}
    We say the sets $F_n^k$ and $F_{-n}^{\ell-k}$ are \emph{dual}. If $v\in F_n^k$ and $v'\in F_{-n}^{\ell-k}$, then we say that $(v,v')$ form a \emph{dual pair}. See \cref{fig:duallemma}.
\end{definition}

In terms of the visual interpretation as described in \cref{rmk: visual guide}, each dual pair $v\in F_n^k$ and $v'\in F_{-n}^{\ell-k}$ corresponds to cells $(n,k)$ and $(-n,\ell-k)$ that are related by rotating the grid $180^\circ$. Point (b) relates the Weil pairing of vectors at cells $(n,k-1)$ and $(n',k')$ to the Weil pairing of vectors at cells $(n,k)$ and $(n',k'-1)$, moving one cell up and the other one down. This mirroring effect of the Weil pairing will play an important role in what follows.

\begin{proof}[Proof of \cref{lem:weilproperties}]
    All three statements depend on the following calculation. Let $u,v\in J[\ell]$. Since $\zeta$ is an automorphism of $J$, we have $e(\zeta u,v)=e(u,\zeta^{-1}v)$. Since $\zeta^{-1}=\zeta^{\ell-1}$ and the Weil pairing is bilinear, we have
\begin{align}\label{eq:lambda_weil}
        e((1-\zeta)u,v)=e(u,(1-\zeta^{\ell-1})v)=e(u,(\zeta^{\ell-2}+\cdots+\zeta+1)(1-\zeta)v).
    \end{align}
    Further, since $e$ is alternating, the same relation holds if we swap the entries on both sides.

    We begin by proving (a). If $k+k'\le \ell-1$, then $v'=(1-\zeta)^{k}w'$ for some $w'\in V^{k'+k}$. Then
    \[
        e(v,v')=e((\zeta^{\ell-2}+\cdots+\zeta+1)^{k}(1-\zeta)^{k}v,w')=e(0,w')=1,\]
    because $(1-\zeta)^{k}$ annihilates $V^k$. 
    
    We now prove (b). By definition of $\eta$ (\cref{def:eta}), we can write $\eta(v)=(\zeta-1)v+w$ for some $w\in V^{k-2}$. So
    \[e(\eta(v),v')=e((\zeta-1)v,v')e(w,v')=e((1-\zeta)v,v')^{-1},\]
    since $e(w,v')=1$ by part (a). By a symmetric argument we have $e(v,\eta(v'))=e(v,(1-\zeta)v')^{-1}$, so it suffices to show that
    \[e((1-\zeta)v,v')=e(v,(1-\zeta)v')^{-1}.\]
    Now note that $\zeta^{\ell-2}+\cdots +\zeta+1$ can be written as an integer polynomial in $1-\zeta$ with constant term $\ell-1$. So applying \cref{eq:lambda_weil},
    \[e((1-\zeta)v,v')=e(v,(\zeta^{\ell-2}+\cdots+\zeta+1)(1-\zeta)v')=e(v,(\ell-1)(1-\zeta)v'+w')\]
    for some $w'\in V^{k'-2}$; again by part (a), $e(v,w')=1$. Therefore
    \[e((1-\zeta)v,v')=e(v,(1-\zeta)v')^{\ell-1}=e(v,(1-\zeta)v')^{-1}.\]

    In order to deduce (c) we will prove a more general result. As in the proposition statement, let $k+k'=\ell$, and $v\in F_n^k$, meaning $v \in V^k\setminus V^{k-1}$ and $\Frob\  v= q^{n-(k-1)}v$ as an element of the quotient $V^k/V^{k-1}$. But now let $\Vbar^{k'}:=V^{k'}\otimes\Fbar_\ell$ (and similarly for $\Vbar^{k'-1}$), and let $u\in \Vbar^{k'}\setminus \Vbar^{k'-1}$ be any vector which reduces to an eigenvector of $\Frob$ in the quotient $\Vbar^{k'}/\Vbar^{k'-1}$. In particular, the eigenvalue $\beta\in\overline{\F}_\ell$ of $u$ does not a priori need to be a power of $q$. The Weil pairing extends in a natural way to $\overline{J[\ell]}:=J[\ell]\otimes\Fbar_\ell$, and under these weaker assumptions we will show that $e(v,u)\neq 1$ if and only if $u\in F_{-n}^{k'}$.
    
    Since $\Frob$ is an endomorphism of degree $q$, we have the following for any $i\geq 0$:
    \begin{align*}
        e(v,u)^{q^i}&=e(\Frob^i v, \Frob^i u)=e(q^{(n-(k-1))i} v+w, \beta^i u+w')
    \end{align*}
    for some $w\in V^{k-1}$ and $w'\in V^{k'-1}$. By part (a), we can eliminate $w$ and $w'$. Thus
    \begin{align*}
        e(v,u)^{q^i}&=e(q^{(n-(k-1))i} v, \beta^i u)=e(v,\beta^i u)^{q^{(n-(k-1))i}}
    \end{align*}
    so solving for $e(v,\beta^iu)$ we find
    \begin{align*}
        e(v, \beta^i u)=e(v,u)^{q^{(k-n)i}}.
    \end{align*}
    Writing the minimal polynomial of $\beta$ as 
    \[h(x)=a_dx^d+\cdots +a_1x+a_0\in \F_\ell[x],\]
    we have
    \begin{align*}
        1&=e(v,h(\beta) u)=\prod_{i=0}^d e(v,\beta^i u)^{a_i}=e(v,u)^{h(q^{k-n})}.
    \end{align*}
    If $e(v,u)\neq 1$, we must have $h(q^{k-n})=0$ in $\F_\ell$. Since $h(x)$ is irreducible but has a root in $\F_\ell$, it must be linear; hence $\beta= q^{k-n}=q^{(1-\ell)+k-n}$, showing that in fact $u\in F_{-n}^{\ell-k}$.

    Conversely, assume $v\in F_n^k$ and $v'\in F_{-n}^{\ell-k}$. First consider the case $k=1$. By (a), we know $e(v,w)=1$ for all $w\in V^{\ell-2}$. By the proof of the reverse direction above, we know that for any $u\in \overline{J[\ell]}$, if $u$ reduces to a $\Frob$ eigenvector in $\overline{J[\ell]}/\Vbar^{\ell-2}$ with any eigenvalue other than $q^{k-n}$, then $e(v,u)=1$. We also know that $\overline{J[\ell]}/V^{\ell-2}$ is spanned by its eigenspaces, which are all one-dimensional. So if $e(v,v')=1$, then we would have $e(v,w)=1$ for all $w$ in a basis of $\overline{J[\ell]}$, contradicting non-degeneracy of $e$. Hence $e(v,v')\neq 1$.

    For $k\geq 1$, note that $\eta^{k-1}(v)\in F_n^1$, and we can write $v'=\eta^{k-1}(u')$ for some $u'\in F_{-n}^{\ell-1}$. So applying the $k=1$ case and part (b),
    \[e(v,v')=e(\eta^{k-1}(v),u')^{(-1)^{k-1}}\neq 1.\qedhere\]
\end{proof}

\subsection{Lifting relations with dual pairs}
For all $n\in\Z/\gamma\Z$ with $\gamma\mid dn$ and $\gamma\nmid n$, fix once and for all a $\Frob$ eigenvector $u_n\in F_n^1$ (which exists by \cref{basic_lifting_properties}(a)).
Now suppose that $F_n^k$ has an eigenvector for some $n\in\Z/\gamma\Z$ and $1\leq k\leq\ell-2$. By \cref{Frob_remainder}, there exists $v\in F_n^{k+1}$ mapping to this eigenvector under $\eta$, such that either $v$ is an eigenvector itself, or $(\Frob-q^{n-k})v\in F_{n-k}^1$ is an eigenvector. In other words, there exists a constant $c\in\F_\ell$ such that 
\begin{equation}\label{eq:vc}
    (\Frob-q^{n-k})v=cu_{n-k},
\end{equation}
and $v$ is an eigenvector if and only if $c=0$. In the same way, if $F_{k-n}^{k}$ has an eigenvector then there exists $v'\in F_{k-n}^{k+1}$ and $d\in\F_\ell$ such that 
\begin{equation}\label{eq:vpd}
    (\Frob-q^{-n}) v'=du_{-n}.
\end{equation}
The key observation behind the following lemma is that if we lift \cref{eq:vpd} along powers of $\eta$ until the preimage of $v'$ reaches the very top of the filtration, then the lifts of $u_{-n}$ and $v'$ form {dual pairs} with the vectors $v$ and $u_{n-k}$ appearing in \cref{eq:vc}. See \cref{fig:duallemma} for a summary of this setup following the visual interpretation laid out in \cref{rmk: visual guide}.

\begin{lemma}\label{lem:pairingrelation}
    Assume $F_n^{k}$ and $F_{k-n}^{k}$ both contain $\Frob$ eigenvectors, and let $v,c,v',d$ be as above. Then there exist $\widehat{v'}\in F_{k-n}^{\ell-1}$ and $\widehat{u_{-n}}\in F_{-n}^{\ell-1-k}$ such that $\eta^{\ell-2-k}(\widehat{v'})=v'$, $\eta^{\ell-2-k}(\widehat{u_{-n}})=u_{-n}$, and
    \[e(\widehat{v'},u_{n-k})^{cq^n}e(\widehat{u_{-n}},v)^{dq^{k-n}}=1.\]
\end{lemma}

\begin{figure}
    \centering
    \begin{tikzpicture}   
        \draw[thick] (0,0.6) -- (8,0.6) -- (8,5.9) -- (0,5.9) -- cycle;
        \node (v) at (2.5, 2.5) {$v$};
        \node (w) at (1, 1) {$u_{n-k}$};
        \node (what) at (5.5, 4) {$\widehat{u_{-n}}$};
        \node (vhat) at (7, 5.5) {$\widehat{v}'$};
        \node (wp) at (5.5, 1) {$u_{-n}$};
        \node (vp) at (7, 2.5) {$v'$};

        \node at (1,0.3) {$n-k$};
        \node at (2.5,0.3) {$n$};
        \node at (5.5,0.3) {$-n$};
        \node at (7,0.3) {$k-n$};
        \node[anchor = east] at (0, 1) {$1$};
        \node[anchor = east] at (0, 2.5) {$k+1$};
        \node[anchor = east] at (0, 4) {$\ell-1-k$};
        \node[anchor = east] at (0, 5.5) {$\ell-1$};

        \path[dashed] (v) edge [bend right=28] node {} (what);
        \path[dashed] (w) edge [bend left=28] node {} (vhat);
        \path[-Stealth, inner sep=5pt, anchor = 135] (v) edge node {} (w);
        \path[-Stealth, anchor = 135] (vp) edge node {} (wp);
        \path[-Stealth, anchor = 135] (vhat) edge node {} (vp);
        \path[-Stealth, anchor = 135] (what) edge node {} (wp);
        \node at (6.25, 3.25) {};
    \end{tikzpicture}
    \caption{An illustration of the setup (vectors $v,v'$) and conclusion ($\widehat{v'},\widehat{u_{-n}}$) of \cref{lem:pairingrelation}. Vertical arrows point from a vector to its image under $\eta^{\ell-2-k}$. A diagonal arrow from $w$ to $u$ means that $w$ is a generalized eigenvector for $\Frob$ with some eigenvalue $\beta$ and $(\Frob-\beta)w$ is in the span of $u$. Dashed curves connect dual pairs.}
    \label{fig:duallemma}
\end{figure}
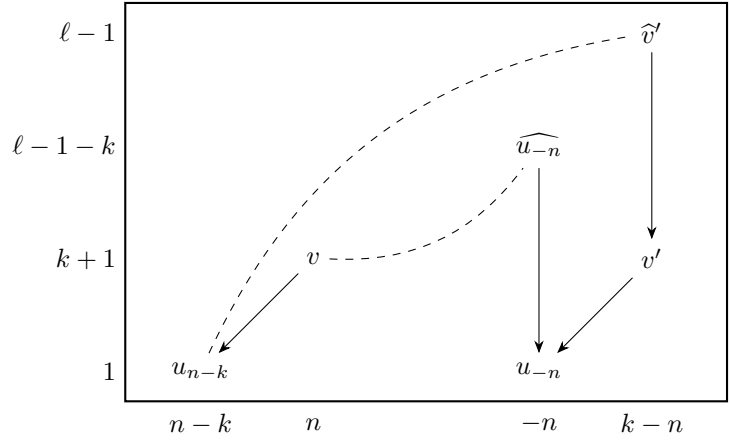

Since $(\widehat{v'},u_{n-k})$ and $(\widehat{u_{-n}},v)$ are dual pairs, their Weil pairings are both nontrivial by \cref{lem:weilproperties}(c). This will allow us to make conclusions about the constants $c$ and $d$. 

\begin{proof}
    By \cref{Fnk_updown}, there exist $\widehat{v'}\in F_{k-n}^{\ell-1}$ and $\widehat{u_{-n}}\in F_{-n}^{\ell-1-k}$ that are preimages of $v'$ and $u_{-n}$, respectively, under $\eta^{\ell-2-k}$. Then
    \[(\Frob-q^{k+2-\ell-n})\widehat{v'}= q^{k+2-\ell}d\widehat{u_{-n}}+w\]
    for some $w\in V^{\ell-2-k}$, which can be checked by showing that both sides have the same image under $\eta^{\ell-2-k}$. Therefore,
    \begin{align*}
        e(\widehat{v'},v)&=e(\widehat{v'},v)^{q^{\ell-1}}\\
        &=e(\Frob\widehat{v'},\Frob v)^{q^{\ell-2}}\\
        &=e(q^{k+2-\ell-n}\widehat{v'}+q^{k+2-\ell}d\widehat{u_{-n}}+w,\;q^{n-k} v+cu_{n-k})^{q^{\ell-2}}\\
        &=e(q^{k-n}\widehat{v'}+q^{k}d\widehat{u_{-n}}+q^{\ell-2}w,\;q^{n-k} v+cu_{n-k}).
    \end{align*}
    We now apply bilinearity. Since $v,u_{n-k}\in V^{k+1}$, we can use \cref{lem:weilproperties}(a) to eliminate all pairings involving $w$, as well as the pairing of $\widehat{u_{-n}}$ with $u_{n-k}$. We obtain
    \[e(\widehat{v'},v)=
    e(\widehat{v'},v)
    e(\widehat{v'},u_{n-k})^{cq^{k-n}}
    e(\widehat{u_{-n}},v)^{dq^{n}}\]
    which implies the desired result.
\end{proof}

With this we can now give proofs of \cref{thm:rooftop_pairs} and \cref{thm:evenselfdual}.

\rooftoppairs*

\begin{proof}
    We first prove by induction on $k$ that if $F_n^k$ is a non-maximal rooftop, then so is $F_{k-n}^k$. We must first show that $F_{k-n}^k$ contains a $\Frob$ eigenvector. Since $F_n^k$ is a rooftop, we have that $F_{k-n}^1$ is nonempty by \cref{basic_lifting_properties}(c). For the base case $k=1$ this already establishes that $F_{k-n}^k$ has a $\Frob$ eigenvector. Otherwise, for the sake of contradiction, suppose $F_{k-n}^{k'}$ is a rooftop for some $1\leq k'<k$. By the induction hypothesis, $F_{k'-(k-n)}^{k'}=F_{n-(k-k')}^{k-(k-k')}$ is also a rooftop. But since $F_n^k$ is a non-maximal rooftop, this contradicts \cref{basic_lifting_properties}(d). Hence $F_{k-n}^k$ has an eigenvector.
    
    Now since $F_n^k$ is a non-maximal rooftop, we can take $v\in F_n^{k+1}$ and $c\in\F_\ell$ as in \cref{eq:vc} with $c\neq 0$. If we assume for the sake of contradiction that $F_{k-n}^k$ is not a rooftop, then we can take $v'\in F_{k-n}^{k+1}$ and $d=0$ in \cref{eq:vpd}. By \cref{lem:pairingrelation} we have
    \[e(\widehat{v'},u_{n-k})^{cq^n}
    e(\widehat{u_{-n}},v)^{dq^{k-n}}=1,\]
    where $(\widehat{v'},w)$ and $(\widehat{w'},v)$ are dual pairs; in particular, we have $e(\widehat{v'},u_{n-k})\neq 1$ by \cref{lem:weilproperties}(c). Since $cq^n\not\equiv 0\bmod\ell$ and $d=0$ we obtain a contradiction. Hence $F_{k-n}^k$ must be a rooftop. This concludes the proof for $k\leq \ell-2$.
\end{proof}

\evenselfdual*

\begin{proof}
    Note that $n\equiv k-n\bmod\gamma$, so we can take $v=v'\in F_n^{k+1}$, and $c=d\in\F_\ell$ satisfying \cref{eq:vc} and \cref{eq:vpd}. By \cref{lem:pairingrelation}, there exist $\widehat{v}\in F_n^{\ell-1}$ and $\widehat{u_{n-k}}\in F_{-n}^{\ell-1-k}$ such that
    \[e(\widehat{v},u_{n-k})^{cq^n}e(\widehat{u_{n-k}},v)^{cq^{k-n}}=1,\]
    where $\eta^{\ell-2-k}(\widehat{v})=v$ and $\eta^{\ell-2-k}(\widehat{u_{n-k}})=u_{n-k}$. Hence, by \cref{lem:weilproperties}(b) and the fact that $e$ is alternating,
    \[e(\widehat{u_{n-k}},v)=e(u_{n-k},\widehat{v})^{(-1)^{\ell-2-k}}=e(\widehat{v},u_{n-k})^{(-1)^k}.\]
    Since $q^n\equiv q^{k-n}\bmod \ell$ we have
    \[e(\widehat{v},u_{n-k})^{cq^n(1+(-1)^k)}=1.\]
    Since $e(\widehat{v},w)\neq 1$ by \cref{lem:weilproperties}, and $k$ is even by assumption, this is only possible if $c=0$, so that $v\in F_n^{k+1}$ is an eigenvector.
\end{proof}

For the proofs of \cref{intro_upper_bound} and \cref{prop:2lifting}, the reader may skip ahead to \cref{sec:consequences}. The intervening sections on Galois representations and Galois cohomology are only required for the proof of \cref{intro_refined_upper_bound}, though they can also be used to provide an alternate cohomological interpretation of some of the preceding results.

\section{From Frobenius eigenvectors to Galois representations}\label{sec:galoisreps}

The primary goal of this section is to show that the existence of an eigenvector of $\Frob$ in $F_n^k$ is equivalent to the existence of a certain $k$-dimensional representation $\psi$ of the absolute Galois group
\[G_{\F_q(x)}\colon= \Gal(\F_q(x)^\sep/\F_q(x)),\]
where $\F_q(x)^\sep$ denotes the separable closure of $\F_q(x)$.
A precise statement is given in \cref{prop:eigen to rep}. In the following sections we use Galois cohomology to analyze conditions under which such a Galois representation can occur, with the goal of proving \cref{thm:congruence_conditions}.

\subsection{Automorphisms of function fields} \label{subsec:setup}

Our first step is to translate our setup to the function field setting. The function field associated to the curve $C$ is $\F_q(C)=\F_q(x,y)$, where $y$ satisfies the equation $y^\ell=f(x)$.Since $\ell$ is coprime to $q$, the extension $\F_q(C)/\F_q(x)$ is separable. The Galois closure of $\F_q(C)/\F_q(x)$ is the field $\F_{q^\gamma}(C)= \F_{q^\gamma}(x,y)$, since $\F_{q^\gamma}/\F_q$ is generated by $\ell^{\text{th}}$ roots of unity. We define two automorphisms of $\F_\qg(C)$ by their actions on $x$, on $y$, and on $c\in\F_\qg$:
\begin{align*}
     \xi:   \qquad x&\mapsto x, & y&\mapsto \zeta^{-1} y, & c&\mapsto c\text{ for }c\in\F_{q^\gamma},\\
     \frob:  \qquad x&\mapsto x, & y&\mapsto y, & c&\mapsto c^q\text{ for }c\in\F_{q^\gamma}.
\end{align*}
Both of these automorphisms preserve $\F_\qg(C)$ and fix $\F_q(x)$ and so define elements of $\Gal(\F_\qg(C)/\F_q(x))$. More precisely, $\xi$ is a generator of the subgroup $\Gal(\F_\qg(C)/\F_\qg(x))\simeq \mathbf{Z}/\ell\mathbf{Z}$, and $\frob$ maps to a generator of the quotient group $\Gal(\F_\qg(x)/\F_q(x))\simeq \mathbf{Z}/\gamma\mathbf{Z}$. As discussed in the following \Cref{lem:G}, we can pick a section $\Gal(\F_{q^\gamma}(x)/\F_q(x)) \to \Gal(\F_{q^\gamma}(C)/\F_q(x))$ and thus we will use $\frob$ to denote both the element in $\Gal(\F_\qg(C)/\F_q(x))$ and its image in $\Gal(\F_\qg(x)/\F_q(x))$.
In summary, we have the following structure.

\begin{lemma}\label{lem:G}
    The Galois group $G_C:=\Gal(\F_{q^\gamma}(C)/\F_q(x))$ is isomorphic to $(\mathbf{Z}/\ell\mathbf{Z})\rtimes (\mathbf{Z}/\gamma\mathbf{Z})$, satisfying the exact sequence
    \[ 1 \to \langle \xi\rangle \to G_C \to \langle \frob\rangle \to 1 \]
    and with the semi-direct product structure given by $\frob\circ \xi=\xi^q\circ \frob$.
\end{lemma}

The following diagram summarizes the fields we are considering so far. All pictured extensions are Galois except for $\F_q(C)/\F_q(x)$.
\begin{center}
\begin{tikzcd}
& \F_{q^\gamma}(C)\arrow[dl, dash] \arrow[dr, dash,"\langle \xi\rangle \simeq \mathbf{Z}/\ell\mathbf{Z}"]\arrow[dd, dash, "G_C"] & \\
\F_q(C) \arrow[dr, dash] & &\F_{\qg}(x) \arrow[dl,dash,"\langle \frob\rangle \simeq \mathbf{Z}/\gamma\mathbf{Z}"] \\
& \F_{q}(x) &
\end{tikzcd}
\end{center}

We briefly discuss how these function field maps relate to the morphisms $\zeta,\Frob:C_{\F_\qg}\to C_{\F_\qg}$ from \cref{sec:setup}. Let $\zeta^\sharp$ and $\Frob^\sharp$ denote the respective maps of function fields induced by the maps $\zeta$ and $\Frob$. Then we can immediately see that $\zeta^\sharp=\xi^{-1}$. On the other hand, the maps $\frob$ and $\Frob^\sharp$ are quite different. The \emph{arithmetic Frobenius} map $\frob$ does not fix the base field $\F_\qg$, and so is not induced by a morphism of $\F_\qg$-varieties. The \emph{relative Frobenius} map $\Frob^\sharp$ -- which is the map from $\F_\qg(C)$ to itself that sends $x\mapsto x^q$, $y\mapsto y^q$, and fixes $\F_\qg$ -- is induced by a morphism of $\F_\qg$-varieties, but is not a field automorphism; the image is a degree $q$ subfield of $\F_\qg(C)$. The composition $\Frob^\sharp\circ\frob$ is the \emph{absolute Frobenius} map, which acts on any $h\in \F_\qg(C)$ by $h\mapsto h^q$. See \cite[Page 93]{Poonen} for more on the decomposition of absolute Frobenius into relative Frobenius and arithmetic Frobenius. These observations give us the following relations for all geometric points $P\in C(\Fbar_q)$ and rational functions $h\in \F_\qg(C)$:
\[\xi(h)(\zeta(P))=h(P),\qquad \frob(h)(\Frob(P))=h(P)^q.\]
As a consequence, on the divisor classes we have
\begin{align}\label{eq:func field to scheme morphism}
    \divisor(\xi(h))=\zeta(\divisor(h)),\qquad \divisor(\frob(h))=\Frob(\divisor(h)).
\end{align}

\subsection{Constructing Galois representations}\label{subsec:defreps}

We first define a few explicit representations that will be used to identify $\Frob$ eigenvectors in $F_n^k$. Recall that $G_C=\Gal(\F_\qg(C)/\F_q(x))$ is generated by $\xi$ and $\frob$, and there is a natural quotient map $G_{\F_q(x)}\to G_C$.

\begin{definition}
    Let $\chi:G_{\F_q(x)}\to \F_\ell^*$ be the representation given by the action on $\mu_\ell \subset \F_\qg$. That is, $\chi$ factors through $G_C$ where it acts by
    \[ \chi(\frob)=q, \qquad \chi(\xi)=1.  \]
\end{definition}

\begin{definition}\label{defrho}
    Let $\rho:G_{\F_q(x)}\to \GL_2(\F_\ell)$ be the representation that factors through $G_C$ where it acts by
    \[\rho(\frob)=\begin{pmatrix} q & 0\\0&1\end{pmatrix},\qquad \rho(\xi)=\begin{pmatrix} 1 & 1\\0&1\end{pmatrix}\]
    for a choice of basis $\{u,v\}\subset \F_\ell^2$. 
\end{definition}
 
The span of $u$ is the unique one-dimensional subrepresentation of $\rho$, which is isomorphic to $\chi$ as can be seen by considering the upper left entry of the matrix form of $\rho$. Moreover, for any $\sigma\in G_{\F_q(x)}$, we have $\rho(\sigma) = \begin{pmatrix}
    \chi(\sigma) & b(\sigma) \\ 0  & 1
\end{pmatrix}$ for some $b(\sigma) \in \F_\ell$; in particular, $b$ is a crossed homomorphism $G_{\F_q(x)} \to \F_\ell$ representing a class in $H^1(G_{\F_q(x)}, \chi)$. Using the cohomological setup we will introduce in \cref{subsec:Cohomologyclasses}, this is equivalent to saying that $\rho$ is the extension of $1$ by $\chi$ associated to the function $b\in H^1(G_{\F_q(x)}, \chi)$.

For $0\leq k\leq \ell-1$, the \emph{$k$-th symmetric power} of $\rho$, $\Sym^k\rho:G_{\F_q(x)}\to\GL_{k+1}(\F_\ell)$, is defined as follows. Taking the basis $\{u,v\}$ for $\F_\ell^2$ as in \cref{defrho}, a basis for $\F_\ell^{k+1}$ is given by formal monomials $e_i:=\frac{1}{i!}u^{k-i}v^i$ for $i=0,\ldots,k$ (note that $i!$ is invertible over $\F_\ell$). For each $\sigma \in G_{\F_q(x)}$, define
\begin{align*}
    \Sym^k\rho(\sigma)(e_i)=\frac{1}{i!}(\rho(\sigma)(u))^i(\rho(\sigma)(v))^j,
\end{align*}
computed by expanding the right-hand side as a sum of monomials. Then $\Sym^k\rho$ factors through $G_C$, and has the following matrix representation on the generators of $G_C$ with respect to the basis $e_0,\ldots,e_k$.
    \begin{equation}\label{eq:SymFrobzeta}
\Sym^k\rho (\frob)=\begin{pmatrix}
    q^k & 0 & 0 & 
   \cdots & 0 \\ 
     & q^{k-1} & 0 & \cdots 
     & 0 \\ 
     &   & q^{k-2} &  \cdots & 
      0 \\ 
     &  &  & \ddots & \vdots   \\ 
     &  &  & & 1 
    \end{pmatrix}, \qquad\qquad \Sym^k\rho (\xi)=\begin{pmatrix}
    1 & 1 & \frac{1}{2!} & 
   \cdots & \frac{1}{k!} \\ 
     & 1 & 1 & \cdots 
     & \frac{1}{(k-1)!} \\ 
     &   & 1 &  \cdots & 
      \frac{1}{(k-2)!} \\ 
     &  &  & \ddots & \vdots   \\ 
     &  &  & & 1 
    \end{pmatrix}.
    \end{equation}
The matrix $\Sym^k\rho(\xi)$ is similar to a Jordan block, so there is no nontrivial decomposition of $\F_\ell^{k+1}$ into a direct sum of two subspaces invariant under $\Sym^k\rho(\xi)$. We can conclude that $\Sym^k\rho$ is indecomposable for $0\leq k\leq \ell-1$. 

\begin{lemma}\label{lem:extend}
    For any $n\in\Z/\gamma\Z$ and $2\leq k\leq \ell-1$, there is a short exact sequence of $G_{\F_q(x)}$-representations
    \begin{equation}\label{eq:reps}
        0 \to \Sym^{k-1}\rho \otimes \chi^{n+1} \to \Sym^k\rho \otimes \chi^n \to \chi^n \to 0.
    \end{equation}
\end{lemma}

\begin{proof}
   The matrix representation for $\bigrep{k}{n}$ is given by the matrix representation for $\Sym^\rho$ with every entry multiplied by $\chi^n(\sigma)$. The span of $e_0,\ldots,e_{k-1}$ is a $k$-dimensional invariant subspace. The action of $\bigrep{k}{n}$ on this subspace given by the $k\times k$ upper left submatrix, giving an explicit isomorphism to $\Sym^{k-1}\rho \otimes \chi^{n+1}$. The corresponding quotient representation is given by the entry at the bottom-right, which is isomorphic to $\chi^n$.
\end{proof}

Under the same assumptions as \cref{lem:extend}, we also have a short exact sequence of the form
\begin{equation}\label{eq:topleft cyclo}
0 \to \cyclo^{k+n} \to \bigrep{k}{n} \to \bigrep{k-1}{n} \to 0.
\end{equation}

We will also compute the dual representations of the representations defined above.
Let $V$ be a vector space over $\F_\ell$, and $V^\vee:=\Hom (V,\F_\ell)$ its dual space, so that there is a natural perfect pairing $V\times V^\vee\to\F_\ell$.

\begin{definition}\label{def:dual}
    The \emph{linear dual} (or \emph{contragredient}) of a representation $\theta: G_{\F_q(x)} \to \GL(V)$ is a representation $\theta^\vee:G_{\F_q(x)}\to\GL(V^\vee)$ characterized by the condition that $V\times V^\vee\to \F_\ell$ induces a $G_{\F_q(x)}$-equivariant homomorphism $\theta\times \theta^\vee\to 1$.
    
    The \emph{cohomological dual} of a Galois representation $\theta: G_{\F_q(x)} \to \GL(V)$ is defined as
    \[\theta^*:=\Hom_{G_{\F_q(x)}}(\theta, \mu_\ell)\simeq \theta^\vee\otimes\chi.\]
\end{definition}
If we identify the vector space $\F_\ell^n$ with its dual via the pairing $(u,v)\mapsto u^Tv$, we can check that for each $\sigma\in G_{\F_q(x)}$, the matrix representing $\theta^\vee(\sigma)$ must be the transpose of the matrix representing $\theta(\sigma)^{-1}$, in order to guarantee that $(\theta^\vee(\sigma) u)^T(\theta(\sigma) v)=u^T v$ for all $u,v\in \F_\ell^n$. 
\begin{lemma}\label{bigrep_dual}
    Let $0\leq k\leq \ell-1$ and $n\in\Z/\gamma\Z$. The linear dual of $\bigrep{k}{n}$ is isomorphic to $\bigrep{k}{-k-n}$.
\end{lemma}

\begin{proof} 
    Since $\bigrep{k}{n}$ factors through $G_C$, the same must be true of its linear dual, so we can restrict our attention to $G_C$. Let $\theta=\bigrep{k}{-k-n}$. Using the change of basis determined by the matrix
    \begin{align}\label{eq:special_diag}
        B=\begin{psmallmatrix}
         &  & & & -1 \\
         &  & & 1 & \\
         &  & -1 &  & \\
         & \iddots  & &  & \\
        (-1)^k &   & &  & \\
        \end{psmallmatrix},
    \end{align}
    we compute
    \begin{align*}
    (B\theta(\frob)B^{-1})^T=\begin{pmatrix}
    q^{-n-k} &  & 
    &  \\ 
     & q^{1-n-k} &  
     &  \\ 
     &  &  \ddots &    \\ 
     &  &  & q^{-n} 
    \end{pmatrix}, \qquad
    (B\theta(\xi)B^{-1})^T=
    \begin{pmatrix}
    1 & -1 & \frac{1}{2!} & 
   \cdots & \frac{(-1)^k}{k!} \\ 
     & 1 & -1 & \cdots 
     & \frac{(-1)^{k-1}}{(k-1)!} \\ 
     &   & 1 &  \cdots & 
      \frac{(-1)^{k-2}}{(k-2)!} \\ 
     &  &  & \ddots & \vdots   \\ 
     &  &  & & 1 
    \end{pmatrix}.
    \end{align*}
    These are the inverses of $(\bigrep{k}{n})(\frob)$ and $(\bigrep{k}{n})(\xi)$, respectively. The former assertion is clear, and the latter uses the observation that for all $0\leq i\leq j\leq k$,
    \begin{align*}
        \sum_{t=i}^j\frac{(-1)^{j-t}}{(t-i)!(j-t)!}=\frac{1}{(j-i)!}\sum_{t=0}^{j-i}(-1)^\ell\binom{j-i}{t}=\left\{\begin{array}{ll}
            1 & \text{if }j=i, \\
            0 & \text{otherwise.}
        \end{array}\right.
    \end{align*}
    We can conclude that the change of basis determined by $B$ takes $\bigrep{k}{-n-k}$ to the linear dual of $\bigrep{k}{n}$. 
\end{proof}

\subsection{$J[\ell](\F_\qg)$ as a Galois representation}\label{subsec:Kummer}

Note that any $\Frob$ eigenvector $v\in F_n^k$ is automatically in $J[\ell](\F_{\qg})$ since $q^\gamma = 1 \in \F_\ell^*$ and
\[\Frob^\gamma v=q^{(n-k+1)\gamma}v=(q^\gamma)^{n-k+1} v=v.\] We therefore restrict our attention to the structure of $J[\ell](\F_\qg)$ as a $G_{\F_q(x)}$ Galois module.

We will define an action of $G_{\F_q(x)}$ on $J[\ell](\F_\qg)$ via Kummer theory. Let $H$ denote the subgroup of $\F_\qg(C)^\times/\F_\qg(C)^{\times \ell}$ defined by the property that $h\in H$ if and only if $\divisor(h)$ is a multiple of $\ell$ in $\Div(C_{\F_\qg})$.

\begin{lemma}\label{lem:H to Jl}
    There is a split exact sequence
    \[0\to \mu_\ell \to H\to J[\ell](\F_\qg)\to 0.\]
\end{lemma}
\begin{proof}
    Define the map $H\to J[\ell](\F_\qg)$ by $h\mapsto \frac1\ell\divisor(h)$. If $h$ is in the kernel of this map then there exists $g\in \F_\qg(C)^\times$ with $\ell\divisor(g)=\divisor(h)$, so $h=cg^\ell$ for some constant $c\in\F_\qg^\times$. This implies that $h=c$ as elements of the quotient group $\F_\qg(C)^\times/\F_\qg(C)^{\times\ell}$. Since $q^\gamma\equiv 1\bmod \ell$, we have $\F_\qg^\times/\F_\qg^{\times\ell}\simeq \mu_\ell$.
    
    It suffices to define a right inverse for the map $h\mapsto \frac1\ell\divisor(h)$. 
    Pick an arbitrary place $Q$ of $\F_\qg(C)$ with corresponding uniformizer $\pi_Q$ and valuation $\ord_Q$, and define the set of ``monic'' rational functions
    \[K^1:=\left\{h\in \F_\qg(C)^\times:\left(\frac{h}{\pi_Q^{\ord_Q(h)}}\right)(Q)=1\right\}.\]
    For every $h \in \F_\qg(C)^\times$, there exists $\alpha \in \F_\qg^\times$ such that $\alpha h \in K^1$.
    Then $K^1$ is a subgroup of $\F_\qg(C)^\times$, and we have an exact sequence
    \[0\to K^1\xrightarrow{\divisor} \Div(C_{\F_\qg})\to J(\F_{q^\gamma})\to 0.\]
    Now for any $D\in J[\ell](\F_\qg)$, there exists $h_D\in K^1$ satisfying $ \divisor(h_D) = \ell D$. Thus $h_D\in H$ maps to $D$.
\end{proof}

The action of $G_{\F_q(x)}$ on $\F_q(x)^\sep$ induces actions on $\mu_\ell$ and $H$, and the map $\mu_\ell\to H$ from \cref{lem:H to Jl} is equivariant under these actions. This induces a quotient representation structure on $J[\ell](\F_\qg)$.
Explicitly, for each $D\in J[\ell](\F_\qg)$, pick a preimage $h_D\in H$. By \cref{eq:func field to scheme morphism}, we have $\xi(h_D)=h_{\zeta(D)}$ and $\frob(h_D)=h_{\Frob(D)}$ as elements of $H/\mu_\ell$. Thus the structure of $J[\ell](\F_\qg)$ as a $G_{\F_q(x)}$-representation is determined by stipulating that the map $G_{\F_q(x)}\to \GL(J[\ell](\F_\qg))$ factors through $G_C$, where it acts by
\begin{align*}
    \frob\cdot D&=\Frob(D),\\
    \xi\cdot D&=\zeta(D).
\end{align*}

Using this setup, we can show that each $\Frob$ eigenvector in $F_n^k$ generates a $G_{\F_q(x)}$-subrepresentation of $J[\ell](\F_\qg)$ that is isomorphic to one of the representations constructed in the previous section.

\begin{lemma}\label{lem:GaloisrepW}
    There exists an eigenvector of $\Frob$ in $F_n^k$ if and only if there exists a $G_{\F_q(x)}$-subrepresentation of $J[\ell](\F_\qg)$ isomorphic to $\Sym^{k-1}\rho\otimes\chi^{n+1-k}$.
\end{lemma}

\begin{proof}
    As was discussed above, the structure of $J[\ell](\F_\qg)$ as a $G_{\F_q(x)}$-representation is determined by the fact that the representation factors through $G_C$, where $\frob$ acts by $\Frob$ and $\xi$ acts by $\zeta$. So to prove the lemma it suffices to consider the actions of $\Frob$ and $\zeta$.
    
    Suppose $v\in F_n^k$ is a $\Frob$ eigenvector.
    Since $v \in V_k\setminus V_{k-1}$, the vectors $\eta^{k-1}v,\eta^{k-2}v,\cdots,v$ are linearly independent and span a $k$-dimensional $\F_\ell$ vector space which we call $W$. By \cref{lem:linearcommute}, the matrix representing $\Frob$ acting on the basis $\{\eta^{k-1}v,\eta^{k-2}v,\cdots,v\}$ is
   \begin{equation} \label{eq:betavFrob}
    \begin{pmatrix}
        q^{n} &   &  &  \\
         & q^{n-1} &   & \\
         &  &\ddots & \\
         &  &  & q^{n-k+1}
    \end{pmatrix}.
    \end{equation}
    Since $\eta^k v=0$, $W$ is also stable under the action of $\eta$. By \cref{def:eta}, the actions of $\eta$ and $\zeta$ satisfy the equation $$\zeta = 1 + \eta + \frac{\eta^2}{2!}+\cdots+\frac{\eta^{k-1}}{(k-1)!}$$
    on $W$. The matrix representing $\zeta$ acting on the basis $\{\eta^{k-1}v,\eta^{k-2}v,\cdots,v\}$ is therefore
    \begin{equation}\label{eq:betavzeta}
        \begin{pmatrix}1 & 1 &  \frac{1}{2!} & \cdots & \frac{1}{(k-1)!}\\
         & 1 & 1 & \cdots & \frac{1}{(k-2)!}\\
         &  & 1 & \cdots & \frac{1}{(k-3)!}\\
         &  &  & \ddots & \vdots\\
         & &  &  & 1
            
        \end{pmatrix}. 
    \end{equation}

    Comparing the two matrices with Equations \cref{eq:SymFrobzeta}, we conclude that the $G_{\F_q(x)}$ action on $W$ (with $\frob$ acting via $\Frob$ and $\xi$ acting via $\zeta$) is isomorphic to $\Sym^{k-1}\rho\otimes\chi^{n-k+1}$.

    Conversely, suppose $W$ is a subgroup of $J[\ell](\F_\qg)$ that is isomorphic as a $G_{\F_q(x)}$-representation to $\Sym^{k-1}\rho\otimes\chi^{n-k+1}$. Let $v\in W$ correspond to the vector $(0,\ldots,0,1)$. Then we have $\Frob v=q^{n-k+1}v$, $(\zeta-1)^{k-1}v\neq 0$, and $(\zeta-1)^kv=0$, establishing that $v$ is a $\Frob$ eigenvector in $F_n^k$ as desired.
\end{proof}

\subsection{Galois representations from unramified extensions}
Recall that $H\leq \F_\qg(C)^\times/\F_\qg(C)^{\times\ell}$ is defined by the property that $h\in H$ if $\divisor(h)\in \ell \ \Div(C_{\F_\qg})$. By \cref{lem:H to Jl} $H$ is a finite group.
Given a subgroup $\Gamma\leq H$, we define
\[K_\Gamma:=\F_{q^\gamma}(C)(\sqrt[\ell]{h}:h\in \Gamma).\]
\begin{lemma}\label{max_unram}
    The maximal unramified elementary $\ell$-extension of $\Kprime$ is $K_H$.
\end{lemma}
\begin{proof}
    By Kummer theory, subgroups of $\Kprime^\times/\Kprime^{\times\ell}$ correspond bijectively with abelian extensions of $\Kprime$ of exponent $\ell$ by taking $\ell$-th roots of all elements of the subgroup. Adjoining an $\ell$-th root of $h\in \Kprime^\times$ results in an unramified extension if and only if for every place $v$ of $\Kprime$, $h=u_v\pi_v^{\ell n_v}$ for some unit $u_v\in \mathcal{O}_v^\times$ and $n_v\in\Z$, where $\pi_v$ is a choice of uniformizer. This is equivalent to requiring $\ell\mid\divisor(h)$, that is, $h\in H$.
\end{proof}

As a consequence of \cref{max_unram}, every subextension $K_\Gamma/\Kprime$ of $K_H/\Kprime$ is an abelian extension of $\F_\qg(C)$. Let $A_\Gamma:=\Gal(K_\Gamma/\F_\qg(C))$.
See the following diagram for a summary of the fields involved, and the Galois groups corresponding to some of the extensions.
\begin{center}
\begin{tikzcd}
& K_\Gamma \ar[d,dash,"A_\Gamma"] &\\
& \F_{q^\gamma}(C)\arrow[dl, dash] \arrow[dr, dash,"\langle \xi\rangle \simeq \mathbf{Z}/\ell\mathbf{Z}"]\arrow[dd, dash, "G_C"] & \\
\F_q(C) \arrow[dr, dash] & &\F_{\qg}(x) \arrow[dl,dash,"\langle \frob\rangle \simeq \mathbf{Z}/\gamma\mathbf{Z}"] \\
& \F_{q}(x) &
\end{tikzcd}
\end{center}

\begin{proposition}\label{relating_reps}
    Let $\Gamma\leq H$ be a $G_{\F_q(x)}$-invariant subgroup. 
    Then we have an isomorphism $A_\Gamma\simeq \Gamma^*$ of $G_C$-representations, where $\Gamma^*=\Hom_{G_{\F_q(x)}}(\Gamma,\mu_\ell)$ is the cohomological dual of $\Gamma$ defined in \Cref{def:dual}.
\end{proposition}
\begin{proof}
    By Kummer theory, we have an isomorphism of groups
    \[\Gamma\simeq \Hom(A_\Gamma,\mu_\ell),\]
    where $f\in\Gamma$ is associated to the homomorphism $\tau\mapsto \tau(\sqrt[\ell]{f})/\sqrt[\ell]{f}$ for any choice of $\ell$-th root of $f$. It suffices to check that this isomorphism is $G_{\F_q(x)}$-equivariant. For $\sigma\in G_{\F_q(x)}$, $\tau\in A_\Gamma$, and $f\in\Gamma$, we have $\sigma(\sqrt[\ell]{f})=\zeta^t\sqrt[\ell]{\sigma(f)}$ for some integer $t$, and so
    \[\sigma\left(\frac{\tau(\sqrt[\ell]{f})}{\sqrt[\ell]{f}}\right)=\frac{\sigma\tau\sigma^{-1}(\zeta^t\sqrt[\ell]{\sigma(f)})}{\zeta^t\sqrt[\ell]{\sigma(f)}}=\frac{\sigma\tau\sigma^{-1}(\sqrt[\ell]{\sigma(f)})}{\sqrt[\ell]{\sigma(f)}},\]
    the last equality following because $\sigma\tau\sigma^{-1}$ is in $A_\Gamma$ and therefore fixes $\zeta\in\F_\qg$. 
\end{proof}

\begin{lemma}\label{lem:split}
    Let $\Gamma\leq H$ be a $G_{\F_q(x)}$-invariant subgroup. The map $\Gal(K_\Gamma/\F_q(x))\to G_C=\Gal(\Kprime/\F_q(x))$ has a splitting; equivalently,
    \[\Gal(K_\Gamma/\F_q(x))\simeq A_\Gamma\rtimes G_C.\]
\end{lemma}
\begin{proof}
    We follow the proof of Lemma 3.1.3 of~\cite{schaeferstubley}. The extension
    \begin{equation}\label{eq: AW GC sequence}
        1\to A_\Gamma\to \Gal(K_\Gamma/\F_q(x))\to G_C\to 1
    \end{equation}
    determines a cohomology class $[\Gal(K_\Gamma/\F_q(x))]\in H^2(G_C,A_\Gamma)$, and it suffices to determine whether this class is $0$. To do this, we consider the subgroup 
    \[\langle \xi\rangle =\Gal(\Kprime/\F_\qg(x))\leq G_C,\]
    which determines a restriction map $H^2(G_C,A_\Gamma)\to H^2(\langle \xi\rangle,A_\Gamma)$. The image of $[\Gal(K_\Gamma/\F_q(x))]$ under this restriction map corresponds to the sequence
    \[1\to A_\Gamma\to \Gal(K_\Gamma/\F_\qg(x))\to \Gal(\Kprime/\F_\qg(x))\to 1.\]
    We show that this map has a splitting. Let $f_1$ be a place of $K_\Gamma$ lying above $f$, the place determined by the irreducible polynomial $f(x)$. Since $f$ is totally ramified in $\Kprime/\F_\qg(x)$, but the extension $K_\Gamma/\Kprime$ is unramified by \cref{max_unram}, the inertia group at $f_1$ is a copy of $\Z/\ell\Z$ in $\Gal(K_\Gamma/\F_\qg(x))$ that maps isomorphically to $\Gal(\Kprime/\F_\qg(x))$, defining a splitting as desired.
    
    Hence $[\Gal(K_\Gamma/\F_q(x))]$ maps to the zero class under the restriction map. But by the Lyndon-Hoschild-Serre spectral sequence we have an inflation-restriction exact sequence
    \[H^2(G_C/\langle \xi\rangle, A_\Gamma^{\langle \xi\rangle})\to H^2(G_C, A_\Gamma)\to H^2(\langle \xi\rangle, A_\Gamma).\]
    Since $G_C/\langle \xi\rangle=\Gal(\F_\qg(x)/\F_q(x))$ has order $\gamma$ (coprime to $\ell$), while $A_\Gamma$ is a vector space over $\F_\ell$, the first term of this sequence is $0$. Hence the restriction map is injective, proving that $[\Gal(K_\Gamma/\F_q(x))]=0$ as desired.
\end{proof}

\begin{lemma}\label{lem:2dim iso to Jl}
    Suppose $\Gamma\leq H$ is isomorphic as a $G_{\F_q(x)}$-representation to $\bigrep{k}{n}$ for $k\geq 2$. Then $\Gamma$ maps isomorphically onto its image in $J[\ell](\F_\qg)$ under the map defined in \cref{lem:H to Jl}.
\end{lemma}
\begin{proof}
    If $\Gamma$ contains $\zeta\in\mu_\ell$, then $\zeta$ spans a one-dimensional $G_{\F_q(x)}$-subrepresentation of $\Gamma$ isomorphic to $\chi$. The unique one-dimensional subrepresentation of $\bigrep{k}{n}$ is $\chi^{n+k}$, so we can conclude $n=1-k$. Since $k\geq 2$, there is a two-dimensional subrepresentation of $\bigrep{k}{1-k}$ isomorphic to $\rho$. Hence there exists $h\in \Gamma$ with $\frob(h)=h$ and $\xi(h)=\zeta h$ as elements of $H$. This implies that the map $\sigma\mapsto \sigma(h)/h$ is a homomorphism $G_{\F_q(x)}\to\mu_\ell$ that factors through $G_C$ and sends $\frob\mapsto 1$ and $\xi\mapsto \zeta$. But the map $\sigma\mapsto \sigma(y^{-1})/y^{-1}$ is identical, so by Kummer theory $h$ is equal to $y^{-1}$ as elements of $\F_\qg(C)^\times/\F_\qg(C)^{\times\ell}$. This is a contradiction because $\divisor(y)\notin\ell \ \Div(C_{\F_\qg})$ and so $y^{-1}\notin H$. We can conclude that $\Gamma\cap\mu_\ell=\{1\}$ and so $\Gamma$ maps isomorphically to its image in $J[\ell](\F_\qg)$.
\end{proof}

\subsection{Relating eigenvectors to Galois representations}
As stated in \cref{lem:GaloisrepW}, the existence of a $\Frob$ eigenvector in $F_n^k$ is equivalent to the existence of a $G_{\F_{q}(x)}$-subrepresentation of $J[\ell](\F_{q^\gamma})$ isomorphic to $\Sym^{k-1}\rho\otimes\chi^{n+1-k}$. Recall the representation $\Sym^{k-1}\rho\otimes\chi^{n+1-k}$ has kernel $\Gal(\F_{q}(x)^\sep/\F_{q^\gamma}(C))$. We will transform this condition to the existence of a $G_{\F_{q}(x)}$ representation related to a field extension of $\F_{q^\gamma}(C))$.

\begin{definition}\label{def: kernel field}
    Let $G=\Gal(L/F)$ for some separable field extension $L/F$, and $\theta$ a representation of $G$. The \emph{kernel field} of $\theta$, denoted $K^\theta$, is the fixed field in $L$ of $\ker\theta$.
\end{definition}
If $\theta$ is a finite-dimensional representation over $\F_\ell$, then $\ker\theta$ is finite index in $G$, and so $K^\theta$ is a finite extension of the base field $F$. The kernel field is a Galois extension of the base field $F$, and by the first isomorphism theorem, $\theta$ descends to a faithful representation of $\Gal(K^\theta/F)$. In the following statement, we consider the case $G=G_{\F_q(x)}=\Gal(\F_q(x)^\sep/\F_q(x))$.

\begin{proposition}\label{prop:eigen to rep}
    Let $2\leq k\leq \ell-1$. There exists an eigenvector of $\Frob$ in $F_n^k$ if and only if there exists a representation $\psi:G_{\F_q(x)}\to \GL_{k+1}(\F_\ell)$ satisfying the following conditions:
    \begin{enumerate}[label=(\alph*)]
        \item there is an exact sequence of $G_{\F_q(x)}$-representations
        \[0\to \bigrep{k-1}{1-n}\to \psi\to\F_\ell\to 0,\]
        where $\F_\ell$ denotes the one-dimensional trivial representation;
        \item the kernel field $K^\psi$ of $\psi$ is an unramified extension of $\F_\qg(C)$ with $\Gal(K^\psi/\F_\qg(C))\simeq \F_\ell^k$.
    \end{enumerate}
  
\end{proposition}
\begin{proof}

    Recall $G_C=\Gal(\F_\qg(C)/\F_q(x))$.
    Given a Frobenius eigenvector $w\in F_n^k$, let $\langle G_Cw\rangle\leq J[\ell](\F_\qg)$ be the span of the $G_C$-orbit of $w$, and let $\Gamma\leq H$ be the image of this subgroup under the map $D\mapsto h_D$ from the proof of \cref{lem:H to Jl}.
    Then as $G_{\F_q(x)}$-representations we have $\Gamma\simeq \langle G_Cw\rangle\simeq \bigrep{k-1}{n-k+1}$ by \cref{lem:GaloisrepW}, and hence $A_\Gamma\simeq \bigrep{k-1}{1-n}$ by \cref{relating_reps} and \cref{bigrep_dual}.
    
    We can use the splitting \cref{lem:split} to construct a representation $\psi:G_{\F_q(x)}\to \GL_{k+1}(\F_\ell)$. We assert that this representation factors through $\Gal(K_\Gamma/\F_q(x))$ and for an arbitrary element $(\tau,\sigma)\in A_\Gamma\rtimes G_C$
    we define
    \[\psi(\tau,\sigma):=\begin{pmatrix}
        \theta(\sigma) & \tau \\ 0 & 1
    \end{pmatrix},\]
    where $\theta\simeq\bigrep{k-1}{1-n}$ is the $G_{\F_q(x)}$-representation on $A_\Gamma$. Since $\ker \psi \le \ker \theta$,
    the kernel field $K^\psi$ contains $K^\theta$, which is equal to $K^\rho$ since $k\geq 2$; hence $K^\psi$ is an extension of $\F_\qg(C)$. By construction, $\psi$ satisfies the desired exact sequence in condition (a). The fact that $\psi$ factors through $\Gal(K_\Gamma/\F_q(x))$ implies $K^\psi$ is contained in $K_\Gamma$, which is an unramified extension of $\F_\qg(C)$ by \cref{max_unram}. In fact $K^\psi=K_\Gamma$ because $\psi$ acts faithfully on $A_\Gamma$, so $\Gal(K^\psi/\F_\qg(C))=A_\Gamma\simeq \F_\ell^k$. 

    Conversely, suppose there exists $\psi:G_{\F_q(x)}\to \GL_{k+1}(\F_\ell)$ with the given properties. Let $\theta$ be the subrepresentation isomorphic to $\bigrep{k-1}{1-n}$. The exact sequence implies that with respect to an appropriate basis, $\psi$ can be written in the form 
    \[\psi(\sigma)=\begin{pmatrix}
        \theta(\sigma) & a(\sigma) \\ 0 & 1
    \end{pmatrix}\]
    for some $a(\sigma)\in \F_\ell^k$. Condition (b) says that $K^\psi/\F_\qg(C)$ is an unramified elementary $\ell$-extension, so by \cref{max_unram}, $K^\psi=K_\Gamma$ for some $\Gamma\leq H$.
   
    From the matrix form for $\psi$ we see that for $\sigma,\tau\in G_{\F_q(x)}$ with $\tau$ mapping into $\Gal(K_\Gamma/\F_\qg(C))$, we have $a(\sigma\tau\sigma^{-1})=\theta(\sigma)a(\tau)$, 
    so $\Gal(K_\Gamma/\F_\qg(C))$ is isomorphic to $\theta$ as a $G_{\F_q(x)}$-representation. So by \cref{relating_reps}, $\Gamma$ is isomorphic as a $G_{\F_q(x)}$-representation to $\bigrep{k-1}{n-k+1}$, which maps isomorphically to a subrepresentation of  $J[\ell](\F_\qg)$ by \cref{lem:2dim iso to Jl}. The structure of $\bigrep{k-1}{n-k+1}$ implies existence of a vector $v\in W$ such that $\Frob(v)=q^{n-k+1}v$ and $G_C\cdot v$ spans $W$; by considering the dimension of $W$ we can conclude $v\in V_k\setminus V_{k-1}$ and therefore $v\in F_n^k$. 
\end{proof}

In the next section we will relate the existence of this representation $\psi$ to the existence of a certain cohomology class in $H^1(G_{\F_q(x)},A_\Gamma)\simeq H^1(G_{\F_q(x)},\bigrep{k-1}{1-n})$.

\begin{remark}
    A version of \cref{prop:eigen to rep} holds also for $k=1$, but this requires a different set of conditions on $\psi$. First, the fixed field of the kernel of $\bigrep{k-1}{1-n}=\chi^{1-n}$ is not $\F_\qg(C)$, but rather some subfield of $\F_\qg(x)$ depending on the value of $n$; thus the fixed field $K^\psi$ of $\ker\psi$ may not be an extension of $\F_\qg(C)$. We must replace condition (b) with the condition that $K^\psi\cdot\F_\qg(C)/\F_\qg(C)$ is unramified, and then we may continue the proof as above but with $K_\Gamma:=K^\psi\cdot\F_\qg(C)$. Second, if $n\equiv 1\bmod\gamma$ then there may exist a representation $\psi:G_{\F_q(x)}\to \GL_2(\F_\ell)$ satisfying all the conditions, but for which the fixed field $K^\psi$ is the degree $\ell$ base field extension $\F_{q^{\ell}}(x)$; then $K^\psi\cdot \F_\qg(C)$ does not correspond to any nontrivial subspace of $J[\ell](\F_\qg)$. To obtain the desired equivalence we must impose the condition that $K^\psi$ is not unramified over $\F_q(x)$.
    While it is possible to keep track of these additional constraints, we consider only $k\geq 2$ for convenience, since we already have a criterion for the existence of a $\Frob$ eigenvector in $F_n^1$ by \cref{basic_lifting_properties}(a).
\end{remark}

\section{Galois cohomology}\label{sec:galois_cohom}

Throughout this section and the next we will rely on many facts about Galois cohomology of global fields, most of which can be found in Neukirch, Schmidt, and Wingberg~\cite{cohomology}.

\subsection{Cohomology classes and kernel fields}\label{subsec:Cohomologyclasses}

Let $G$ be a group. 
Given an $n$-dimensional $\F_\ell$ representation $\theta:G\to \GL_n(\F_\ell)$, an element $a\in H^1(G,\theta)$ can be represented by a crossed homomorphism $\alpha:G\to \F_\ell^n$ satisfying $\alpha(\sigma\tau)=\theta(\sigma)\alpha(\tau)+\alpha(\sigma)$ for $\sigma,\tau\in G$. Any different representative $\alpha'$ for the same cohomology class $a$ differs from $\alpha$ by a coboundary. That is, there exists an element $v \in \F_\ell^n$ such that $\alpha'(\sigma)-\alpha(\sigma)=\theta(\sigma)v-v$ for all $\sigma \in G$.
\begin{definition}
    With a representation $\theta: G \to \GL_n(\F_\ell)$ and a crossed homomorphism $\alpha:G \to \F_\ell^n$ as above, define a $n+1$ dimensional $\F_\ell$ representation $\theta[\alpha]:G\to\GL_{n+1}(\F_\ell)$ by
    \begin{align*}
        \theta[\alpha](\sigma):\F_\ell^n\times\F_\ell&\to \F_\ell^n\times\F_\ell\\
        (v,c)&\mapsto (\theta(\sigma)v+c\alpha(\sigma),\,c).
    \end{align*}
    We say that $\theta[\alpha]$ is the \emph{extension of $1$ by $\theta$} associated to $\alpha$.
\end{definition}

This definition can be summarized using matrix notation: 
\begin{align}
    \theta[\alpha]:=\left(\begin{array}{ccc;{2pt/3pt}c}
         &  & & \\
        & \theta &  & \alpha \\
         &  & & \\ \hdashline[2pt/3pt]
        & 0 & & 1 
    \end{array}\right).
\end{align}
If $\alpha'$ is a different crossed homomorphism which represents the same class $a \in H^1(G,\theta)$, then we have
\[ \theta[\alpha']=
    \left(\begin{array}{ccc;{2pt/3pt}c}
         &  & & \\
        & I &  & -v \\
         &  & & \\ \hdashline[2pt/3pt]
        & 0 & & 1 
    \end{array}\right)
    \left(\begin{array}{ccc;{2pt/3pt}c}
         &  & & \\
        & \theta &  & \alpha \\
         &  & & \\ \hdashline[2pt/3pt]
        & 0 & & 1 
    \end{array}\right)
     \left(\begin{array}{ccc;{2pt/3pt}c}
         &  & & \\
        & I &  & v \\
         &  & & \\ \hdashline[2pt/3pt]
        & 0 & & 1 
    \end{array}\right)
 \]
where $v\in \F_\ell^n$ is the vector satisfying $\alpha'(\sigma)-\alpha(\sigma)=\theta(\sigma)v-v$.
Thus, different representatives of a cocycle class give rise to representations that are equivalent up to conjugation. From now on, we will denote this representation class by $\theta[a]$ since it only depends on $\theta$ and the cocycle class $a$.

It is straightforward to check that $\theta[a]$ is a $G$-representation and fits into an exact sequence 
$$0\to \theta\to\theta[a]\to \F_\ell\to 0$$
of $G$-representations, where $\F_\ell$ represents the one-dimensional trivial representation. This exact sequence induces a long exact sequence in cohomology which begins
\[0\to H^0(G,\theta)\to H^0(G,\theta[a])\to H^0(G,\F_\ell)\xrightarrow{\delta} H^1(G,\theta),\]
and we have $a=\delta(1)$. Conversely, for any exact sequence $0\to \theta\to\kappa\to \F_\ell\to 0$ of $G$-representations there is a class $a\in H^1(G,\theta)$ with $\kappa\simeq \theta[a]$.

Recall the definition of the kernel field of a Galois representation in \cref{def: kernel field} and for $a \in H^1(G,\theta)$, we will denote the kernel field of $\theta[a]$ as $K^a$.
Since $\ker\theta[a]\leq \ker\theta$, $K^a$ is necessarily an extension of $K^\theta$. 
One can check that the kernel field is well-defined on cohomology classes (in particular, if $a$ is a coboundary then $K^a=K^\theta$, though the converse does not necessarily hold), and in fact that the kernel field is invariant under scaling:
\begin{lemma}\label{lem: kernel field scaling}
    If $a\in H^1(G,\theta)$ and $c\in\F_\ell^\times$ then $K^{ca}=K^a$.
\end{lemma}

\subsection{Selmer conditions}\label{sec:selmer}

Let $\mathcal{M}$ denote the set of all places of $\F_q(x)$.
For each $v\in\mathcal{M}$, let $\F_q(x)_v$ denote the localization of $\F_q(x)$ at $v$, and pick once and for all an inclusion $\F_q(x)^\sep\hookrightarrow\F_q(x)_v^\sep$ of separable closures. This is equivalent to picking a prime above $v$ in $\F_q(x)^\sep$, or equivalently a compatible system of one place above $v$ in each finite extension of $\F_q(x)$. We define
\begin{align*}
    G_{\F_q(x)_v}&:=\Gal(\F_q(x)_v^\sep /\F_q(x)_v),
\end{align*}
and the inclusion $\F_q(x)^\sep\hookrightarrow\F_q(x)_v^\sep $ induces an inclusion $G_{\F_q(x)_v}\hookrightarrow G_{\F_q(x)}$ by restriction. The image of $G_{\F_q(x)_v}$ is the decomposition group of the prime above $v$ in $\F_q(x)^\sep $.

Given a Galois representation $\theta:G_{\F_q(x)}\to \GL_n(\F_\ell)$, we define $\theta_v$ to be its restriction to the decomposition group $G_{\F_q(x)_v}$. 
For convenience we will define the notation
\[H^i(\theta):=H^i(G_{\F_q(x)},\theta)\qquad\text{and}\qquad H^i(v,\theta):=H^i(G_{\F_q(x)_v},\theta_v).\]
By restriction to the decomposition group $G_{\F_q(x)_v}$, we obtain a map
\[\res_{v}: H^{i}(\theta) \to H^{i}(v,\theta).\]

For any $v\in\mathcal{M}$, let $k_v$ denote the residue field of $\F_q(x)$ at the place $v$, so that $G_{k_{v}}$ is the quotient of $G_{\F_q(x)_v}$ by the inertia group $I_v$ above $v$. Let $S=\{f,\infty\}\subseteq\mathcal{M}$, where $f$ denotes the place determined by the irreducible polynomial $f(x)$ defining the curve $C$, and $\infty$ is the place determined by $\frac1x$. 

For each $v\in\mathcal{M}$ we define a subgroup $L_{v} \subseteq H^{1}(v, \theta)$:
\begin{equation}\label{eq:Lvdef}
    L_{v}:=\left\{\begin{array}{ll}
    H^{1}(v,\theta) & v\in S, \\
    H^{1}(G_{k_{v}}, \theta) & v\in\mathcal{M}\setminus S.
\end{array}\right.
\end{equation}
The subgroup $L_v$ for $v\notin S$ is the ``unramified subspace'' 
\[H^{1}(G_{k_{v}}, \theta)=\ker(H^1(G_{\F_q(x)_v},\theta)\to H^1(I_v,\theta)),\]
which is equal to the group defined in \cite[Definition 7.2.14]{cohomology} by the inflation-restriction exact sequence.
The unramified subspace is so called because kernel fields of classes in the unramified subspace introduce no new ramification at $v$: if $I_v\leq \ker\theta$ (so the kernel field of $\theta$ is unramified over $\F_q(x)$ at $v$), and if $a\in H^1(G_{\F_q(x)},\theta)$ satisfies $\res_v(a)\in H^1(G_{k_v},\theta)$, then in fact we also have $I_v\leq \ker \theta[a]$ (the kernel field of $a$ is unramified over $\F_q(x)$ at $v$).
With this setup, we can define the \emph{Selmer group}
\begin{align*}
H^{1}_S(\theta) 	& := \{a \in H^{1}(\theta): \res_{v}(a) \in L_{v} \text{ for all } v\in \mathcal{M}\} \\
								& = \ker\left(H^{1}(\theta) \overset{\res}{\to} \prod_{v\in \mathcal{M}} H^{1}(v,\theta)/L_{v}\right).
\end{align*}

\subsection{A basis for local cohomology groups}\label{sec:local_cohom_dim}
Recall that $\gamma$ is the order of $q$ in $\F_\ell^\times$ and $d=\deg f$. In \cref{subsec:defreps} we defined $G_{\F_q(x)}$-representations $\chi$ and $\rho$ which factor through $G_C=\Gal(\F_{q^\gamma}(C)/\F_q(x))$.

Let 
\[\theta=\bigrep{k}{n},\] 
and let $h^i(\theta)$ denote the $\F_\ell$-dimension of cohomology group $H^i(\theta)$. 
Since $\theta(\xi)$ is similar to a Jordan block as is described in \cref{eq:SymFrobzeta}, the dimension of $H^0(\theta)$ depends only on whether the top-left entry of $\theta(\frob)$ equals $1$. That is, 
\begin{equation}\label{eq:h0}
    h^0(\theta)=\begin{cases} 1 & \chi^{n+k} = 1 \\ 0 & \text{otherwise} \end{cases}=\begin{cases} 1 & \gamma \mid n+k \\ 0 & \text{otherwise.} \end{cases}
\end{equation}
We now consider $\theta_v$, the restriction of $\theta$ to the decomposition group $G_{\F_q(x)_v}$, for $v\in S=\{f,\infty\}$.
The polynomial $f(x)$ splits over $\F_\qg$ into $\gcd(d,\gamma)$ factors which are cyclically permuted by $\frob\in\Gal(\F_\qg(x)/\F_q(x))$. The decomposition group $G_{\F_q(x)_f}$ fixes these factors, so the only powers of $\frob$ lying in the image of $G_{\F_q(x)_f}\to G_C$ are powers of $\frob^d$. On the other hand $\infty$ has a unique prime above it in $\F_\qg(x)$, so $\frob$ is in the image of $G_{\F_q(x)_\infty}$. We can conclude that
\begin{align}\label{h0local}
h^{0}(f, \theta) & = \begin{cases} 1 & \gamma\mid d(n+k) \\ 0 & \text{otherwise,} \end{cases}, &\qquad 
h^{0}(\infty, \theta) & = \begin{cases} 1 & \gamma\mid (n+k) \\ 0 & \text{otherwise.} \end{cases}
\end{align}
By Tate duality~\cite[(7.2.6)]{cohomology} and \cref{bigrep_dual}, this implies
\begin{align}\label{h2local}
h^{2}(f, \theta) & = \begin{cases} 1 & \gamma\mid d(1-n) \\ 0 & \text{otherwise,} \end{cases}, &\qquad 
h^{2}(\infty, \theta) & = \begin{cases} 1 & \gamma\mid (1-n) \\ 0 & \text{otherwise.} \end{cases}
\end{align}
Finally, since $\ell$ is coprime to $q$, the local Euler characteristic of $\theta$ at any $v\in\mathcal{M}$ is trivial~\cite[(7.3.2)]{cohomology} and so
\begin{align}
h^{1}(v, \theta) & = h^0(v,\theta)+h^2(v,\theta).
\end{align}
In particular, we observe that $H^1(v,\theta)$ is at most $2$-dimensional for $v\in S$.

We now give an explicit basis for $H^1(v,\theta)$ and discuss their corresponding kernel fields.

\begin{lemma}\label{lem:H1v basis}
    Let $v=f$ or $\infty$, and let $\delta=d$ or $1$ respectively. Let $\theta=\bigrep{k}{n}$ for $0\leq k\leq \ell-1$, and $\theta_v$ the restriction to $G_{\F_q(x)_v}$.
    \begin{enumerate}[label=(\alph*)]
        \item If $\gamma\mid \delta(n+k)$ then there exists a nonzero element $\mathbf{ur}\in H^1(v,\theta)$ such that the kernel field $K^{\mathbf{ur}}$ is the degree $\ell$ unramified extension of $K^{\theta_v}$.
        \item If $\gamma\mid \delta(1-n)$ then there exists a nonzero element $\mathbf{b}\in H^1(v,\theta)$ such that the kernel field $K^{\mathbf{b}}$ is $\F_\qg(C)_v$. 
    \end{enumerate}
    Further, $H^1(v,\theta)$ has a basis consisting of whichever of $\mathbf{b}$ and $\mathbf{ur}$ it contains.
\end{lemma}
Here $\F_\qg(C)_v$ means the completion of $\F_\qg(C)$ at the place above $v$ determined by the decomposition group $G_{\F_q(x)_v}$.

\begin{remark}
    If $k\geq 1$ then the $K^{\theta_v}=\F_\qg(C)_v$, so when $\mathbf{b}$ exists in $H^1(v,\theta)$, $K^{\mathbf{b}}$ is not a nontrivial extension of $K^{\theta_v}$; this serves as a caution that a nontrivial cohomology class may define a trivial extension of kernel fields. However, if $k=0$ then $K^{\theta_v}$ is a subfield of $\F_\qg(x)_v$, so $K^{\mathbf{b}}/K^{\theta_v}$ is a nontrivial extension in this case.
\end{remark}

\begin{proof}   
    We start by defining an \emph{unramified class} $\mathbf{ur}\in H^1(v,\F_\ell)$. Since $\F_\ell$ is the trivial representation, a class in $ H^1(v,\F_\ell)$ is represented by a group homomorphism.
Let $\mathbf{ur}\in H^1(v,\F_\ell)$ be defined by the property that it factors through $\Gal(\F_{q^\ell}(x)_v/\F_q(x)_v)$, where $\mathbf{ur}$ is represented by $\Gal(\F_{q^\ell}(x)_v/\F_q(x)_v) \to \F_\ell$ given by sending the Frobenius map $c\mapsto c^q$ to $1$.  
    
    If $\gamma\mid \delta(n+k)$, then following \cref{eq:h0} $\F_\ell$ is a subrepresentation of $\theta_v$, and the image of $\mathbf{ur}\in H^1(v,\F_\ell)$ under $H^1(v,\F_\ell)\to H^{1}(v, \theta_v)$ we also denote by $\mathbf{ur}$. The kernel field of $\mathbf{ur}\in H^{1}(v, \theta_v)$ is the compositum $\F_{q^\ell}\cdot K^{\theta_v}$; as $K^{\theta_v}$ is the composite of an unramified $\Z/\gamma\Z$ extension with a ramified $\Z/\ell\Z$ extension, it never contains $\F_{q^\ell}$. Thus $K^{\mathbf{ur}}/K^{\theta_v}$ is a nontrivial extension and so the class $\mathbf{ur}$ is nonzero.
    
    If $\gamma\mid \delta(1-n)$, then $\chi^n_v=\chi_v$, so the localizations at $v$ of $\bigrep{k}{n}$ and $\bigrep{k}{}$ are isomorphic. The representation $\Sym^{k+1}\rho$ is an extension of $1$ by $\bigrep{k}{}$ following \cref{lem:extend}, so we can define $\mathbf{b}\in H^{1}(v,\theta_v)$ to be the class corresponding to this extension by $1$, or equivalently the image of $1\in H^0(v,\F_\ell)$ under $\delta$ in the long exact sequence
    \[0\to H^0(v,\bigrep{k}{})\to H^0(v,\Sym^{k+1}\rho)\to H^0(v,\F_\ell)\xrightarrow{\delta} H^1(v,\bigrep{k}{}).\]
    The injection $H^0(v,\bigrep{k}{})\to H^0(v,\Sym^{k+1}\rho)$ is an isomorphism because both groups have the same dimension, and therefore $\mathbf{b}=\delta(1)\neq 0$. 
    The kernel field of $\mathbf{b}$ is equal to the fixed field of $\ker\rho_v$. Note that this equals the fixed field of $\ker\theta_v$ provided $k\geq 1$, which is why we can't use the kernel field to deduce that $\mathbf{b}$ defines a nonzero cohomology class.

    If both $\gamma\mid \delta(n+k)$ and $\gamma\mid \delta(1-n)$, then $\mathbf{b}$ and $\mathbf{ur}$ define independent classes in $H^1(v,\bigrep{k}{n})$ because they define distinct kernel fields over $\F_q(x)_v$ (\cref{lem: kernel field scaling}). So whichever of the elements $\mathbf{b}$ and $\mathbf{ur}$ exist in $H^1(v,\theta_v)$, they span a subspace that matches the dimension $h^1(v,\theta)$ computed above, and therefore they must form a basis.
\end{proof}

\subsection{Selmer groups of characters}\label{sec:selmer_characters}

Using the computations from \cref{sec:local_cohom_dim}, we can compute the dimension of global Selmer groups of characters. 

\begin{lemma}\label{h1S_character_dimension} Recall $h^1_S(\cyclo^n)$ denotes the dimension of the cohomology group $H^1_S(\cyclo^n)$. We have
\[h^1_S(\cyclo^n)=\begin{cases} 1 & \gamma \mid d(1-n)  \\ 0 & \text{otherwise} \end{cases}+\begin{cases} 1 & \gamma \mid n  \\ 0 & \text{otherwise.} \end{cases}\]
\end{lemma}
\begin{proof}
We begin by defining an alternate Selmer group. Let $\theta$ be a $G_{\F_q(x)}$-representation and $\theta^*$ its cohomological dual (\cref{def:dual}). For each place $v\in\mathcal{M}$ we define a subgroup of $H^1(v,\theta^*)$ by
\[L_{v}^\perp:=\left\{\begin{array}{ll}
    0 & v\in S, \\
    H^{1}(G_{k_{v}}, \theta^*) & v\in\mathcal{M}\setminus S,
\end{array}\right.\]
where $H^1(G_{k_v},\theta^*)$ is the unramified subspace as defined in \cref{eq:Lvdef}. For all places $v$ of $\F_q(x)$, the subgroup $L_v^\perp\leq H^1(v,\theta^*)$ is precisely the annihilator under the local Tate pairing of the subgroup $L_v(\theta)\leq H^1(v,\theta)$ as defined in \cref{eq:Lvdef}  \cite[Theorem 7.2.15]{cohomology}. 
Using these subgroups we define
\begin{align*}
H^{1}_{S^*}(\theta^*) 	& := \{a \in H^{1}(\theta^*): \res_{v}(a) \in L_{v}^\perp \text{ for all } v\}.
\end{align*}

In this setting, the Greenberg-Wiles formula~\cite[Theorem 8.7.9]{cohomology} reduces to
\[\frac{\#H^1_S(\chi^n)}{\#H^1_{S^*}(\chi^{1-n})}=\frac{\#H^0(\chi^n)}{\#H^0(\chi^{1-n})}\cdot\frac{\#H^1(f,\chi^n)}{\#H^0(f,\chi^n)}\cdot\frac{\#H^1(\infty,\chi^n)}{\#H^0(\infty,\chi^n)},\]
because for all $v\notin\{f,\infty\}$, we have $\#L_v=\#H^0(v,\theta)$ (see the proof of \cite[(8.7.9)]{cohomology}).
The right-hand side of this equation can be determined using the dimension computations in \cref{sec:local_cohom_dim}, and equals the right-hand side of the statement of the lemma. So it suffices to show that $H^1_{S^\ast}(\chi^{1-n})$ is trivial.

Now suppose there exists a nonzero class $a\in H^1_{S^\ast}(\cyclo^{1-n})$ and let $K^a$ denote the kernel field of $a$. For $v\in S$, the Selmer condition $L_v^\perp$ implies 
that the $G_{\F_q(x)_v}$-representations $\theta^*_v$ and $\theta^*[a]_v=\theta^*_v[\res_v(a)]$ have the same kernel field,
so the extension $K^a/K^\theta$ is totally split at all places over $v\in S$. Since $a$ is not a coboundary, the extension of $1$ determined by $a$ must have a nontrivial unipotent element in its image, so there is an element of order $\ell$ in $\Gal(K^a/\F_q(x))$. Since $\Gal(\F_\qg(x)/\F_q(x))$ has order coprime to $\ell$, this implies that $K^a$ is a $\Z/\ell\Z$ extension of some subfield of $\F_{q^{\gamma}}(x)$
that is unramified everywhere and split at $f$ and $\infty$. No such extensions exist, so in fact $H^1_{S^\ast}(\cyclo^{1-n})$ is trivial. Hence the dimension of $H^1_S(\chi^n)$ is exactly as predicted by the statement of the lemma. 
\end{proof}

\subsection{From $\Frob$ eigenvectors to cohomology}

Using the cohomology computations above, we can now complete the work we began in \cref{sec:galoisreps} of relating the existence of eigenvectors of $\Frob$ in $F_n^k$ to the existence of certain cohomology classes.

\begin{lemma}\label{lem:extension_unramified}
    Let $a\in H^1_S(\bigrep{k}{n})$ for some $1\leq k\leq \ell-1$, and let $L/\F_q(x)$ be the kernel field of $a$. Then $L$ is an unramified extension of $\F_\qg(C)$.
\end{lemma}

\begin{proof}
    The kernel field contains the fixed field of $\ker\bigrep{k}{n}$, which is $\F_\qg(C)$ because $k\geq 1$. The Selmer condition ensures $L$ is unramified over $\F_\qg(C)$ at all $v\notin S$, so it suffices to check the ramification at $v\in S$.
    By \cref{lem:H1v basis}, for $v\in S$, $\res_v(a)$ is a linear combination of $\mathbf{ur}$ and $\mathbf{b}$ (allowing the coefficient to be $0$ if the corresponding class is not in $H^1(v,\theta)$). Since these classes both define unramified extensions of $\F_\qg(C)_v$, the corresponding extensions of $1$ both vanish on the inertia group of $\F_\qg(C)_v$, so the same is true of any linear combination. This implies that the kernel field of $\res_v(a)$ is
    an unramified extension of $\F_\qg(C)_v$. Since the kernel field of $\res_v(a)$ is the completion at a prime above $v$ of the kernel field of $a$, we can conclude that $L$ is unramified over $\F_\qg(C)$ at $v\in S$, and also at all $v\notin S$ by the Selmer condition.
\end{proof}

\begin{proposition}\label{prop:eigen to cohom}
    Let $2\leq k\leq \ell-1$ and $n\in\Z/\gamma\Z$. The following are equivalent:
    \begin{itemize}
        \item There exists an eigenvector of $\Frob$ in $F_n^k$.
        \item $\gamma\nmid n$, and there exists a class $a \in H^1_S(\bigrep{k-1}{1-n})$ that maps to a nonzero class $a'\in H^1_S(\cyclo^{1-n})$ under the map induced by \cref{eq:reps}.
    \end{itemize}
\end{proposition}

\begin{proof}
    Given a Frobenius eigenvector $w\in F_n^k$, we obtain a representation $\psi$ as in \cref{prop:eigen to rep}, which is an extension of $1$ by $\theta:=\bigrep{k-1}{1-n}$ and therefore corresponds to a class $a\in H^1(\theta)$. Letting $K^\psi$ denote the kernel field, by \cref{prop:eigen to rep} we have $K^\psi/\F_\qg(C)$ unramified and therefore $a\in H^1_S(\theta)$. If we write $\psi$ in matrix form by picking a cocycle $\alpha$ as
    \begin{equation}\label{eq:psi matrix}
        \psi(\tau)=\begin{pmatrix}
        \theta(\tau) & \alpha(\tau) \\ 0&1
    \end{pmatrix},
    \end{equation}
    then $\tau\in\Gal(\F_q(x)^\sep/\F_\qg(C))$ maps to a matrix of the form $\begin{pmatrix}
        I_{k} & \alpha(\tau) \\ 0&1
    \end{pmatrix}$.
   Since $\psi$ descends to a faithful representation of $\Gal(K^\psi/\F_\qg(C))$, the map $\tau\mapsto \alpha(\tau)$ is an isomorphism onto $\F_\ell^k$ by \cref{prop:eigen to rep}. So letting $a'\in H^1_S(\chi^{1-n})$ be the class given by restriction of $\alpha$ to the bottom entry, there exists $\tau\in\Gal(K^\psi/\F_\qg(C))$ for which $a'(\tau)\neq 0$; this proves that the kernel field of $a'$ is strictly larger than $\F_\qg(C)$, so $a'$ must define a nonzero class in $H^1_S(\chi^{1-n})$.
    
    We also have $\gamma\nmid n$. For if $\gamma\mid n$, then $\chi^{1-n}=\chi$ and $h^1_S(\chi^{1-n})=1$ by \cref{h1S_character_dimension}. This implies that up to scalar multiple, $a'\in H^1_S(\chi^{1-n})$ is the class defining $\rho$ as an extension of $1$ by $\chi$, and therefore has kernel field $\F_\qg(C)$, a contradiction.

    Conversely, suppose we are given an element $a\in H^1_S(\theta)$ satisfying the described condition with $\gamma\nmid n$. This class corresponds to an extension of $1$ by $\theta$ which we denote $\psi$. Let $K^\psi$ denote the kernel field of $\psi$. By \cref{lem:extension_unramified}, $K^\psi$ is an unramified extension of $\F_\qg(C)$. 

    We can pick a cocycle $\alpha$ such that the matrix form $\psi$ has the form as in \cref{eq:psi matrix} where the last entry $\alpha'$ of $\alpha$ represents the nonzero class $a'\in H^1_S(\chi^{1-n})$ given by the assumption. 
    
    Now we claim that with this matrix form,
    there must exist $\tau\in \Gal(\F_q(x)^\sep /\F_\qg(C))$ for which $\alpha'(\tau) \ne 0$. If not, the fact $\alpha'(\tau)=0$ for all $\tau\in \Gal(\F_q(x)^\sep /\F_\qg(C))$ implies that the kernel field of $a'$ is contained in $\F_\qg(C)$. Thus the representation 
    $\begin{psmallmatrix}
        \chi^{1-n} & \alpha' \\ 0&1
    \end{psmallmatrix}$
    factors through $G_C=\Gal(\F_\qg(C)/\F_q(x)) = \langle \xi,\frob \rangle$. Since 
    \begin{align*}
        q\alpha'(\xi)=\alpha'(\xi^q)=\alpha'(\frob \xi\frob^{-1})=q^{1-n}\alpha'(\xi),
    \end{align*}
    and $\gamma\nmid n$ by assumption, we must have $\alpha'(\xi)=0$, so the representation factors through $\Gal(\F_\qg(x)/\F_q(x))$ which is a cyclic group.
    This contradicts the assumption that $a'$ is a nonzero class in $H^1_S(\chi^{1-n})$.

    Hence there exists $\tau\in \Gal(\F_q(x)^\sep /\F_\qg(C))$ for which $\alpha(\tau)$ has nonzero bottom entry. Now for arbitrary $\sigma\in G_{\F_q(x)}$ we have $\alpha(\sigma\tau\sigma^{-1})=\theta(\sigma)\alpha(\tau)$, so every vector in the orbit of $\alpha(\tau)$ under the $G_{\F_q(x)}$-representation $\theta$ is obtained as $\alpha(\tau')$ for some $\tau'\in \Gal(\F_q(x)^\sep /\F_\qg(C))$. The orbit under $\theta$ of a vector with nonzero bottom entry has full dimension $k$, so the restriction of $\psi$ to $\Gal(\F_q(x)^\sep /\F_\qg(C))$ has image isomorphic to $\F_\ell^k$; in particular this implies $\Gal(K^\psi/\F_\qg(C))\simeq \F_\ell^k$. We thus obtain a representation $\psi$ as in \cref{prop:eigen to rep}, from which we obtain an eigenvector in $F_n^k$.
\end{proof}

\section{Cup products}\label{sec:cup products}

Suppose $0\leq k\leq\ell-3$ and we want to determine the existence of $\Frob$ eigenvectors in $F^{k+2}_{1-n}$ for $n\in\Z/\gamma\Z$. In light of \cref{prop:eigen to cohom}, we are led to consider when $H^1_S(\bigrep{k+1}{n})$ contains an element that maps to a nontrivial element of $H^1_S(\chi^{n})$ under the map induced by \cref{eq:reps}. The map $\bigrep{k+1}{n}\to\chi^{n}$ factors through the map $\bigrep{k+1}{n} \to \bigrep{k}{n}$ coming from \cref{eq:topleft cyclo}, so a necessary condition is that $H^1_S(\bigrep{k}{n})$ contains a class that maps to a nontrivial element of $H^1_S(\chi^{n})$. So we will now assume we are given a class in $H^1_S(\bigrep{k}{n})$, and want to know when it lifts to an element in $H^1_S(\bigrep{k+1}{n})$. We will see that this condition is equivalent to the vanishing of a local cup product, and explain how to determine this vanishing condition explicitly in the case $k=0$.

\subsection{Lifting cohomology classes}\label{sec:lifting_local}

Let $\F_q(x)_S$ denote the maximal separable extension of $\F_q(x)$ that is unramified outside $S$. Then $\F_q(x)_S$ contains $\F_\qg(C)$, and so for any $G_{\F_q(x)}$-representation $\theta$ that factors through $G_C$, $\theta$ descends to a representation of $G_{\F_q(x),S}:=\Gal(\F_q(x)_S/\F_q(x))$. For the Selmer conditions defined in \cref{sec:selmer}, we can write the corresponding Selmer group as a full cohomology group:
\begin{align*}
    H^1_S(\theta)&=\ker\left(H^{1}(G_{\F_q(x)}, \theta) \to H^{1}(\Gal(\F_q(x)^\sep /\F_q(x)_S), \theta)\right)\\
    &=H^{1}(G_{\F_q(x),S}, \theta).
\end{align*}
This will allow us to locate $H^1_S(\theta)$ within a long exact sequence in cohomology. 
We also define $H^i_S(\theta):=H^{i}(G_{\F_q(x),S}, \theta)$ for all $i\geq 0$.

Let $0\leq k\leq \ell-2$. Take the short exact sequence of representations 
\[
0 \to \cyclo^{1+k+n} \to \bigrep{k+1}{n} \to \bigrep{k}{n} \to 0
\]
from \cref{eq:topleft cyclo}, but now considered as representations of $G_{\F_q(x),S}$.  The corresponding long exact sequence has a portion given by
\begin{align}\label{eq:long_exact_seq}
    \cdots\to H^{1}_S(\bigrep{k+1}{n}) \to H^{1}_S( \bigrep{k}{n}) \overset{\delta}{\to} H^{2}_S(\cyclo^{1+k+n}) \to\cdots.
\end{align}
From this we see that a class $a \in H^{1}_S(\bigrep{k}{n})$ lifts to a class $a' \in H^{1}_S(\bigrep{k+1}{n})$ if and only if the image of $a$ under the connecting homomorphism $\delta$ to $H^{2}_S(\cyclo^{1+k+n})$ is $0$.
We will show that this condition can be detected by the vanishing of a cup product
\begin{align*}
    H^{1}_S(\bigrep{k}{})\otimes H^{1}_S(\bigrep{k}{n})\xrightarrow{\cup} H^2_S(\chi^{1+k+n})
\end{align*}
induced by the dual pairing $(\Sym^k\rho)\otimes (\Sym^k\rho)^\vee\to\F_\ell$ together with the calculation in \cref{bigrep_dual}. 

Note that $\Sym^{k+1}\rho$ is an extension of $1$ by $\bigrep{k}{}$, unramified away from $S$, and therefore corresponds to a class in $H^{1}_S(\bigrep{k}{})$. Under the map $H^{1}_S(\bigrep{k}{})\to \chi$ as in \cref{eq:reps}, this class maps to the class $b\in H^1_S(\chi)$ that determines $\rho$ as an extension of $1$, as discussed in \cref{subsec:defreps}. We therefore denote this class by $b^{[k]}$. As was discussed in \cref{subsec:Cohomologyclasses}, we can also identify this class as the image of $1\in H^0_S(\F_\ell)$ under the connecting homomorphism $H^0_S(\F_\ell)\to H^1_S(\bigrep{k}{})$ induced by the following short exact sequence obtained from \cref{eq:reps} by taking $n=0$,
\[0 \to \bigrep{k}{} \to \Sym^{k+1}\rho \to \F_\ell \to 0.\]

\begin{lemma}\label{connecting_is_cup}
The image of a class $a\in H^{1}_S(\bigrep{k}{n})$ under the connecting homomorphism in \cref{eq:long_exact_seq}
\[
\delta: H^{1}_S(\bigrep{k}{n}) \to H^{2}_S(\cyclo^{1+k+n})
\]
equals $b^{[k]} \cup a$.
\end{lemma}

\begin{proof}
This follows from the formal properties of cup products and connecting homomorphisms in group cohomology. To be precise, if we let $\theta=\bigrep{k}{n}$ then we have a $G_C$-equivariant map of short exact sequences
\[\begin{tikzcd}
    0\arrow{r} & (\Sym^k\rho\otimes\chi)\otimes \theta\arrow{r}\arrow{d} & (\Sym^{k+1}\rho)\otimes \theta \arrow{r}\arrow{d} &  \F_\ell\otimes \theta \arrow{r}\arrow{d} &  0 \\
    0 \arrow{r} & \cyclo^{1+k+n} \arrow{r} & \bigrep{k+1}{n} \arrow{r} & \bigrep{k}{n} \arrow{r} & 0.
\end{tikzcd}\]
The left vertical arrow is induced by the dual pairing; the right vertical arrow is the isomorphism $\F_\ell\otimes \theta\simeq \theta$; then there is a unique choice of middle arrow that makes the diagram commute. Then \cite[(1.4.3.i)]{cohomology} says that the following diagram commutes:
\[\begin{tikzcd}
    H^0_S(\F_\ell)\otimes H^1_S(\bigrep{k}{n})\arrow{r}{\cup}\arrow{d}{\delta\otimes 1} & H^1_S(\bigrep{k}{n})\arrow{d}{\delta} & \\
    H^{1}_S(\bigrep{k}{})\otimes H^{1}_S(\bigrep{k}{n})\arrow{r}{\cup} & H^2_S(\chi^{1+k+n}).
\end{tikzcd}\]
Note in particular that the top cup product sends $1\cup a\mapsto a$. Therefore we have
\[\delta(a)=\delta(1\cup a)=\delta(1)\cup a = b^{[k]} \cup a.\qedhere\]
\end{proof}

\subsection{Local cup product}\label{subsec: local cup product}
In the previous section we showed that a class in $H^1_S(\bigrep{k}{n})$ lifts to a class in $H^1_S(\bigrep{k+1}{n})$ if and only if its cup product with $b^{[k]}$ vanishes.
Our next goal is to show that the vanishing of the cup product of $b^{[k]}$ and a class in $H^1_S(\bigrep{k}{n})$ can be detected locally at the place $f$ (\cref{lifting_is_local}).

\begin{lemma}\label{h2S_character_dimension} Recall $h^{2}_{S}(\cyclo^{n})$ denotes the dimension of the cohomology group $H^{2}_{S}(\cyclo^{n})$. We have
\[
h^{2}_{S}(\cyclo^{n}) = \begin{cases} 1 & \gamma \mid d(1-n) \\ 0 & \text{otherwise.} \end{cases}
\]
\end{lemma}
\begin{proof}
This formula follows from \cref{eq:h0}, \cref{h1S_character_dimension}, and the triviality of the global Euler characteristic over function fields~\cite[(8.7.4)]{cohomology}.
\end{proof}

\begin{lemma}\label{h2_injection}
The restriction map $\res_{f}: H^{2}_{S}(\cyclo^{n}) \to H^{2}(f, \cyclo^{n})$ is an isomorphism.
\end{lemma}
\begin{proof}
By \cref{h2local} and \cref{h2S_character_dimension}, if $\gamma\nmid d(1-n)$ then both groups are trivial, so we may suppose $\gamma\mid d(1-n)$. Then both groups have dimension $1$, so it suffices to show that $\res_f$ is nonzero. 
We will extract this from the end of the Poitou-Tate exact sequence~\cite[(8.6.10)]{cohomology}.
Recall since $\chi$ is finite and unramified, we have $\Hom(\chi,\mathcal{O}_S^\times) = \chi^*=\chi^{1-n}$ (\cite[page 387]{cohomology}). The relevant portion of the sequence is
\begin{align}\label{eq:poitou tate part}
    \to H^{2}_{S}(\cyclo^{n}) \xrightarrow{\oplus\res_v} \bigoplus_{v \in S} H^{2}(v, \cyclo^{n}) \to H^{0}_{S}(\cyclo^{1-n})^{\vee} \to 0.
\end{align}
Our ramification set $S$ contains only two places, $f$ and $\infty$.
We know the dimensions of every term appearing in this sequence: the global $H^{2}_{S}$ by \cref{h2S_character_dimension}, the middle terms by \cref{h2local}, and the term $H^{0}_{S}(\cyclo^{1-n})^{\vee}$ has dimension $1$ when $n\equiv 1\bmod \gamma$ and dimension $0$ otherwise.

When $n\not\equiv 1\bmod \gamma$ the only non-zero terms are $H^{2}_{S}(\cyclo^{n})$ and $H^{2}(f, \cyclo^{n})$. Exactness of this sequence implies that the map $\res_{f}$ is not the zero map, and is therefore an isomorphism.

When $n\equiv 1\bmod \gamma$ the groups $H^2_S(\chi^n)$, $H^2(f,\chi^n)$, $H^2(\infty,\chi^n)$ and $H^0_S(\chi^{1-n})^\vee$ are all $1$-dimensional. 
The second map
\[H^2(f,\chi^n)\oplus H^2(\infty,\chi^n)\to H^0_S(\chi^{1-n})^\vee\]
is by definition (e.g.~\cite[page 495]{cohomology}) the linear dual of the restriction map
\[H^0_S(\chi^{1-n})\xrightarrow{\oplus\res_v} H^0(f,\chi^{1-n}) \oplus H^0(\infty,\chi^{1-n})\simeq (H^2(f,\chi^n) \oplus H^2(\infty,\chi^n))^\vee,\]
with the isomorphism following by Tate duality~\cite[(7.2.6)]{cohomology} and \cref{bigrep_dual}.
Since the image of this restriction map has non-zero projection in both local factors, the kernel of the linear dual map necessarily also has non-zero projection on both local factors.
From this, together with exactness of \cref{eq:poitou tate part}, we conclude that the map $\res_{f}$ necessarily has non-trivial image in $H^{2}(f, \cyclo^{n})$, and so again must be an isomorphism.
\end{proof}

\begin{proposition}\label{lifting_is_local}
Let $k\leq \ell-2$. A class $a \in H^{1}_{S}(\bigrep{k}{n})$ lifts to a class in $H^{1}_{S}(\bigrep{k+1}{n})$ if and only if the local cup product
\[
\res_f(b^{[k]}) \cup \res_{f}(a) \in H^{2}(f, \cyclo^{1+k+n})
\]
vanishes.
\end{proposition}
\begin{proof}
From \cref{eq:long_exact_seq} we see that $a$ lifts to a class $a'\in H^{1}_S(\bigrep{k+1}{n})$ exactly when $\delta(a) = 0 \in H^{2}_S(\cyclo^{1+k+n})$. 
\Cref{h2_injection} shows that $\res_{f}: H^{2}_{S}(\cyclo^{1+k+n}) \to H^{2}(f, \cyclo^{1+k+n})$ is injective, so the vanishing of $\delta(a)$ can be checked locally at $f$.
Cup products and connecting homomorphism commute with restriction maps, so by \cref{connecting_is_cup}, we have $\delta(a) = 0$ exactly when
\[
0  = \res_{f}(\delta(a)) = \res_{f}(b^{[k]} \cup a) = \res_{f}(b^{[k]}) \cup \res_{f}(a).\qedhere
\]
\end{proof}

Since we've reduced our lifting question to one about the vanishing of local cup products, it is now in our interest to do a detailed analysis of these local cup products. Recall from \cref{lem:H1v basis} that $H^1(f,\bigrep{k}{n})$ contains a nonzero element $\mathbf{ur}$ when $\gamma\mid d(n+k)$, and contains a nonzero element $\mathbf{b}$ when $\gamma\mid d(1-n)$, and whichever of these classes exist form a basis of $H^1(f,\bigrep{k}{n})$.
Note that $b^{[k]}\in H^1_S(\bigrep{k}{})$ maps under $\res_f$ to $\mathbf{b}\in H^1(f,\bigrep{k}{})$.

\begin{lemma}\label{local_cup_products}
Let $0\leq k \leq \ell - 1$.
    Under the cup product pairing 
\[
H^{1}(f, \bigrep{k}{m}) \times H^{1}(f, \bigrep{k}{n}) \xrightarrow{\cup} H^{2}(f, \cyclo^{k+m+n}),
\]
we have
\begin{align*}
\mathbf{b} \cup \mathbf{b} & = 0, &
\mathbf{b} \cup \mathbf{ur} & \neq 0, \\
\mathbf{ur} \cup \mathbf{ur} & = 0, & 
\mathbf{ur} \cup \mathbf{b} & \neq 0
\end{align*}
whenever each exist in their respective cohomology group.
In particular we have that $\mathbf{b} \cup a = 0$ if and only if $a$ is in the span of $\mathbf{b}$.
\end{lemma}
\begin{proof}
We divide into cases depending on which of $\mathbf{ur},\mathbf{b}$ exist in each of the two groups $H^{1}(f, \bigrep{k}{m})$ and $H^{1}(f, \bigrep{k}{n})$.
\begin{itemize}
    \item If the groups are $1$-dimensional and spanned by the same class (for example $\gamma\mid d(n-1)$ and $\gamma\mid d(m-1)$, but $\gamma\nmid d(n+k)$ and $\gamma\nmid d(m+k)$), then the representations $\bigrep{k}{n}$ and $\bigrep{k}{m}$ have the same restriction to $G_{\F_q(x)_f}$. The cup product is alternating on $H^{1}$ \cite[(1.4.4)]{cohomology}, so the cup product of a class with itself equals zero.
    \item If the groups are $1$-dimensional but spanned by different classes (for example $\gamma\mid d(n-1)$ and $\gamma\mid d(m+k)$, but $\gamma\nmid d(n+k)$ and $\gamma\nmid d(m-1)$), then the restrictions of $\bigrep{k}{n}$ and $\bigrep{k}{m}$ to $G_{\F_q(x)_f}$ are cohomological duals of each other. In this case the cup product is the local Tate pairing~\cite[(7.2.6)]{cohomology}, which is non-degenerate; hence the cup product of the respective generators is nonzero.
    \item If both groups are $2$-dimensional, then the representations $\bigrep{k}{n}$ and $\bigrep{k}{m}$ have the same \emph{self-dual} restriction to $G_{\F_q(x)_f}$, so the cup product is an alternating perfect pairing: classes pair with themselves to be zero, and independent elements pair to be nonzero. 
\end{itemize} 
Note that it is impossible for one group to be $1$-dimensional and the other to be $2$-dimensional; for instance, if $\gamma\mid d(n-1)$, $\gamma\mid d(n+k)$, and $\gamma\mid d(m-1)$, then 
\[\gamma\mid d(m-1)+d(n+k)-d(n-1)=d(m+k).\]
Hence the only case remaining is when $H^1(f,\bigrep{k}{n})$ or $H^1(f,\bigrep{k}{m})$ is $0$-dimensional, in which case the claim is vacuously true.
\end{proof}

\subsection{Detecting vanishing of local cup product}\label{subsec:congruence conditions}

Suppose $\gamma\mid d$, and $f(x)$ factors in $\F_\qg[x]$ as $f_1(x)\cdots f_{\gamma}(x)$,  arranged so that $\frob f_i=f_{i+1}$ for all $i$. For $n\geq 1$ define
\[
g_n(x) = \prod_{i=1}^{\gamma} f_{i}(x)^{q^{(i-1)(n-1)}},
\]
and set
\[K_n:=\F_\qg(\sqrt[\ell]{g_n}).\]
Then $K_n$ is a $(\Z/\ell\Z)$-extension of $\F_\qg(x)$, Galois over $\F_q(x)$, and unramified away from $f$ and $\infty$. Note that $g_1(x)=f(x)$ and so $K_1=\F_\qg(C)$. Also,
\[
\frac{g_{n+\gamma}(x)}{g_n(x)} = \prod_{i=1}^{\gamma} f_{i}(x)^{(q^{\gamma(i-1)}-1)q^{(i-1)(n-1)}}
\]
is an $\ell$-th power in $\F_\qg(x)$ because $q^{\gamma(i-1)}\equiv 1\bmod \ell$. Hence $K_n=K_{n+\gamma}$, so $K_n$ is well-defined for $n\in\Z/\gamma\Z$. The fields $K_n$ can also be produced using explicit class field theory for the rational function field $\F_\qg(x)$, as in \cite{hayes} for example.

\begin{lemma}\label{lem:Kn cocycle}
    Suppose $\gamma\mid d$. For all $n$, there exists a nonzero cocycle $a_n\in H^1(\chi^n)$ with kernel field $K_n$.
\end{lemma}
\begin{proof}
    We have 
    \[\frob(g_{n})^{q^{n-1}} = \left(f_1^{q^{(\gamma-1)(n-1)}}\prod_{i=2}^{\gamma} f_{i}^{q^{(i-2)(n-1)}}\right)^{q^{n-1}}=f_1^{q^{\gamma(n-1)}-1}g_n.\]
    Writing $q^{\gamma(n-1)}-1=j\ell$ for some integer $j$, we have $\frob(g_n)^{q^{n-1}}=f_1^{j\ell} g_n$. We therefore have $\frob(\sqrt[\ell]{g_n})^{q^{n-1}}=\zeta^t f_1^j\sqrt[\ell]{g_n}$ for some integer $t$. If we let $\tau\in \Gal(K_n/\F_\qg(x))$ denote the automorphism sending $\sqrt[\ell]{g_n}\mapsto \zeta\sqrt[\ell]{g_n}$ then
    \[\frob(\tau(\sqrt[\ell]{g_n}))^{q^{n-1}}=\frob(\zeta\sqrt[\ell]{g_n})^{q^{n-1}}=\zeta^{q^n+t} f_1^j\sqrt[\ell]{g_n}=\tau^{q^n}(\zeta^tf_1^j\sqrt[\ell]{g_n})=\tau^{q^n}(\frob(\sqrt[\ell]{g_n}))^{q^{n-1}},\]
    and since $\gcd(q,\ell)=1$, we can conclude $\frob\circ\tau=\tau^{q^n}\circ\frob$. 
    
    Analogously to \cref{defrho}, let $\rho_n:\Gal(K_n/\F_q(x))\to \GL_2(\F_\ell)$ be the representation defined by
    \[\rho_n(\frob)=\begin{pmatrix}
        q^n&0\\0&1
    \end{pmatrix},\qquad \rho_n(\tau)=\begin{pmatrix}
        1&1\\0&1
    \end{pmatrix}\]
    for $\tau$ as above; this gives a well-defined representation because $\frob\circ\tau=\tau^{q^n}\circ\frob$. This extends to a representation of $G_{\F_q(x)}$ that factors through $\Gal(K_n/\F_q(x))$. Then $\rho_n$ is an extension of $1$ by $\chi^n$ corresponding to a cocycle $a_n\in H^1(\chi^n)$ with kernel field $K_n$. Since $K_n/\F_q(x)$ is unramified away from $f$ and $\infty$ we in fact have $a\in H^1_S(\chi^n)$.
\end{proof}

We finally reach the proof of \cref{thm:congruence_conditions}. Recall that for $n=2,\ldots,\gamma-1$ we have
\[h_n(x):=\prod_{i=1}^{\gamma} f_{i}(x)^{q^{(i-1)(\gamma-n)}-1}=\frac{g_{\gamma+1-n}(x)}{f(x)}.\]

\congruenceconditions*
\begin{proof}

Note that $F_0^2$ does not have a $\Frob$ eigenvector by \cref{basic_lifting_properties}(a) and (b).
Since we assume $\gamma\mid d$, $F_1^1$ has a $\Frob$ eigenvector by \cref{basic_lifting_properties}(a); if it were a rooftop then $F^1_0$ would contain a $\Frob$ eigenvector by \cref{basic_lifting_properties}(c), but this contradicts \cref{basic_lifting_properties}(a). Hence $F_1^2$ has a $\Frob$ eigenvector. Correspondingly, we have $0\notin \mathcal{T}$ and $1\in \mathcal{T}$.

Now suppose $n\not\equiv 0,1\bmod\gamma$. By \cref{h1S_character_dimension}, $H^1_S(\chi^{1-n})$ is one-dimensional, and by \cref{lem:Kn cocycle} it is spanned by a cocycle $a_{1-n}$ with kernel field $K_{1-n}$. By \cref{prop:eigen to cohom}, 
$F_n^2$ has an eigenvector if and only if $H^1_S(\rho\otimes \chi^{1-n})$ has a nontrivial element mapping to $a_{1-n}$.
By \cref{lifting_is_local}, such an element exists if and only if $\res_f(a_{1-n})\cup \mathbf{b}=0$, and by \cref{local_cup_products} it suffices to check whether $\res_f(a_{1-n})$ is in the span of $\mathbf{b}$. By \cref{lem:H1v basis}, we know that $H^1(f,\rho\otimes\chi^{1-n})$ is spanned by $\mathbf{ur}$, which defines an extension with nontrivial residue field degree, and by $\mathbf{b}$, which has kernel field $\F_\qg(C)_f$. Hence we can detect whether $\res_f(a_{1-n})$ is in the span of $\mathbf{b}$ by checking whether $K_{1-n}$ and $\F_\qg(C)$ have the same completion at a prime over $f$. Taking the place $f_1$ above $f$ in $\F_\qg(x)$, this is equivalent by Kummer theory to checking that $h_n(x)=g_{\gamma+1-n}(x)/f(x)$ is an $\ell$-th power in the residue field $\F_\qg[x]/(f_1(x))$. 
\end{proof}

\begin{remark}
    \Cref{basic_lifting_properties} and \cref{thm:rooftop_pairs} may also be proven using an argument similar to that of \cref{thm:congruence_conditions}, that is, by combining \cref{prop:eigen to cohom} with the cohomological lifting conditions developed in \cref{subsec: local cup product}. While the current proof of \cref{thm:rooftop_pairs} invokes the Weil pairing, the cohomological proof instead uses the Poitou-Tate exact sequence~\cite[(8.6.10)]{cohomology} together with the observation that $\bigrep{k-1}{1-n}$ and $\bigrep{k-1}{n+1-k}$ are cohomological duals. 
\end{remark}

\section{Consequences of lifting conditions}\label{sec:consequences}

Using the relations set up in the previous sections, we are now reduced to an essentially combinatorial problem: under the constraints described in \cref{sec:lifting_behavior_functions}, what are the possibilities for the set of pairs $(n,k)$ such that $F_n^k$ has an eigenvector? In particular, following \cref{thm:rlC_count}, we are interested in the $\ell$-rank of the divisor class group $r_\ell(C)$, which equals the number of $k$ for which $F_{k-1}^k$ contains a $\Frob$ eigenvector. In this section we prove all of the constraints on $r_\ell(C)$ mentioned in \cref{sec:rlC_constraints}. 

For $n\in\Z/\gamma\Z$, let $0\leq k_n\leq \ell-1$ denote the smallest integer such that for all $k_n<k'\leq \ell-1$, there is no $\Frob$ eigenvector in $F_n^{k'}$. In other words, if $k_n\geq 1$ then $F_n^{k_n}$ is a rooftop, and if $k_n=0$ then there is no value of $k$ for which $F_n^k$ has a $\Frob$ eigenvector. We will say that $k_n$ is the \textbf{``rooftop height''} over $n$. 
If we take an integer representative $n\in\{0,1,\ldots,\gamma-1\}$, then the number of $1\leq k\leq k_n$ with $k-1\equiv n\bmod \gamma$ is equal to
\[c(n):=\left\lfloor \frac{k_n-1-n}{\gamma}\right\rfloor+1.\]
Following the visual interpretation as described in \cref{rmk: visual guide}, $c(n)$ counts the number of dark grey circles in column $n$.
Since $k_0=0$ by \cref{basic_lifting_properties}(a), we have $$r_\ell(C)=c(1)+c(2)+\cdots+c(\gamma-1)$$ by \cref{thm:rlC_count}, so $c(n)$ counts the number of contributions to $r_\ell(C)$ from $n$. For all $0\leq k_n\leq\ell-1$ and $0\leq n\leq \gamma-1$, we have $0\leq c(n)\leq \frac{\ell-1}{\gamma}$; specifically, if $k_n=0$ then $c(n)=0$, and if $k_n=\ell-1$ then $c(n)=\frac{\ell-1}{\gamma}$.

The proofs of \cref{intro_upper_bound} and \cref{intro_refined_upper_bound} have many similar features: we will first prove the upper bounds in both theorems, then the lower bounds, then the parity constraints.

\subsection{Upper bounds} We first prove the upper bound in \cref{intro_upper_bound}, namely 
\[r_\ell(C)\leq B:=(\gcd(d,\gamma)-1)\frac{\ell-1}{\gamma}.\]
We have $k_n\geq 1$ if and only if $\gamma\mid dn$ and $\gamma\nmid n$ by \cref{basic_lifting_properties}(a). There are $\gcd(d,\gamma)-1$ such values of $n\in\Z/\gamma\Z$, and for each of these we have $c(n)\leq\frac{\ell-1}{\gamma}$. For all other values of $n$ we have $k_n=0$ and therefore $c(n)=0$. Combining these bounds gives the desired result. ($B$ counts the set of pairs $(n,k)$ for which $F_n^k$ is non-empty and $n\equiv k-1\bmod\gamma$; in the visual interpretation of \cref{rmk: visual guide}, this corresponds to circles in cells that are either light or dark gray.)

We can similarly prove the upper bound in \cref{intro_refined_upper_bound}: if $3\leq \gamma\mid d$ then 
\[r_\ell(C)\leq B':=|\mathcal{T}|\frac{\ell-1}{\gamma}.\]
We have $k_n\geq 2$ if and only if $n\in\mathcal{T}$ by \cref{thm:congruence_conditions}, and for each of these $n$ we have $c(n)\leq \frac{\ell-1}{\gamma}$ as above. We have $c(0)=0$ as before, and for all other $n\notin\mathcal{T}$ we have $n\geq 1$ and $k_n\leq 1$ so that $c(n)=0$. Combining these bounds gives the desired result.

\subsection{Lower bounds}

We begin by proving a general lemma that will be useful in producing lower bounds. For fixed $n,k$, the \emph{diagonal} containing $F_n^k$ is the set of all $F_{n+i}^{k+i}$ for $i\in\Z$ with $1\leq k+i\leq \ell-1$. Equivalently, each diagonal is determined by a constant value of $k-n\in\Z/\gamma\Z$. We have strong constraints on how rooftops can be arranged across diagonals: \cref{basic_lifting_properties}(d) states that any given diagonal can contain at most one non-maximal rooftop, and \cref{basic_lifting_properties}(c) limits which diagonals are allowed to contain rooftops at all. So if we can find many sets $F_n^k$ that contain $\Frob$ eigenvectors among a small collection of diagonals right below the main diagonal $F_{k-1}^k$ (the circles in \cref{fig:lifting_charts}), the pigeonhole principle will ensure that only a few of them can be rooftops; the rest must all lift past this main diagonal and hence contribute to $r_\ell(C)$. 

\begin{lemma}\label{lem:diagonal_pigeons}
    Let $m\in \{0,\ldots,\gamma-1\}$, and suppose there are $r$ values of $n\in \{m,\ldots,\gamma-1\}$ such that $F_n^{n-m+1}$ has a $\Frob$ eigenvector. Then $r_\ell(C)\geq r$.
\end{lemma}
\begin{proof}
    If $m=0$ this is immediate from \cref{thm:rlC_count} (we have no $\Frob$ eigenvectors in $F_0^1$ by \cref{basic_lifting_properties}(a)), so from now on assume $m\geq 1$. 
    Define the sets
    \begin{align*}
        R&:= \{n\in\{m,\ldots,\gamma-1\}: F_n^{n-m+1} \text{ contains a } \Frob \text{ eigenvector}\},\\
        S&:=\{n\in\{1,\ldots,m-1\}:\gamma\mid dn\}.
    \end{align*}
    Finally, let $L\subseteq R\cup S$ be the set of all $n\in R\cup S$ for which $n-k_n>-1$, where $k_n$ is the rooftop height over $n$; equivalently, the set $L$ consists of all $n\in R\cup S$ with $c(n)=0$. We will prove that $|L|\leq |S|$. 

    To this end, let $n\in L$. By \cref{basic_lifting_properties}(c) we cannot have $n-k_n=0$, so in fact $n-k_n\geq 1$. We also have $n-k_n\leq m-1$: for $n\in R$ this follows from $k_n\geq n-m+1$, and for $n\in S$ this follows from $k_n\geq 0$ and $n\leq m-1$. Again by \cref{basic_lifting_properties}(c) we have $\gamma\mid d(n-k_n)$. Taken together, we can conclude that $n-k_n\in S$. In other words, all $n\in L$ lie on one of a set of $|S|$ diagonals.

    Now for each $n\in L$, we have $k_n\geq 1$: this follows by \cref{basic_lifting_properties}(a) if $n\in S$, and by $k_n\geq n-m+1$ if $n\in R$. So consider the rooftop $F_n^{k_n}$. Since $n-k_n>-1$ and $n\leq\gamma-1$ we have $k_n<\gamma\leq \ell-1$, so $F_n^{k_n}$ is a non-maximal rooftop. By \cref{basic_lifting_properties}(d), it is impossible to have $n-k_n=n'-k_{n'}$ for any distinct $n,n'\in L$. So by the pigeonhole principle, there are at most $|S|$ elements in $L$. 
    
    In conclusion, there are at least $r$ elements in $R\cup S$ satisfying $k_n\geq n+1$, so $c(n)\geq 1$. We can conclude that $r_\ell(C)\geq r$.
\end{proof}

\noindent \textbf{Lower bound of \cref{intro_upper_bound}:} We will show that $B\geq 1$ implies $r_\ell(C)\geq 1$. If $B\geq 1$ then $\gcd(d,\gamma)>1$, so there exists $n\in\Z/\gamma\Z$ with $\gamma\mid dn$ and $\gamma\nmid n$. For this value of $n$, $F_n^1$ has an eigenvector by \cref{basic_lifting_properties}(a). Taking $m=n$ in \cref{lem:diagonal_pigeons} proves $r_\ell(C)\geq 1$.\footnote{Alternatively, simply use the fact that any generalized eigenspace must contain at least one true eigenvector.}

\medskip

\noindent \textbf{First lower bound of \cref{intro_refined_upper_bound}:} 
    Suppose $3\leq \gamma\mid d$. If $B'=1$, then since $B\geq B'$ we have $r_\ell(C)\geq B'$ by the lower bound of \cref{intro_upper_bound}. So it is sufficient to assume 
    \[B':=|\mathcal{T}|\frac{\ell-1}{\gamma}\geq 2\]
    and prove $r_\ell(C)\geq 2$. Note that the assumption $\gamma\mid d$ implies $F_n^1$ has an eigenvector for all $1\leq n\leq\gamma-1$ by \cref{basic_lifting_properties}(a). We always have $1\in\mathcal{T}$ by definition of $\mathcal{T}$, which implies by \cref{thm:congruence_conditions} that $F_1^2$ has an eigenvector.
    
    First consider the case $|\mathcal{T}|=1$, so that $F_n^1$ is a rooftop for all $2\leq n\leq \gamma-1$. Since $2=2\cdot 1$, we can apply \cref{thm:evenselfdual} to conclude that $F_1^2$ is not a rooftop and thus $F_1^3$ has a $\Frob$ eigenvector. Now since $B'\geq 2$ but $|\mathcal{T}|=1$ we can conclude $\ell-1\geq 2\gamma$. So for each $2\leq i\leq \gamma-1$, if $F_1^{i+1}$ were a rooftop, then it would be a non-maximal rooftop; since $F_{1-i}^1$ is also a non-maximal rooftop, this contradicts \cref{basic_lifting_properties}(d). Therefore $F_1^{\gamma+1}$ has an eigenvector, but is not a rooftop by \cref{basic_lifting_properties}(c). Hence $k_1\geq\gamma+2$, which implies $c(1)\geq 2$ and hence $r_\ell(C)\geq 2$.

    Now consider the case $|\mathcal{T}|\geq 2$, so $F_n^2$ has an eigenvector for some $n\in \{2,\ldots,\gamma-1\}$. Since $F_n^2$ and $F_{n-1}^1$ both have eigenvectors, we have $r_\ell(C)\geq 2$ by \cref{lem:diagonal_pigeons}, taking $m=n-1$. Together with the previous case, we see that $r_\ell(C)\geq 2$ whenever $B'\geq 2$.

\medskip

\noindent \textbf{Second lower bound of \cref{intro_refined_upper_bound}:} 
    Assuming $\gamma$ is even and $1+\frac{\gamma}{2}\in\mathcal{T}$, we will show that $r_\ell(C) \ge 3$.
    Since $1+\frac{\gamma}{2}\in\mathcal{T}$, we can conclude that $F_{1+\gamma/2}^{2}$ has a $\Frob$ eigenvector by \cref{thm:congruence_conditions}. Since $2 \equiv 2(1+\frac{\gamma}{2}) \bmod \gamma$, we are in the setting of \cref{thm:evenselfdual}, and can conclude that $F_{1+\gamma/2}^{2}$ is not a rooftop: thus
    $F_{1+\gamma/2}^{3}$ also has a $\Frob$ eigenvector. Further, since $F_{1+\gamma/2}^{1}$ is not a rooftop, neither is $F_{\gamma/2}^{1}$ by \cref{thm:rooftop_pairs}. Hence $F_{1+\gamma/2}^{3}$, $F_{\gamma/2}^{2}$, and $F_{\gamma/2-1}^{1}$ all contain eigenvectors, so $r_\ell(C)\geq 3$ by \cref{lem:diagonal_pigeons}.

\subsection{Parity}\label{subsec:parity}

    We first prove the parity constraint of \cref{intro_upper_bound}. 
    Let 
    \[S:=\{n\in\{1,\ldots,\gamma-1\}:\gamma\mid dn, \ k_n\neq \ell-1\}.\] 
    By \cref{basic_lifting_properties}(a), it is equivalent to say that $S$ is the set of $n\in\{0,\ldots,\gamma-1\}$ with $1\leq k_n\leq \ell-2$. For each $n\in S$, let $n^\vee$ denote the unique value in $\{0,\ldots,\gamma-1\}$ satisfying $n^\vee\equiv k_n-n\bmod\gamma$. 
    By \cref{thm:rooftop_pairs} we have $k_{n^\vee}=k_n$ for all $n\in S$, so $n\mapsto n^\vee$ is an involution on $S$ and $c(n)=c(n^\vee)$ for all $n\in S$.
    We can write
    \begin{align*}
        r_\ell(C)&=\sum_{\substack{n\in\{1,\ldots,\gamma-1\},\\ \gamma\mid dn}} \left(\frac{\ell-1}{\gamma}-\left(\frac{\ell-1}{\gamma}-c(n)\right)\right)\\
        &=B-\sum_{n\in S} \left(\frac{\ell-1}{\gamma}-c(n)\right),
    \end{align*}
    since in the first equality we only remove terms with $c(n)=0$, and in the second equality we only remove terms with $c(n)=\frac{\ell-1}{\gamma}$. If $n\neq n^\vee$, then $n$ and $n^\vee$ contribute equal quantities to the sum. So to prove $r_\ell(C)\equiv B\bmod 2$, it suffices to show that the terms indexed by fixed points of the involution are even.

    Suppose $n=n^\vee$, so that $k_n\equiv 2n\bmod \gamma$. By definition of $c(n)$ we have
    \[(c(n)-1)\gamma\leq k_n-1-n<c(n)\gamma.\]
    Since we additionally have $1\leq n<\gamma$ we can conclude
    \[(c(n)-2)\gamma+2n+1< k_n<c(n)\gamma+2n,\]
    so that in fact $k_n=(c(n)-1)\gamma+2n$. We can therefore write
    \[\frac{\ell-1}{\gamma}-c(n)=\frac{\ell-1-k_n+2n}{\gamma}+1.\]
    By \cref{thm:evenselfdual}, $k_n$ is odd, and therefore $\ell-1-k_n+2n$ is odd. Thus $\frac{\ell-1}{\gamma}-c(n)$ is even as desired.
    
    We now assume $3\leq \gamma\mid d$ and prove the parity constraint of \cref{intro_refined_upper_bound}, namely that $r_\ell(C) \equiv B' \bmod 2$. Following \cref{intro_upper_bound}, it suffices to prove that $B'\equiv B\bmod 2$. If $\gamma$ is odd, then $B$ and $B'$ are both even because $\frac{\ell-1}{\gamma}$ is even. So suppose instead that $\gamma$ is even. By \cref{thm:rooftop_pairs}, $F_n^1$ is a rooftop if and only if $F_{1-n}^1$ is a rooftop; since $\gamma$ is even, we always have $1-n\not\equiv n\bmod \gamma$, and so the rooftops $F_n^k$ with $k=1$ always come in pairs. Now $|\mathcal{T}|$ counts the number of $n$ for which $F_n^2$ contains an eigenvector, which equals $\gamma-1$ minus the number of rooftops with $k=1$. Hence $|\mathcal{T}|\equiv \gamma-1\bmod 2$. Since we are assuming $\gamma\mid d$ we have $|\mathcal{T}|\equiv \gcd(d,\gamma)-1\bmod 2$, which implies $B'\equiv B\bmod 2$.

\subsection{Proof of \cref{prop:2lifting}}
    Finally, we assume $\gcd(d,\gamma)=2$ and prove that $r_\ell(C)=\frac{\ell-1}{\gamma}$. By \cref{basic_lifting_properties}(a), $F_n^1$ has an eigenvector only for $n=\frac{\gamma}{2}$. We will prove that the rooftop height over $n$ is $k_n=\ell-1$, which will imply $r_\ell(C)=c(n)=\frac{\ell-1}{\gamma}$ as desired.
    
    Suppose $1\leq k\leq \ell-2$ is such that $F_n^k$ is a rooftop. By \cref{basic_lifting_properties}(c), this implies $k\equiv 0\bmod{\frac{\gamma}{2}}$ and $k\not\equiv \frac{\gamma}{2}\bmod{\gamma}$; equivalently, $k$ is a multiple of $\gamma$. But then $k\equiv 2n\bmod\gamma$ and $k$ is even, so $F_n^k$ is not a rooftop by \cref{thm:evenselfdual}. Hence $F_n^k$ is not a rooftop for any $1\leq k\leq \ell-2$.

\appendix
\section{Proof of \cref{lem:linearcommute}}\label{appendix_linearcommute}

We prove that $\Frob\circ \eta=q\eta\circ\Frob$.

\begin{lemma}\label{comblemma}
    For all integers $q$ and $0\leq k\leq \ell-1$,
    \[\sum_{i=1}^{\ell-1}\sum_{j=0}^i \frac{(-1)^{j+1}}{i}\binom{i}{j}\binom{qj}{k}=\left\lbrace\begin{array}{ll}
        \displaystyle\frac{(-1)^{k+1}q}{k} & k>0 \\
        0 & k=0.
    \end{array}\right.\]
\end{lemma}

\begin{proof}
    For $k=0$, this follows from the fact that the alternating sum of $\binom{i}{j}$ for $0\leq j\leq i$ is equal to $0$. So from now on we assume $k\geq 1$.
    We will first prove the identity
    \begin{align}\label{eq:poly_identity}
        \sum_{i=0}^{\ell-1}\sum_{j=0}^i (-1)^{i-j}\binom{i}{j}\binom{qj}{k}\binom{x}{i}=\binom{qx}{k}
    \end{align}
    as an equality in $\Q[x]$. First consider the value of 
    \[\sum_{j=0}^i (-1)^{i-j}\binom{i}{j}\binom{qj}{k}\]
    for $i>k$. This equals the coefficient of $t^k$ in $((t+1)^q-1)^i$. But $(t+1)^q-1$ is a multiple of $t$ and so $((t+1)^q-1)^i$ is a multiple of $t^i$, and hence the coefficient of $t^k$ is zero. Hence only terms with $i\leq k$ contribute to the left-hand side of \cref{eq:poly_identity}, so both sides of the desired identity are polynomials of degree at most $k$.
    
    If $x\in\{0,1,\ldots,\ell-1\}$, we can compute the coefficient of $t^k$ in $(((t+1)^q-1)+1)^x=(t+1)^{qx}$, obtaining
    \[\sum_{i=0}^{x}\binom{x}{i}\sum_{j=0}^i (-1)^{i-j}\binom{i}{j}\binom{qj}{k}=\binom{qx}{k}.\]
    Since $x\leq \ell-1$, and $\binom{x}{i}=0$ for $i>x$, we can sum $i$ from $0$ to $\ell-1$ without changing the value, giving us the desired equality for these specified values of $x$. We therefore have two polynomials of degree at most $k$ in $\Q[x]$ with equal values at $\ell>k$ points, and therefore the polynomials are equal.
    
    Now compare the coefficient of $x$ in each side of \cref{eq:poly_identity}. The coefficient of $x$ in $\binom{x}{i}$ is $0$ if $i=0$ and $(-1)^{i+1}/i$ if $i\geq 1$, so we obtain the desired result.
\end{proof}

\Cref{comblemma} is an identity over $\Q$, but the only denominators that occur are coprime to $\ell$; it therefore descends to an identity over $\F_\ell$. It is in this form that we apply it below.

\begin{proof}[Proof of \cref{lem:linearcommute}]
Using the definition
\[\eta=-\sum_{i=1}^{\ell-2}i^{-1}(1-\zeta)^i=\sum_{i=1}^{\ell-2}\frac{(-1)^{i+1}}{i}(\zeta-1)^i,\] 
we have
\begin{align*}
    \Frob\circ \eta&=\sum_{i=1}^{\ell-1} \frac{(-1)^{i+1}}{i}\Frob\circ (\Czeta-1)^i\\
    &=\sum_{i=1}^{\ell-1}  \frac{(-1)^{i+1}}{i}(\Czeta^q-1)^i\circ\Frob\\
    &=\sum_{i=1}^{\ell-1} \frac{(-1)^{i+1}}{i}\sum_{j=0}^i \binom{i}{j}(-1)^{i-j}(\Czeta-1+1)^{qj}\circ\Frob\\
    &=\sum_{i=1}^{\ell-1} \sum_{j=0}^i \frac{(-1)^{j+1}}{i}\binom{i}{j}\sum_{k=0}^{qj}\binom{qj}{k}(\Czeta-1)^{k}\circ\Frob.
\end{align*}
Now we have $(\Czeta-1)^{k}=0$ for all $k>\ell-1$ and $\binom{qj}{k}=0$ for all $k>qj$, so the sum over $k$ can be indexed from $0$ to $\ell-1$ without changing the value (in both cases, only terms with $k\leq\min\{qj,\ell-1\}$ contribute). Using \cref{comblemma}, we can conclude:
\begin{align*}
    \Frob\circ \eta&=\sum_{k=0}^{\ell-1}\sum_{i=1}^{\ell-1} \sum_{j=0}^i \frac{(-1)^{j+1}}{i}\binom{i}{j}\binom{qj}{k}(\Czeta-1)^{k}\circ\Frob\\
    &=\sum_{k=0}^{\ell-1}\left(\frac{(-1)^{k+1}q}{k}\right)(\Czeta-1)^{k}\circ\Frob\\
    &=q\eta\circ\Frob.\qedhere
\end{align*}
\end{proof}

\printbibliography

@book{arul2020,
    AUTHOR = {Arul, Vishal},
     TITLE = {Explicit {D}ivision and {T}orsion {P}oints on {S}uperelliptic
              {C}urves and {J}acobians},
      NOTE = {Thesis (Ph.D.)--Massachusetts Institute of Technology},
 PUBLISHER = {ProQuest LLC, Ann Arbor, MI},
      YEAR = {2020},
     PAGES = {(no paging)},
   MRCLASS = {Thesis},
  MRNUMBER = {4158225},
       URL =
              {http://gateway.proquest.com/openurl?url_ver=Z39.88-2004&rft_val_fmt=info:ofi/fmt:kev:mtx:dissertation&res_dat=xri:pqm&rft_dat=xri:pqdiss:28216982}
}

@InProceedings{cohenlenstra,
author="Cohen, H.
and Lenstra, H. W.",
editor="Jager, Hendrik",
title="Heuristics on class groups of number fields",
booktitle="Number Theory Noordwijkerhout 1983",
year="1984",
publisher="Springer Berlin Heidelberg",
address="Berlin, Heidelberg",
pages="33--62",
isbn="978-3-540-38906-4"
}

@article{Wittmann,
title = {l-Class groups of cyclic function fields of degree l},
journal = {Finite Fields and Their Applications},
volume = {13},
number = {2},
pages = {327-347},
year = {2007},
issn = {1071-5797},
doi = {https://doi.org/10.1016/j.ffa.2005.09.001},
url = {https://www.sciencedirect.com/science/article/pii/S1071579705000754},
author = {Christian Wittmann}
}

@article {hayes,
    AUTHOR = {Hayes, D. R.},
     TITLE = {Explicit class field theory for rational function fields},
   JOURNAL = {Trans. Amer. Math. Soc.},
  FJOURNAL = {Transactions of the American Mathematical Society},
    VOLUME = {189},
      YEAR = {1974},
     PAGES = {77--91},
      ISSN = {0002-9947,1088-6850},
   MRCLASS = {12A65 (12A90)},
  MRNUMBER = {330106},
MRREVIEWER = {Bostwick\ F.\ Wyman},
       DOI = {10.2307/1996848},
       URL = {https://doi.org/10.2307/1996848},
}

@article {calegari_emerton,
    AUTHOR = {Calegari, Frank and Emerton, Matthew},
     TITLE = {On the ramification of {H}ecke algebras at {E}isenstein
              primes},
   JOURNAL = {Invent. Math.},
  FJOURNAL = {Inventiones Mathematicae},
    VOLUME = {160},
      YEAR = {2005},
    NUMBER = {1},
     PAGES = {97--144},
      ISSN = {0020-9910},
   MRCLASS = {11F80 (11F33)},
  MRNUMBER = {2129709},
       DOI = {10.1007/s00222-004-0406-z},
       URL = {https://doi.org/10.1007/s00222-004-0406-z},
}

@book{cohomology,
  title={Cohomology of Number Fields},
  author={Neukirch, J. and Schmidt, A. and Wingberg, K.},
  isbn={9783540378891},
  series={Grundlehren der mathematischen Wissenschaften},
  url={https://books.google.ca/books?id=ZX8CAQAAQBAJ},
  year={2013},
  publisher={Springer Berlin Heidelberg}
}

@Inbook{Milne1986,
    author="Milne, J. S.",
    editor="Cornell, Gary
    and Silverman, Joseph H.",
    title="Abelian Varieties",
    bookTitle="Arithmetic Geometry",
    year="1986",
    publisher="Springer New York",
    address="New York, NY",
    pages="103--150",
    isbn="978-1-4613-8655-1",
    doi="10.1007/978-1-4613-8655-1_5",
    url="https://doi.org/10.1007/978-1-4613-8655-1_5"
}

@article{schaeferstubley,
    author = {Karl Schaefer and Eric Stubley},
    title = {Class groups of Kummer extensions via cup products in Galois cohomology},
    year = {2019},
    journal = {Trans. Amer. Math. Soc.},
    volume = {372},
    pages = {6927--6980},
    doi = {https://doi.org/10.1090/tran/7746}
}

@article {wake_wangerickson,
    AUTHOR = {Wake, Preston and Wang-Erickson, Carl},
     TITLE = {The rank of {M}azur's {E}isenstein ideal},
   JOURNAL = {Duke Math. J.},
  FJOURNAL = {Duke Mathematical Journal},
    VOLUME = {169},
      YEAR = {2020},
    NUMBER = {1},
     PAGES = {31--115},
      ISSN = {0012-7094},
   MRCLASS = {11F80 (11F33)},
  MRNUMBER = {4047548},
MRREVIEWER = {Atsushi Yamagami},
       DOI = {10.1215/00127094-2019-0039},
       URL = {https://doi.org/10.1215/00127094-2019-0039},
}

@article {wawrow,
    AUTHOR = {Wawr\'{o}w, Wojciech},
     TITLE = {On torsion of superelliptic {J}acobians},
   JOURNAL = {J. Th\'{e}or. Nombres Bordeaux},
  FJOURNAL = {Journal de Th\'{e}orie des Nombres de Bordeaux},
    VOLUME = {33},
      YEAR = {2021},
    NUMBER = {1},
     PAGES = {223--235},
      ISSN = {1246-7405},
   MRCLASS = {14H40 (11G30 14G10 14H45)},
  MRNUMBER = {4312706},
MRREVIEWER = {Zijian Zhou},
       URL = {http://jtnb.cedram.org/item?id=JTNB_2021__33_1_223_0},
}

@article {jedrzejak,
    AUTHOR = {J\k{e}drzejak, Tomasz},
     TITLE = {On the torsion of the {J}acobians of superelliptic curves
              {$y^q=x^p+a$}},
   JOURNAL = {J. Number Theory},
  FJOURNAL = {Journal of Number Theory},
    VOLUME = {145},
      YEAR = {2014},
     PAGES = {402--425},
      ISSN = {0022-314X},
   MRCLASS = {11G05 (14G25)},
  MRNUMBER = {3253312},
MRREVIEWER = {Ariyan Javanpeykar},
       DOI = {10.1016/j.jnt.2014.06.013},
       URL = {https://doi.org/10.1016/j.jnt.2014.06.013},
}

@article {poonen_schaefer,
    AUTHOR = {Poonen, Bjorn and Schaefer, Edward F.},
     TITLE = {Explicit descent for {J}acobians of cyclic covers of the
              projective line},
   JOURNAL = {J. Reine Angew. Math.},
  FJOURNAL = {Journal f\"{u}r die Reine und Angewandte Mathematik. [Crelle's
              Journal]},
    VOLUME = {488},
      YEAR = {1997},
     PAGES = {141--188},
      ISSN = {0075-4102},
   MRCLASS = {11G30 (11G10 14H30 14H40)},
  MRNUMBER = {1465369},
MRREVIEWER = {David Grant},
}

@article {cornelissen,
    AUTHOR = {Cornelissen, Gunther},
     TITLE = {Two-torsion in the {J}acobian of hyperelliptic curves over
              finite fields},
   JOURNAL = {Arch. Math. (Basel)},
  FJOURNAL = {Archiv der Mathematik},
    VOLUME = {77},
      YEAR = {2001},
    NUMBER = {3},
     PAGES = {241--246},
      ISSN = {0003-889X},
   MRCLASS = {11G20 (11R29 11R58)},
  MRNUMBER = {1865865},
MRREVIEWER = {Andreas Schweizer},
       DOI = {10.1007/PL00000487},
       URL = {https://doi.org/10.1007/PL00000487},
}

@article {ELS,
    AUTHOR = {Ellenberg, Jordan S. and Li, Wanlin and Shusterman, Mark},
     TITLE = {Nonvanishing of hyperelliptic zeta functions over finite
              fields},
   JOURNAL = {Algebra Number Theory},
  FJOURNAL = {Algebra \& Number Theory},
    VOLUME = {14},
      YEAR = {2020},
    NUMBER = {7},
     PAGES = {1895--1909},
      ISSN = {1937-0652},
   MRCLASS = {11M38},
  MRNUMBER = {4150253},
MRREVIEWER = {Adam Morgan},
       DOI = {10.2140/ant.2020.14.1895},
       URL = {https://doi-org.libproxy.wustl.edu/10.2140/ant.2020.14.1895},
}

@book {Poonen,
    AUTHOR = {Poonen, Bjorn},
     TITLE = {Rational points on varieties},
    SERIES = {Graduate Studies in Mathematics},
    VOLUME = {186},
 PUBLISHER = {American Mathematical Society, Providence, RI},
      YEAR = {2017},
     PAGES = {xv+337},
      ISBN = {978-1-4704-3773-2},
   MRCLASS = {14G05 (11G35)},
  MRNUMBER = {3729254},
MRREVIEWER = {Daniel\ Loughran},
       DOI = {10.1090/gsm/186},
       URL = {https://doi.org/10.1090/gsm/186},
}

\end{document}